\documentclass[reqno]{amsart}
\usepackage{amsmath,amssymb,cite}
\usepackage[mathscr]{euscript}

\usepackage{xcolor}

\definecolor{bblue}{rgb}{.2,0.2,.8}

\theoremstyle{plain}
\newtheorem{theorem}{Theorem}[section]
\newtheorem{proposition}[theorem]{Proposition}
\newtheorem{lemma}[theorem]{Lemma}
\newtheorem{corollary}[theorem]{Corollary}

\theoremstyle{definition}

\theoremstyle{remark}

\newcommand{\res}{\mathop{\hbox{\vrule height 7pt width .5pt depth 0pt\vrule height .5pt width 6pt depth 0pt}}\nolimits}
\numberwithin{equation}{section}
\numberwithin{theorem}{section}

\def\be{\begin{equation}}
\def\ee{\end{equation}}
\def\bp{\begin{pmatrix}}
\def\ep{\end{pmatrix}}
\def\bea{\begin{eqnarray}}
\def\eea{\end{eqnarray}}

\def\\{\par\medskip}

\let\0=\noindent

\newcommand{\mc}[1]{{\mathcal #1}}

\newcommand{\mf}[1]{{\mathfrak #1}}
\newcommand{\bs}[1]{{\boldsymbol #1}}
\newcommand{\bb}[1]{{\mathbb #1}}
\newcommand{\ms}[1]{{\mathscr #1}}

\renewcommand{\epsilon}{\varepsilon}
\renewcommand{\hat}{\widehat}

\newcommand{\<}{\langle}
\renewcommand{\>}{\rangle}

\title[Donsker-Varadhan and hydrodynamical large deviations]{
  Concurrent Donsker-Varadhan and hydrodynamical large deviations
  }

\author{L. Bertini}
\address{Lorenzo Bertini \hfill\break \indent
   Dipartimento di Matematica, Universit\`a di Roma `La Sapienza'
   \hfill\break \indent   P.le Aldo Moro 2, 00185 Roma, Italy}
 \email{bertini@mat.uniroma1.it}

\author{D. Gabrielli}
\address{\noindent Davide Gabrielli \hfill\break\indent 
 DISIM, Universit\`a dell'Aquila
\hfill\break\indent 
67100 Coppito, L'Aquila, Italy
}
\email{davide.gabrielli@univaq.it}

\author{C. Landim}
\address{Claudio Landim
  \hfill\break\indent IMPA \hfill\break\indent Estrada Dona Castorina
  110, \hfill\break\indent
J. Botanico, 22460 Rio de Janeiro, Brazil\hfill\break\indent
  {\normalfont and} \hfill\break\indent CNRS UMR 6085, Universit\'e de
  Rouen, \hfill\break\indent Avenue de l'Universit\'e, BP.12,
  Technop\^ole du Madril\-let, \hfill\break\indent
F76801 Saint-\'Etienne-du-Rouvray, France.} 
\email{landim@impa.br}

\begin{document}

\begin{abstract}
  We consider the weakly asymmetric exclusion process on the
  $d$-dimensional torus. We prove a large deviations principle for the
  time averaged empirical density and current in the joint limit in
  which both the time interval and the number of particles diverge.
  This result is obtained both by analyzing the variational
  convergence, as the number of particles diverges, of the
  Donsker-Varadhan functional for the empirical process and by
  considering the large time behavior of the hydrodynamical rate
  function.  The large deviations asymptotic of the time averaged
  current is then deduced by contraction principle. The structure of
  the minimizers of this variational problem corresponds to the
  possible occurrence of dynamical phase transitions.
\end{abstract}

\noindent
\keywords{Exclusion processes, Hydrodynamical limits,
  Large deviations, Empirical process, Dynamical phase transitions}

\subjclass[2010]
{Primary 
60K35, 
60F10; 
Secondary 
82C22, 
82C70. 
}

\maketitle
\thispagestyle{empty}

\section{Introduction}
\label{s:00}

Stochastic lattice gases, that describe the evolution of interacting
random particles on a lattice of mesh $1/N$, have been an instrumental
tool in the development of non-equilibrium statistical mechanics
\cite{mft,kl,Sp}.  Their macroscopic behavior, usually referred to as
hydrodynamic scaling limit, is described as follows. Given a
microscopic realization of the process, the empirical density
$\bs \pi_N$ is defined by counting locally the average number of
particles while the empirical current $\bs J_N$ is defined by counting
the net flow of particles.  By the local conservation of the number of
particles, $\bs \pi_N$ and $\bs J_N$ satisfy the continuity
equation. The content of the hydrodynamical limit is the law of 
large numbers for the pair $(\bs \pi_N, \bs J_N)$ in the limit
$N\to \infty$. For driven-diffusive systems the limiting evolution is
given by
\begin{equation}
\label{0hl}
\begin{cases}
\partial_t \bs \rho + \nabla \cdot \bs j =0, \\
\bs j = -D(\bs \rho)\nabla \bs \rho + \sigma(\bs \rho) E,
\end{cases}
\end{equation}
where $E=E(x)$ is the applied external field, $D$ is the diffusion
matrix, and $\sigma$ is the mobility. In particular, the density
profile $\bs \rho=\bs \rho(t,x)$ solves the non-linear driven
diffusive equation
\begin{equation}
\label{0hld}
\partial_t \bs \rho + \nabla\cdot \big( \sigma(\bs\rho) E\big) =
\nabla\cdot \big( D(\bs\rho) \nabla \bs\rho \big).
\end{equation}
We refer to \cite{kl} for the details on the derivation of
\eqref{0hld}, while the scaling limit of the empirical current leading
to \eqref{0hl} is discussed in \cite{bdgjl07} in the case of the
symmetric exclusion process.

The large deviations with respect to the hydrodynamic limit in the
time window $[0,T]$ are characterized by the rate function
\begin{equation}
\label{0ld}
A_T(\bs \rho,\bs j) = \int_0^T\!dt \!\int\!dx \,
\frac { \big|\bs j + D(\bs\rho)\nabla \bs\rho -\sigma(\bs \rho ) E
  \big|^2}
{4\, \sigma(\bs \rho)},
\end{equation}
that is at the base of the Macroscopic Fluctuations Theory and it is
widely used in non-equilibrium statistical mechanics \cite{mft}.  We
refer to \cite{kl,kov} for the derivation of this rate function when
the empirical current is disregarded.

A significant problem is the behavior of the average of empirical
current over the time interval $[0,T]$ in the limit when $N\to\infty$
and then $T\to\infty$. By the hydrodynamical large deviations
principle and contraction principle, this amounts to analyze the
behavior as $T\to \infty$ of the minimizers to \eqref{0ld} with the
constraint $\frac 1T \int_0^T\!dt \, \bs j = J$. This problem has been
initially raised in \cite{bd0} while in \cite{bdgjl05} it has been
pointed out that the minimizers can exhibit a non-trivial time
dependent behavior. In \cite{bdgjl06, bd} it has been then shown that
this is actually the case for the weakly asymmetric exclusion process
and the Kipnis-Marchioro-Presutti model where, for suitable value of
the parameters, traveling waves are more convenient than constant
profiles.

Denote by $I^{(2)} (J)$ the limiting value as $T\to\infty$ of the
minimum to $T^{-1}\, A_T$ with the constraint
$\frac 1T \int_0^T\!dt \, \bs j = J$. Varadhan \cite{V0} proposed the
following representation for $I^{(2)}$ 
\begin{equation}
\label{0I2}
I^{(2)} (J) = \inf \bigg\{ \!\! \int \!\! dP 
\!\!  \int\!\!dx\, \frac { \big|\bs j(t) + D(\bs\rho(t))\nabla \bs\rho(t)
  -\sigma(\bs \rho(t) ) E\big|^2} {4\sigma(\bs \rho(t))}; \, \int\!
\!dP 
\, \bs j(t) = J \!\bigg\}
\end{equation}
where the infimum is carried out over the probabilities $P$ invariant by
time translations on the set of paths $(\bs \rho, \bs j)$ satisfying
the continuity equation $\partial_t \bs\rho + \nabla\cdot \bs j=0$.
Note that $I^{(2)}$ is
convex and that, by the stationarity of $P$, the actual value of $t$
on the right hand side of \eqref{0I2} is irrelevant.

The purpose of the present analysis is to prove the validity of the
representation \eqref{0I2} in the context of the weakly asymmetric
exclusion process for which $D=1$ and $\sigma(\rho) =
\rho(1-\rho)$. This will be achieved both when the limit $T\to \infty$
is carried out after the hydrodynamic limit $N\to \infty$ and when the
limits are carried out in the opposite order.  In fact, the
representation \eqref{0I2} will be deduced by the contraction
principle from a large deviation result at the level of the empirical
processes that we next introduce.

Consider first the case in which the limit $N\to \infty$ is taken
after $T\to \infty$. By the Donsker-Varadhan result, see e.g.\
\cite{dv1-4,Var84}, as $T\to \infty$ the empirical process associated to the
weakly asymmetric exclusion process satisfies a large deviation
principle in which the affine rate function is the relative entropy per unit
of time with respect to the stationary process. By projecting this
functional to the stationary probabilities on the empirical density
and current, and analyzing its variational convergence as $N\to \infty$
we deduce the desired large deviation principle with affine rate function
given by
\begin{equation}
\label{0Ib}
\bs I(P) = \int \! P(d\bs \rho,d \bs j) \, \int\!dx \,\, \frac {
  \big|\bs j(t) + D(\bs\rho(t))\nabla \bs\rho(t) -\sigma(\bs \rho(t) )
  E\big|^2} {4\, \sigma(\bs \rho(t))} \;\cdot
\end{equation}
The main ingredient in this derivation is, as for hydrodynamical large
deviations, the validity of local equilibrium with probability
super-exponentially close to one as $N\to \infty$. Observe that the
rate function in \eqref{0I2} is obtained from \eqref{0Ib} by
contraction.  The proof for the case in which the limit $T\to \infty$
is taken after $N\to \infty$ is achieved by lifting the hydrodynamical
rate function \eqref{0ld} to the set of stationary probabilities on
density and current and analyzing its variational convergence as
$T\to \infty$.

As a corollary of the analysis here presented, we also deduce the
``level two'' large deviations relative to the family of random
probability measures $\frac 1T \int_0^T\!dt\, \delta_{\bs\pi_N(t)}$ in
the joint limit $N,T\to \infty$.  Letting
$\imath_t(\bs \rho,\bs j) = \bs\rho(t)$, the corresponding rate
function is
\begin{equation}
\label{0Is}
\ms I (\wp ) = \inf \Big\{ \bs I (P) \;;\; P\circ \imath_t^{-1} \, =
\wp \Big\}
\end{equation}
Since $\ms I(\wp) =0$ if and only if $\wp$ is a stationary measure for
the flow associated to the hydrodynamic equation \eqref{0hld}, this
large deviations statement implies the hydrostatic limit for the
weakly asymmetric exclusion process: in the limit $N\to \infty$ the
empirical density constructed by sampling the particles according to
the stationary measure converges to the unique stationary solution to
\eqref{0hld}.

\section{Notation and Results}
\label{sec01}

\subsection*{Microscopic dynamics}
Denote by $\bb T^d = [0,1)^d$ the $d$-dimensional torus of length
$1$ and let $dx$ be the corresponding Haar measure.
Fix $N\ge 1$, and let $\bb T^d_N$ the discretization of
$\bb T^d$: $\bb T^d_N = \bb T^d \cap(N^{-1}\mathbb{Z})^{d}$. The
elements of $\bb T^d$ and $\bb T^d_N$ are represented by $x$ and
$y$. Let $\color{bblue} \bb B_N$ be the set of ordered,
nearest-neighbor pairs $(x,y)$ in $\bb T^d_N$.

Denote by $\color{bblue} \Sigma_N:=\{0,1\}^{\bb T^d_N}$ the space of
configurations.  Elements of $\Sigma_N$ are represented by $\eta$, so
that $\eta_x=1$, resp.\ $0$, if site $x$ is occupied, resp.\ vacant,
for the configuration $\eta$.  Fix $E$ in
$\color{bblue} C^1(\bb T^d ; \bb R^d)$, the space of continuously
differentiable vector fields defined on $\bb T^d$.  In some statements
we assume that $E$ is \emph{orthogonally decomposable}: there are
$U \in C^2(\bb T^d)$ and $\tilde E\in C^1(\bb T^d ; \bb R^d)$ with
vanishing divergence, $\nabla\cdot \tilde E =0$, satisfying the
pointwise orthogonality $\nabla U(x) \cdot \tilde E(x) =0$,
$x\in \bb T^d$, such that $E=-\nabla U +\tilde E$.

The weakly asymmetric exclusion process (WASEP) with external field
$E$ is the Markov process on $\Sigma_N$ whose generator $L_N$ acts on
functions $f\colon\Sigma_N\to \bb R$ as
\begin{equation}
\label{22}
(L_{N} f)(\eta) \;=\; N^2 \sum_{(x,y)\in \bb B_N}  \eta_x\, [ 1-\eta_y] \,
e^{(1/2)\,  E_N(x,y)}
\big[ f(\sigma^{x,y} \eta)-f(\eta)\big] \;.
\end{equation}
In this formula, the configuration $\sigma^{x,y}\eta$ is obtained from
$\eta$ by exchanging the occupation variables $\eta_x$, $\eta_y$:
\begin{equation*}
(\sigma^{x,y} \eta)_z \;:=\;
\begin{cases}
        \eta_y & \textrm{ if \ } z=x\;, \\
        \eta_x & \textrm{ if \ } z=y\;, \\
        \eta_z & \textrm{ if \ } z\neq x,y\;,
\end{cases}
\end{equation*}
and $E_N(x,y)$ represents the line integral of $E$ along the oriented
segment from $x$ to $y$:
\begin{equation}
  \label{dvf}
E_N(x,y) \;=\; \int_{x}^{y} E\cdot  \, d\ell \;=\;
\int_0^1 E(x + r \, [y-x]) \cdot [y-x] \, dr\;,
\end{equation}
where $a\cdot b$ is the inner product in $\bb R^d$.  Note that
$E_N(\cdot,\cdot)$ is antisymmetric and that $E_N$ is of order $1/N$.

Denote by
$\color{bblue} \Sigma_{N,K} = \{\eta \in \Sigma_N : \sum_{x\in \bb
  T^d_N} \eta_x = K\}$, $K=0,\ldots, N^d$, the set of configurations
with $K$ particles. The Markov process with generator $L_N$ is
irreducible in the finite state space $\Sigma_{N,K}$. It has therefore
a unique stationary probability measure, denoted by
$\color{bblue} \mu_{N,K}$.

Hereafter, $\color{bblue} \mc R$ represents either $\bb R_+$ or
$\bb R$.  Denote by $\color{bblue} D(\mc R, \Sigma_{N})$,
the set of right-continuous functions with left-limits from $\mc R$
to $\Sigma_{N}$, endowed with the Skorohod
topology and the corresponding Borel $\sigma$-algebra.  Elements of
$D(\mc R, \Sigma_{N})$ are represented by $\bs \eta$.

For a probability measure $\nu$ on $\Sigma_{N}$, denote by
$\color{bblue} \bb P^N_{\nu}$ the probability measure on
$D(\bb R_+, \Sigma_{N})$ induced by the Markovian dynamics associated
to the generator $L_N$ starting from $\nu$. When the measure $\nu$ is
concentrated on a configuration $\eta\in \Sigma_{N}$,
$\nu = \delta_\eta$, we write $\color{bblue} \bb P^N_{\eta}$ instead
of $\bb P^N_{\delta_\eta}$.  For $K=0,\ldots,N^d$, the stationary
processes associated to the WASEP dynamics with $K$ particles is
denoted by $\color{bblue} \bb P_{\mu_{N,K}}^N$ that we regard as a
probability measure on $D(\bb R, \Sigma_{N,K})$ invariant with respect
to time-translations. Expectation with respect to
$\bb P_{\mu_{N,K}}^N$ is represented by
$\color{bblue} \bb E_{\mu_{N,K}}^N$.

\subsection*{Empirical density}
Let $\color{bblue} \mc M_+(\bb T^d)$ be the set of positive measures
on $\bb T^d$ with total mass bounded by $1$, endowed with the weak
topology and the corresponding Borel $\sigma$-algebra. Let also
$\color{bblue} \mc M_m(\bb T^d)$, $m\in [0,1]$, be the closed subset
of $\mc M_+(\bb T^d)$ given by the measures whose total mass is
equal to $m$.

The \emph{empirical density} is the map
$\pi_N \colon \Sigma_N \to \mc M_+(\bb T^d)$ defined by
\begin{equation}
\label{11}
\pi_N (\eta) \;:=\; \frac 1{N^d} \sum_{x\in \bb T^d_N} \eta_x\, \delta_x\;,
\end{equation}
where $\delta_x$, $x\in \bb T^d$, is the point mass at $x$. For a
continuous function $f\colon \bb T^d\to\bb R$ and a measure $\nu$ in $\mc
M_+(\bb T^d)$, we represent by $\<\nu \,,\, f\>$ the integral of $f$
with respect to $\nu$ so that
\begin{equation}
\label{34}
\<\pi_N \,,\, f\> \;=\; \frac 1{N^d} \sum_{x\in \bb T^d_N} \eta_x\,
f(x) \;.
\end{equation}
We also call empirical density the map
$\bs \pi_N \colon D (\mc R, \Sigma_N) \to D (\mc R, \mc M_{+} (\bb
T^d))$ defined by
\begin{equation}
  \label{empden}
[\, \bs \pi_N (\bs \eta) \, ]\, (t) \;:=\; 
\pi_N(\bs \eta (t)) \;=\; \frac 1 {N^d} 
\sum_{x\in \bb T^d_N} \bs \eta_x(t) \,
\delta_{x} \;, \quad t\in \mc R\;.
\end{equation}

\subsection*{Empirical current.} For an oriented bond
$(x,y)\in \bb B_N$ and $s<t$, let $\ms N^{x,y}_{(s,t]}(\bs \eta)$ be
the number of jumps from $x$ to $y$ in the time interval $(s,t]$ of
the path $\bs\eta\in D(\mc R,\Sigma_N)$:
\begin{equation}
\label{27}
\ms N^{x,y}_{(s,t]}(\bs \eta) \;=\; \sum_{s< r \le t} \bs\eta_{x}(r-) \,
[1-\bs\eta_y(r-) ]\, \bs 1\{ 
\bs \eta(r) = \sigma^{x,y} \bs \eta(r-)\}\;. 
\end{equation}
Fix a trajectory $\bs \eta \in D (\mc R, \Sigma_N)$, and denote by
$\color{bblue} C(\bb T^d; \bb R^d)$ the space of continuous vector
fields on $\bb T^d$. If $\bs \eta$ is a trajectory compatible with the
WASEP dynamics, i.e.\ such that for each jump time $t$ we have
$\bs \eta(t) =\sigma^{x,y}\bs \eta(t-)$ for some $(x,y)\in \bb B_N$,
we define the \emph{integrated empirical current}
$\bs J_{N}(\bs \eta)$ as follows. Let
$[\, \bs J_{N}(\bs \eta)\,]\, (0)=0$, and, for $t\,>\, 0$, let
$[\, \bs J_{N}(\bs \eta)\,]\, (t)$ be the linear functional on
$C(\bb T^d; \bb R^d)$ defined by
\begin{equation}
\label{ecurr}
\< \bs J_{N}(\bs \eta) \, (t) \,,\, F \>
\;:=\;  \frac 1{N^d} \sum_{(x,y) \in \bb B_N} \ms N^{x,y}_{(0,t]}(\bs \eta)
\int_{x}^{y} F  \cdot  d\ell \;, \quad
F\,\in\, C(\bb T^d; \bb R^d)\;.
\end{equation}

For $t<0$, we replace in the previous formula $\ms N^{x,y}_{(0,t]}(\bs
\eta)$ by $- \, \ms N^{x,y}_{(t,0]}(\bs \eta)$.
If $F=(F_1, \ldots, F_d)$, then for $t\ge 0$
\begin{equation*}
\< \bs J_{N}(\bs \eta) \, (t) \,,\, F \> = \frac 1{N^d} \sum_{j=1}^d
\sum_{x \in \bb T^d_N} \frac 1N\, \Big\{ \ms N^{x,x+\mf
  e_j}_{(0,t]}(\bs \eta) - \ms N^{x+\mf e_j,x}_{(0,t]}(\bs \eta)
\Big\} \int_{0}^{1} F_j (x+r\mf e_j) \, dr \;,
\end{equation*}
where, $\color{bblue} \mf e_j = e_j/N$ and $\{e_1, \dots, e_d\}$
represents the canonical basis of $\bb R^d$.

\smallskip\noindent{\bf Discrete vector fields.}
For technical issues, we need to define the empirical current also for
paths $\bs \eta$ not coming from the WASEP dynamics.  Consider an
arbitrary path $\bs \eta \in D\big(\mc R;\Sigma_{N,K}\big)$. If for
some $t\in\mc R$ it happens $\bs \eta(t) \neq \bs\eta(t-)$ some of the
particles in the configuration $\bs \eta(t-)$ have rearranged
themselves to construct the configuration $\bs\eta(t)$.  The
definition of the empirical current requires to decide the actual path
taken by those particles.

A \emph{discrete vector field} $W$ is an antisymmetric function 
$W\colon \bb B_N \to \bb R$. The \emph{discrete divergence} of the
discrete vector field $W$ is the function
$\nabla_N\cdot W\colon \bb T_N^d\to \bb R$ defined by
\begin{equation*}
(\nabla_N\cdot W) \, (x) : = \sum_{y\colon (x,y) \in \bb B_N}  W(x,y)\;. 
\end{equation*}

Fix $0\le K \le N^d$, and consider two configurations $\eta$,
$\xi\in \Sigma_{N,K}$. Let $W_{\eta,\xi}$ be a discrete vector field
which solves
\begin{equation}
\label{tp0}
(\nabla_N\cdot W_{\eta,\xi}) \, (x) \;=\; \eta_x \,-\,  \xi_x\;, \quad
x\,\in\, \bb T^d_N\;.
\end{equation}
Such a discrete vector field $W_{\eta,\xi}$ always exists. The
configuration $\zeta = \eta - \xi$ belongs to $\{-1,0,1\}^{\bb T^d_N}$
and $\sum_{x\in \bb T^d_N} \zeta_x=0$. To define $W_{\eta,\xi}$, one
just needs to create nearest-neighbor flows from each $x\in \bb T^d_N$
such that $\zeta_x=1$ to each $x'\in \bb T^d_N$ such that
$\zeta_{x'}=-1$, and add all these flows.

Regarding $\eta$ and $\xi$ as positive measures on $\bb T^d_N$, both
of mass $K$, by the analysis of the discrete Beckmann's problem, see
e.g.\ \cite{Sa}, we deduce that there exists a constant $C_0$ such
that for all $N\ge 1$, $0\le K\le N^d$, $\eta$, $\xi \in \Sigma_{N,K}$
there exists a discrete vector field $W_{\eta,\xi}$ such that
\begin{equation}
\label{beck}
\sum_{(x,y)\in \bb B_N} \big|\, W_{\eta,\xi} (x,y)\,\big|
\;\le\; C_0\,  N^{d+1}\;.
\end{equation}

Of course, equation \eqref{tp0} admits more than one solution and we
do not claim that there is only one satisfying the previous bound.  It
turns out, however, that the scaling limit of the empirical current
does not depend on the specific selection among those fulfilling
\eqref{beck}. Hence, in the sequel and without further mention, we
assume that, for each pair $(\eta, \xi) \in \Sigma^2_{N,K}$, a
discrete vector field which match \eqref{beck} $W_{\eta,\xi}$ has
been chosen.

When $\xi =\sigma^{x,y} \eta$ for some $(x,y)\in \bb B_N$ and
$\eta_x=1$, $\eta_y =0$, we define $W_{\eta,\xi}$ as
\begin{equation}
\label{tp2}
W_{\eta,\xi}(x',y') \;=\;
\begin{cases}
1 & \text{if } (x',y')=(x,y) \;, \\
-1 & \text{if }  (x',y')=(y,x)\;, \\
0 & \text{otherwise}\;.
\end{cases}
\end{equation}
This discrete vector field clearly satisfies \eqref{tp0} and \eqref{beck}.

\smallskip\noindent{\bf Integrated currents for generic paths.}
Fix a generic path $\bs \eta \in D(\mc R; \Sigma_{N,K})$ and denote by
$\tau_i$ its jump times. Let $W_i$ be the discrete vector field given by
\begin{equation}
\label{tp}
W_i \;=\; W_{\bs\eta (\tau_i-) , \bs\eta (\tau_i)} \;.
\end{equation}
For $t>0$, the integrated empirical current of the path $\bs \eta$
is then defined by
\begin{equation}
\label{39}
\langle \bs J_N(\bs \eta)(t), F \rangle = \frac 1{2 N^d}
\, \sum_{i\colon
  \tau_i\in(0,t]} \sum_{(x,y)\in \bb B_N} W_i(x,y) \, F_N(x,y)
\end{equation}
where $F_N$ is the discrete vector field constructed from
$F\in C(\bb T^d;\bb R^d)$ by \eqref{dvf}.  In view of \eqref{tp2}, for
trajectories $\bs \eta$ coming from the WASEP dynamics, this
definition corresponds to the original one, given in \eqref{ecurr}.

As before, $\bs J_N(\bs \eta)(0)=0$ and for $t<0$
\begin{equation*}
\langle \bs J_N(\bs \eta)(t), F \rangle = - \frac 1{2 N^d}
\sum_{i\colon \tau_i\in [t,0)} \sum_{(x,y)\in \bb B_N} W_i(x,y)
F_N(x,y).
\end{equation*}
In view of the bound \eqref{beck}
and noticing that $|F_N|\le \|F\|_\infty /N$, if $\bs\eta$ has a single jump
then $|\langle \bs J_N(\bs \eta)(t), F \rangle| \le C \|F\|_\infty$ for
some constant $C$ independent of $N$.

\subsection*{Sobolev spaces.} Let $h_n \in L^2(\bb T^d)$,
$n = (n_1, \dots, n_d) \in \bb Z^d$, be the orthonormal basis of
$L^2(\bb T^d)$ given by
$\color{bblue} h_n (x) = \exp \{2\pi i (n\cdot x)\}$.

Denote by
$\langle\cdot,\cdot\rangle$ the inner product in $L^2(\bb T^d)$,
and by $\mf f \colon \bb Z^d \to \bb C$ the Fourier coefficients of
the function $f$ in $L^2(\bb T^d)$:
\begin{equation*}
\mf f (n) := \langle f, h_n \rangle \;=\; \int_{\bb T^d} f(x) \,
\overline{h_n(x)} \, dx \;, \quad n\,\in\, \bb Z^d\;,
\end{equation*}
where $ \overline{z}$ represents the complex conjugate of
$z\in \bb C$. Hence,
\begin{equation*}
f \;=\; \sum_{n\in\bb Z^d} \mf f(n)\,  h_n\; 
\end{equation*}

Denote by $\color{bblue} \mc H_p$, $p\in\bb R$, the Hilbert space
obtained by completing the space of smooth complex-valued functions on
$\bb T^d$ endowed with the scalar product
$\< \, \cdot \,,\, \cdot\, \>_p$ defined by
\begin{equation}
\label{ap01}
\< f\,,\, g\>_p \;=\;  \< (1-\Delta)^p f \,,\, g\>\;, 
\end{equation}
where $\Delta$ represents the Laplacian. An elementary computation
yields that
\begin{equation}
\label{ap02}
\< f \,,\, g \>_p \; :=\; \sum_{n\in\bb Z^d} 
\big(1 + 4\pi^2 |n|^2\big)^p \, \mf f(n)  \, \overline{\mf g(n)} 
\; ,
\end{equation}
where $|n|^2=|(n_1, \ldots, n_d)|^2 = \sum_{1\le j\le d} n^2_j$.  Denote by
$\Vert \,\cdot\, \Vert_p$ the norm of $\mc H_p$:
$\color{bblue} \Vert f \Vert^2_p = \< f \,,\, f \>_p$.  It is well
known that $\mc H_{-p}$, is the dual of $\mc H_p$ relatively to the
pairing $\langle\,\cdot \,,\, \cdot \rangle$ defined by
\begin{equation}
\label{32}
\langle  j \,,\, g \rangle \; :=\; 
\sum_{n\in\bb Z^d} \mf j(n)  \, \overline{\mf g(n)} 
\; , \quad j \,\in\, \mc H_{-p} \;,\, g\,\in\, \mc H_{p} \;.
\end{equation}
Moreover, it follows from the definition that $\mc H_p \subset \mc
H_{p'}$ for $p>p'$.
Let $\color{bblue} \mc H^d_p = \mc H_p \times \cdots \times \mc H_p$
that we consider endowed with the strong topology.
We represent below by
$\color{bblue} \< J \,,\, H\>$ the value at $H\in \mc H^d_p$ of a
bounded linear functional $J$ defined on $\mc H^d_{p}$.

Fix $p>d/2$. By the Sobolev embedding,
$\mc H^d_p\subset C(\bb T^d;\bb R^d)$. In particular, by the
definition of the empirical current, for each $t\in \mc R$, the
functional $[\bs J_{N}(\bs \eta)]\, (t)$ is bounded on $\mc H^d_p$.
Therefore, for each $\bs \eta \in D(\mc R;\Sigma_{N,K})$ and
$t\in \mc R$ it belongs to dual of $\mc H^d_p$, that is, to
$\mc H^d_{-p}$. Furthermore, it is easy to check that $\bs J_{N}$ is
right-continuous and has left-limits. Hence, the empirical current
$\bs J_{N}$ is a map from $D (\mc R, \Sigma_{N,K})$ to
$D (\mc R, \mc H^d_{-p})$.

Fix $p\ge 0$ and let $J$ be a bounded linear functional on
$\mc H^d_{p}$, i.e.\ $J\in \mc H^d_{-p}$. By Riesz representation
theorem, there exists $G=G(J)$ in $\mc H^d_p$ such that
$J (F) = \< F\,,\, G\>_p$ for all $F$ in $\mc H^d_p$, and
$\Vert J \Vert^2_{-p} = \< G\,,\, G\>_p$.
Letting $G=(G_1,\ldots,G_d)$ and defining $J_k(n)$ as
$\{1 + 4\pi^2 |n|^2\}^p \, \mf G_k(n)$, where $\mf G_k(n)$,
$n\in \bb Z^d$, are the Fourier coefficients of $G_k$, $k=1,\ldots,d$,
by \eqref{ap02} we deduce that
\begin{equation*}
J (F)   \;=\; \sum_{k=1}^d \sum_{n\in\bb Z^d} 
\mf F_k(n)  \, \overline{J_k(n)}\,, \; \qquad F=(F_1,\ldots, F_d) \; .
\end{equation*}
For $n\in \bb Z^d$, $j=1,\ldots, d$, let
$H^{j,n} \colon \bb T^d \to \bb C^d$ be the vector fields given by
$H^{j,n} = ( H^{j,n}_1, \dots, H^{j,n}_d)$ where
$H^{j,n}_k = \delta_{j,k} \, h_n$.  Taking $F= H^{j,n}$ in the
previous displayed identity, one concludes that
$\overline{J_k(n)} = J(H^{k,n})$, so that
\begin{equation}
\label{ap03}
J (F)  \;=\; \sum_{k=1}^d \sum_{n\in\bb Z^d} 
\mf F_k(n)  \,   J(H^{k,n})\,, \qquad
\Vert J \Vert^2_{-p} \;=\; \sum_{k=1}^d \sum_{n\in\bb Z^d}
\frac {\big|\, J(H^{k,n}) \,\big|^2}
{\big(1 + 4\pi^2 |n|^2\big)^p} \;\cdot
\end{equation}

\subsection*{Continuity equation.}
It follows from the conservation of mass that for each $x\in\bb T^d_N$,
$t>0$, and each path $\bs \eta$ compatible with the WASEP dynamics
\begin{equation*}
\bs \eta_x(t) \;-\; \bs \eta_x(0) \;=\; \sum_{j=1}^d 
\Big\{ \, \ms N^{x+ \mf e_j, x}_{(0,t]}(\bs \eta) \,-\,
\ms N^{x,x+ \mf e_j}_{(0,t]}(\bs \eta) \,+\, 
\ms N^{x- \mf e_j, x}_{(0,t]}(\bs \eta) \,-\,
\ms N^{x,x- \mf e_j}_{(0,t]}(\bs \eta) \, \Big\}\;.
\end{equation*}
Let $f$ be a function in $C^\infty(\bb T^d)$ and recall the notation
introduced in \eqref{34}. Multiply the previous equation by $f(x)$,
sum over $x\in\bb T^d_N$ and divide by $N^d$ to get, after a summation
by parts, that
\begin{align*}
& \< [\, \bs \pi_N (\bs \eta) \, ]\, (t) \,,\, f\>
\;-\; \< [\, \bs \pi_N (\bs \eta) \, ]\, (0) \,,\, f \> \\ 
&\qquad \;=\;  \frac 1{N^d} \sum_{j=1}^d \sum_{x\in\bb T^d_N} 
\big\{ \, f(x+\mf e_j) - f(x) \, \big\} \,
\big\{ \, \ms N^{x,x+ \mf e_j}_{(0,t]}(\bs \eta)
\,-\, \ms N^{x+ \mf e_j, x}_{(0,t]}(\bs \eta) \, \big\}\;.
\end{align*}
Since
$f(x+\mf e_j) - f(x) = \int_{x}^{x+\mf e_j} (\nabla f) \cdot d\ell$,
in view of the definition of the map $\bs J_N$,
\begin{equation}
\label{31bis}
\< [\, \bs \pi_N (\bs \eta) \, ]\, (t) \,,\, f\>
\;-\; \< [\, \bs \pi_N (\bs \eta) \, ]\, (0) \,,\, f \> 
\;=\; \< \bs J_N (\bs \eta) \, (t) \,,\, \nabla f\> \;.
\end{equation}
This is the microscopic version of the continuity equation.  Observe
that for paths $\bs \eta$ not coming from the WASEP dynamics, the
definition of the integrated empirical current $\bs J_N$ has been
engineered so that \eqref{31bis} always holds.

\subsection*{Hydrodynamical limit}
Let $(\eta^N : N\ge 1)$, $\eta^N \in \Sigma_N$, be a sequence of
configurations associated to a density profile
$\rho\colon \bb T^d \to [0,1]$ in the sense that for each continuous
function $f\colon \bb T^d \to \bb R$,
\begin{equation*}
\lim_{N\to\infty} \< \pi_N(\eta^N), f\>  \;=\;
\lim_{N\to\infty} \frac 1{N^d} \sum_{x\in\bb T^d_N} f(x) \, \eta^N_x
\;=\; \int_{\bb T^d} f(x)\, \rho(x)\, dx\;.
\end{equation*}
It is proven in \cite{bdgjl07} that $\big(\bs
\pi_N, \bs J_{N} )$ converge in $\bb P^N_{\eta^N}$-probability, as
$N\to\infty$, to $\big(\bs\rho(t,\cdot) \, dx \, ,\, \int_0^t\!ds\,
\bs j (s,\,\cdot\,) \big)_{t\ge 0}$ where $(\bs\rho,\bs j)$ is
the unique weak solution to the Cauchy problem
\begin{equation}
\label{26}
\left\{
\begin{aligned}
& \partial_t \bs\rho \;+\; \nabla \cdot \bs j \;=\; 0\;, \\
& \bs j \;=\; -\, \nabla \bs \rho \;+\; \sigma(\bs \rho)\, E\;, \\
& \bs \rho(0,\,\cdot\,) \;=\; \rho(\,\cdot\,)
\;,
\end{aligned}
\right.
\end{equation}
in which ${\color{bblue}\sigma} \colon [0,1]\to \bb R_+$, given by
$\sigma(\rho) = \rho(1-\rho)$, is the mobility of the exclusion
process.

If one considers only the empirical density and disregards the
empirical current, the above result has been proven in \cite{dps,kov},
see also \cite[Ch.~10]{kl}.  The case in which one considers also the
empirical current is discussed in \cite{bdgjl07} for the SEP.  The
topology on the set of currents used in \cite{bdgjl07} is different
from the one employed in the present paper.  Actually, the proof of
the tightness of the empirical current in \cite{bdgjl07} is incomplete
but the arguments presented below in Section~\ref{sec03} fix this
issue (in the topology here introduced).  In view of the
super-exponential estimates in \cite{kov} or \cite[Ch.~10]{kl}, the
hydrodynamical limit extends directly to the WASEP dynamics.

\subsection*{Empirical process}
Given $K=0,\ldots, N^d$, $T>0$, and
$\bs\eta \in D(\bb R_+, \Sigma_{N,K})$, let
$\bs\eta^T \in D(\bb R, \Sigma_{N,K})$ be the $T$-periodization of the
trajectory $\bs\eta$, defined by
\begin{equation*}
\bs\eta^T(t) \;=\; 
\bs\eta \Big( t - \Big \lfloor \frac
{t}{T} \Big\rfloor \, T \Big) \;,
\end{equation*}
where $\color{bblue} \lfloor a\rfloor$ represents the largest integer
less than or equal to $a\in \bb R$.
A probability measure on $D(\bb R, \Sigma_{N,K})$ is stationary if it
is invariant with respect to the group of time-translations
$(\vartheta_t : t\in\bb R)$, defined by
\begin{equation*}
(\vartheta_t \bs \eta)(s) \;:=\; \bs \eta(s- t) \;, \quad s\,\in\, \bb
R\;. 
\end{equation*}
Denote by $\color{bblue} \ms P^{N,K}_{\rm stat}$ the set of stationary
probability measures on $D(\bb R, \Sigma_{N,K})$ that we consider
endowed with the topology induced by the weak convergence and the
corresponding Borel $\sigma$-algebra.  For a trajectory
$\bs\eta \in D(\bb R_+, \Sigma_{N,K})$ and $T>0$, the \emph{empirical
process} $R_T(\bs \eta)$ is the element in $\ms P^{N,K}_{\rm stat}$
given by
\begin{equation}
\label{empr}
R_T (\bs\eta) \;:=\; \frac 1T \int_0^T \delta_{\vartheta_t\bs\eta^T} \,dt \;.
\end{equation}

Observe that the path $\bs\eta^T$ is not necessarily compatible with
the WASEP dynamics; indeed at the times that are integral multiples of
$T$ there is a jump that is not - in general - coming from the WASEP
dynamics. This is the reason for which we needed to define the
empirical current for generic paths. However, in view of the bound
\eqref{beck}, there exists a finite constant $C_0$, depending only on
the space dimension $d$, such that
\begin{equation}
\label{djj}
\sup_{t\in[0,T]} \Big| \,\langle\bs J_N ({\bs\eta}^T)(t) -\bs J_N
({\bs\eta})(t)\,,\, F\rangle \, \Big| \;\le\; C_0\, \|F\|_\infty
\end{equation}
for all vector field $F$ in $C(\bb T^d, \bb R^d)$.  Indeed, by the
definition of ${\bs\eta}^T$ and by \eqref{39}, the left-hand side is
equal to
\begin{equation*}
\Big| \,\langle\bs J_N ({\bs\eta}^T)(T) -\bs J_N
({\bs\eta})(T)\,,\, F\rangle \, \Big| \;=\;
\Big| \, \frac{1}{2N^d}\sum_{(x,y)\in \bb B_N} W(x,y) \, F_N(x,y)  \, \Big| \;, 
\end{equation*}
where $W(x,y) = W_{\bs\eta (T-) , \bs\eta (0)}$. Hence, by
\eqref{beck}, the right-hand side of the previous identity is bounded
by $C_0 \Vert F \Vert_\infty$, as claimed in \eqref{djj}.

Fix $p>(d+2)/2$, and let
$\color{bblue} \ms M = \mc M_+ (\bb T^d) \times \mc H^d_{-p}$. Denote
by $\color{bblue} \ms S$ the closed subset of
$D (\bb R, \ms M)\equiv D \big(\bb R; \mc M_+ (\bb T^d)) \times D
\big(\bb R; \mc H^d_{-p})$ given by the pairs $(\bs \pi, \bs J)$ which
satisfy the continuity equation in the sense that for each $s<t$ and
$f$ in $C^{\infty}(\bb T^d)$,
\begin{equation}
\label{31}
\int_s^t ds_1 \int_{s_1}^t ds_2 \; 
\Big\{ \< \bs \pi  (s_2) \,,\, f\> \;-\; \< \bs \pi (s_1) \,,\, f \> 
\;-\;  \< \bs J (s_2) - \bs J (s_1) \,,\, \nabla f \> \Big\} \;=\; 0\;.
\end{equation}
We consider $\ms S$ endowed with the relative topology and the
associated Borel $\sigma$-algebra.
Denote by $\vartheta_t \colon \ms S \to \ms S$, $t\in\bb R$, the
time-translation defined by
\begin{equation}
\label{36}
\vartheta_t (\bs \pi, \bs J) \;=\;
(\vartheta_t \bs \pi, \vartheta_t \bs J)\;,
\end{equation}
where $(\vartheta_t \bs \pi)(s) = \bs \pi(s-t)$, and
$(\vartheta_t \bs J)(s) = \bs J (s-t) - \bs J(-t)$, $s\in\bb R$.  
Note that the time translations defined on $D(\bb R;\Sigma_{N,K})$ and
$\ms S$ are compatible in the sense that
$\vartheta_t \circ (\bs \pi_N,\bs J_N) = (\bs \pi_N,\bs J_N) \circ
\vartheta_t$.

Given a path $(\bs\pi, \bs J)$ in $\ms S$ and $T>0$, denote by
$(\bs\pi^T, \bs J^T)$  its $T$-periodization:
\begin{equation}
\label{Tper}
\begin{split}
& \bs\pi^T(t) \;=\; \bs\pi \Big( t - \Big\lfloor \frac{t}{T}
\Big\rfloor \Big) \;,\\
& \bs J^T (t) \;=\; \bs J \big( t - \big\lfloor \tfrac{t}{T}
\big\rfloor\big) \;+\; \big\lfloor \tfrac{t}{T} \big\rfloor \big( \bs
J(T) +A \big)
\end{split}
\end{equation}
where $A$ is an element of $\mc H^d_{-p}$ satisfying
$\nabla\cdot A \;+\; \bs\pi(T) \,-\, \bs\pi(0) = 0$. A straightforward
computation shows that $(\bs\pi^T, \bs J^T)$ satisfies the continuity
equation \eqref{31}. We hereafter assume that the choice of $A$ in the
previous definition and of the discrete vector field $W$ in \eqref{tp0} 
are compatible in the sense that whenever 
$\bs\pi(0)=N^{-d}\sum_{x\in \bb T^d_N}\eta_x$ and
$\bs\pi(T)=N^{-d}\sum_{x\in \bb T^d_N}\xi_x$ 
for some $\eta,\xi\in \Sigma_N$ then 
\begin{equation*}
A = \frac 1 {N^d} \sum_{x\in \bb T^d_N}\sum_{k=1}^d
W_{\eta,\xi}(x,x+\mf e_k) \,e_k \, \mf H^1_{\res [x,x+\mf e_k]}
\end{equation*}
where $\mf H^1_{\res [x,x+\mf e_k]}$ is the one-dimensional Hausdorff
measure of the interval $[x,x+\mf e_k]$.  This compatibility implies
that
$(\, \bs\pi_N(\bs\eta^T) \,,\, \bs J_N(\bs \eta^T)\, ) =
(\bs\pi_N(\bs\eta)^T, \bs J_N(\bs \eta)^T )$, for each
$\eta\in D(\bb R_+; \Sigma_{N,K})$.

Given $\bs \eta\in D(\bb R_+;\Sigma_{N,K})$, let finally $\mf R_{T,N}
(\bs \eta)$ be the stationary probability (with respect to the map
$\vartheta_t$ defined above) on $\ms S$ given by
\begin{equation}
  \label{Rgot}
\mf R_{T,N} (\bs \eta) := \frac 1T \int_0^T \delta_{\big(
  \bs\pi_N(\vartheta_t \bs\eta^T)\,,\,\bs J_N(\vartheta_t
  \bs\eta^T)\big)} \,dt = \frac 1T \int_0^T \delta_{\vartheta_t \big(
  \bs\pi_N(\bs\eta)^T\,,\,\bs J_N(\bs\eta)^T\big)} \,dt
\end{equation}
and observe that, according to definitions \eqref{empden},
\eqref{ecurr}, \eqref{empr}, we have
\begin{equation*}
\mf R_{T,N} = R_T \circ \big(\bs \pi_N,\bs J_N\big)^{-1}\;.  
\end{equation*}

\subsection*{Large deviations asymptotic} 
Our main result establishes the large deviations principle for
$\mf R_{T,N}(\bs\eta)$ in the joint limit $T\to \infty$ and
$N\to\infty$ when $\bs \eta$ is sampled according to the WASEP
dynamics. We prove this result both when $T\to \infty$ before
$N\to\infty$ and when $N\to\infty$ before $T\to \infty$.  The
corresponding rate function is independent of the limiting procedure.

The statement of this result requires further notation.
Denote by $\color{bblue} \ms S_{m}$, $m\in (0,1)$, the closed set of
trajectories $(\bs \pi, \bs J)$ in $\ms S$ such that
$\bs \pi (t , \bb T^d) =m$ for all $t\in \bb R$.  Recalling that
$\sigma(\rho) = \rho(1-\rho)$ is the mobility of the exclusion
process, let finally $\color{bblue} \ms S_{m, {\rm ac}}$ be the
subset of elements $(\bs \pi, \bs J)$ in $\ms S_{m}$ such that
\begin{itemize}
\item[(a)] $\bs\pi \in C(\bb R, \mc M_m (\bb T^d))$,
  $\bs\pi(t,dx) = \bs \rho(t,x)\, dx$ for some $\bs \rho$ such that
  $0\le \bs \rho(t,x) \le 1$, and, for any $T>0$,
\begin{equation*}
 \int_{-T}^T dt \int_{\bb T^d} dx\, \frac{|\, \nabla
   \bs\rho\,|^2}{\sigma(\bs\rho)}\: < \; \infty\;;
\end{equation*}
\item[(b)] $\bs J \in C(\bb R, \mc H^d_{-p})$, and
  $ \bs J (t) \,=\, \int_0^t \bs j(s)\, ds$, $t\in\bb R$,
  for some $\bs j$ in
  $L^2_{\rm loc} (\bb R \times \bb T^d, \sigma(\bs \rho(t,x))^{-1} \,
  dt\, dx ; \bb R^d)$. Thus, for any $T>0$
  \begin{equation*}
    \int_{-T}^T dt \int_{\bb T^d} dx\, \frac{|\, \bs
      j\,|^2}{\sigma(\bs \rho)}\: < \; \infty\;.
  \end{equation*}
\end{itemize}

Let the \emph{action} $A_{m,T}\colon \ms S \to [0, +\infty]$,
$m\in (0,1)$, $T>0$, be defined by
\begin{equation}
  \label{hyld}
A_{m,T} (\bs\pi, \bs J) =
\begin{cases}
\displaystyle \int_{0}^T dt\, \int_{\bb T^d} dx\,
\frac{|\bs j + \nabla
\bs \rho - \sigma(\bs \rho)\, E|^2}{4\, \sigma(\bs \rho)}  & \text{
if $(\bs \pi, \bs J) \in \ms S_{m, {\rm ac}}$\;, } \\ \\
+\infty &  \text{ otherwise.}
\end{cases}
\end{equation}
By the arguments presented in \cite[\S~4]{blm}, the functional
$A_{m,T}$ is lower semicontinuous.
Note that if $\bs J$ has a density $\bs j$ as in item (b) above, then  
the density of $\vartheta_t \bs J$ is given by 
$(\vartheta_t\bs j)(s) := \bs j (s-t)$. 

For $m\in (0,1)$, let $\color{bblue} \ms P_{{\rm stat}}$,
$\color{bblue} \ms P_{{\rm stat},m}$, be the set of translation
invariant probability measures on $\ms S$, $\ms S_m$, respectively.
We consider $\ms P_{{\rm stat}}$ and $\ms P_{{\rm stat},m}$ endowed
with the topology induced by the weak convergence and the
corresponding Borel $\sigma$-algebra.  Let
$\bs I_m\colon \ms P_{{\rm stat}} \to [0, +\infty]$ be the functional
defined by
\begin{equation}
\label{06}
\bs I_m (P) \; :=\;  \frac {1}{T}\, E_P \big[ \, A_{m,T} \,\big] \;,
\end{equation}
observing that the right-hand side does not depend on $T>0$ by
stationarity. Moreover, by using the continuity equation \eqref{31},
the identity $\nabla \bs \rho/ \sigma(\bs \rho) = \nabla h'(\bs
\rho)$, where $h(\rho):= \rho\log \rho + (1-\rho) \log (1-\rho)$ is
the Bernoulli entropy, and the stationarity of $P$, integrating by
parts we deduce that if $\bs I_m(P) <+\infty$ then for any $T>0$
(equivalently for some $T>0$)
\begin{equation}
\label{fe}
\frac {1}{T}\, E_P \bigg[ \int_{-T}^T dt \int_{\bb T^d}
dx\, \Big( \frac{|\, \nabla \bs\rho\,|^2}{4\, \sigma(\bs\rho)}
+\frac{|\, \bs j\,|^2}{4\, \sigma(\bs \rho)} \Big) \bigg] < +\infty
\;.
\end{equation}

In the next statements, by $\limsup_{T,N}$, we mean either
$\limsup_{N} \limsup_{T}$ or $\limsup_{T} \limsup_{N}$. Analogously,
$\liminf_{T,N}$ stands for either $\liminf_{N} \liminf_{T}$
or $\liminf_{T} \liminf_{N}$.

\begin{theorem}
\label{cor01}
Fix $m\in(0,1)$, $p>(d+2)/2$, and a sequence $K_N$ such that $K_N/
N^d\to m$.  For each closed subset $\ms F$ of $\ms P_{\rm stat}$
\begin{equation*}
\limsup_{N,T\to \infty}\;
\sup_{\eta\in \Sigma_{N,K_N}} \frac 1{N^d}\, \frac 1T \,
\log \bb P^N_\eta \,\big[ \mf R_{T,N}  \in
\ms F\,]  
\;\le\; -\, \inf_{P \in \ms F} \bs I_{m}(P) \;.
\end{equation*}
Moreover, if $E$ is orthogonally decomposable, then, for each open
subset $\ms G$ of $\ms P_{\rm stat}$,
\begin{equation*}
\liminf_{N,T\to \infty}\; \inf_{\eta\in \Sigma_{N,K_N}} \frac 1{N^d}\,
\frac 1T \,\log \bb P^N_\eta \,\big[ \mf R_{T.N} \in \ms G\,] \;\ge\;
-\, \inf_{P \in \ms G} \bs I_{m} (P)\;.
\end{equation*}
Finally, the functional $\bs I_{m} \colon \ms P_{\rm stat} \to [0, +\infty]$ is
good and affine.
\end{theorem}

In the lower bound, the technical condition that the external field
$E$ is orthogonally decomposable is only used to guarantee that the
quasi-potential of the WASEP is bounded, an ingredient that enters in
the proof of Lemma~\ref{t:l4.7}.  In the common terminology of large
deviations, the previous statement is analogous to a level-three large
deviation principle, for which the rate function has an explicit
expression. The contraction principle permits to derive from this
result large deviations principles for relevant observables.

\subsection*{Level two large deviations}

Let $\color{bblue} \ms P(\mc M_{+}(\bb T^d))$ be the space of
probability measures on $\mc M_{+} (\bb T^d)$ endowed with the weak
topology. Recalling \eqref{empden}, define $\wp_{T,N}$ as the map from
$D\big(\bb R_+;\Sigma_{N,K}\big)$ to $\ms P\big(\mc M_+(\bb T^d)\big)$
by
\begin{equation*}
\wp_{T,N}(\bs\eta) := \frac 1T \int_0^T\!dt\,
\delta_{\pi_N(\bs\eta(t))}\,,
\end{equation*}
i.e.\ $\wp_{T,N}$ is the time empirical measure associated to the path
$\pi_N(\bs\eta)$.  Letting $\imath_t\colon \ms S\to\mc M_+(\bb T^d)$
be the map $(\bs \pi,\bs J) \mapsto \bs \pi (t)$, then
$\mf R_{T,N} \circ \imath_t^{-1} = \wp_{T,N}$. Denote by $\ms
I_m\colon \ms P\big(\mc M_+(\bb T^d)\big) \to [0,+\infty]$, $m\in
(0,1)$, the functional given by
\begin{equation}
  \label{Is}
  \ms I_m (\wp ) = \inf \big\{ \bs I_m (P) : P \in \ms P_{{\rm stat}}
  \,,\, P\circ \imath_t^{-1} \, = \wp \big\}\,.
\end{equation}

Given $m\in (0,1)$, let $(\Phi^m_t : t\ge 0)$ be the flow induced by the
hydrodynamic equation \eqref{26} on the set of densities with total
mass equal to $m$. Namely, when $\int \rho \, dx =m $ we set
$\Phi^m_t(\rho)=\bs \rho(t)$ where $(\bs \rho,\bs j)$ is the unique weak
solution to \eqref{26}. By identifying measures absolutely continuous
with respect to $dx$ with their densities, we regard
$(\Phi^m_t : t\ge 0)$ as a flow on $\mc M_m(\bb T^d)$.  The following
result is obtained from Theorem~\ref{cor01} by the contraction principle
and implies the \emph{hydrostatic limit}: in the limit $N\to \infty$
the empirical density constructed by sampling the particles according
to the stationary measure $\mu_{N,K}$ converges to the unique
stationary solution to the hydrodynamic equation.

\begin{corollary}
\label{cor02}
Fix $m \in(0,1)$ and a sequence $K_N$ such that $K_N/ N^d\to m$.  For
each closed subset $\ms F$ of $\ms P\big(\mc M_+(\bb T^d)\big)$
\begin{equation*}
\limsup_{N,T\to \infty}\;
\sup_{\eta\in \Sigma_{N,K_N}} \frac 1{N^d}\, \frac 1T \,
\log \bb P^N_\eta \,\big[ \wp_{T,N} \in \ms F\,]  
\;\le\; -\, \inf_{\wp \in \ms F} \ms I_{m}(\wp) \;.
\end{equation*}
If $E$ is orthogonally decomposable, then for each open subset $\ms G$
of $\ms P\big(\mc M_+(\bb T^d)\big)$
\begin{equation*}
\liminf_{N,T\to \infty}\; \inf_{\eta\in \Sigma_{N,K_N}} \frac 1{N^d}\,
\frac 1T \,\log \bb P^N_\eta \,\big[ \wp_{T,N} \in \ms G\,] \;\ge\;
-\, \inf_{\wp \in \ms G} \ms I_{m} (\wp)\;.
\end{equation*}
Finally, the functional $\ms I_{m} \colon \ms P\big(\mc M_+(\bb T^d)\big)\to
[0, +\infty]$ is good, convex, and vanishes only on the
invariant probabilities for the flow $\Phi^m$. In particular, if $E$
is orthogonally decomposable, then $\ms I_m(\wp)=0$ if and only if
$\wp=\delta_{\bar\rho\, dx}$ where $\bar\rho$ is the unique stationary
solution to the hydrodynamic equation with mass $m$.
\end{corollary}

\subsection*{Level one large deviations} 
For $T>0$ the \emph{time-averaged empirical density} is the map
$\pi_{T,N}: D(\bb R_+, \Sigma_N) \to \mc M_+(\bb T^d)$ defined by
\begin{equation}
\label{05}
\pi_{T,N} (\bs \eta)  \;: =\; \frac{1}{T} \int_0^T
[\, \bs \pi_N (\bs \eta) \, ]\, (t) \, dt\;.
\end{equation}
Likewise, for $p>d/2$, the \emph{time-averaged empirical current} is the map
$J_{T,N}$ from  $D(\bb R_+, \Sigma_{N,K})$ to $\mc H^d_{-p}$ defined by
\begin{equation}
\label{04}
J_{T,N} (\bs \eta) \;=\; \frac{1}{T}
\, [\, \bs J_{N} (\bs \eta) \, ] \,  (T) \;,
\end{equation}
which can also be written as 
\begin{equation*}
\< J_{T,N} (\bs \eta) \,,\, F\> \;=\; \frac{1}{T} 
\frac 1{N^d} \sum_{(x,y)\in \bb B_N} \ms N^{x,y}_{(0,T]}(\bs \eta)
\int_{x}^{y} F  \cdot  d\ell \;, 
\quad F \,\in\, C^\infty(\bb T^d; \bb R^d)\;.
\end{equation*}
Note that
\begin{equation}
\label{pl1}
(\pi_{T,N},J_{T,N}) = \int\! d \mf R_{T,N} \, \Big (\bs\pi (t), \frac
1t \bs J(t) \Big) + \frac 1T \big(0, \mc E_{T,N} \big)
\end{equation}
where the first term on right hand side does not depend on $t\neq 0$
by the stationarity of $\mf R_{T,N}$ and
\begin{equation*}
  \mc E_{T,N}(\bs\eta)  =
  \bs J_N (\bs \eta^T (T))- \bs J_N (\bs \eta(T))
  \;.
\end{equation*}
By \eqref{djj} and the Sobolev embedding, for $p> d/2$ we deduce that
$\| \mc E_{T,N}(\bs\eta)\|_{-p}$ is bounded uniformly in $\bs\eta$,
$T$, and $N$. Hence the second term on the right hand side of
\eqref{pl1} is irrelevant for large deviations in the asymptotics
$T\to\infty$.

Let $I_m \colon \ms M \to [0, +\infty]$, $m\in (0,1)$, be
the functional defined by
\begin{equation*}
I_m (\pi, J) \; :=\; \inf \big\{ \, \bs I_m (P) : P \in \ms P_{{\rm
    stat}} \,,\, E_P[\, \bs \pi (t)\, ] \,=\, \pi \,,\, E_P[\, \bs J
(\, t \, ) \,] \, =\, t J \, \}\;,
\end{equation*}
which does not depend on $t\neq 0$.
If the vector field $J$ is not divergence free, the set
on the right-hand side is empty. Indeed, by stationarity and the
continuity equation \eqref{31}, if the above constraints are
satisfied, we deduce that for each smooth function $f$ on $\bb T^d$
and $t>0$
\begin{equation*}
0 \;=\; E_P \Big[ \int_0^t  \<\bs\pi(s) -\bs \pi(0)\,,\, f\>
\, ds  \, \big]
\;=\;  E_P \big[ \int_0^t \< \bs J(s) \,,\, \nabla f \> \, ds \, \Big]
\;=\;  \frac{t^2}{2}\, \< J \,,\, \nabla f \> \;.   
\end{equation*}

By the contraction principle, Theorem~\ref{cor01} implies the following
statement.

\begin{corollary}
\label{th01}
Fix $m\in(0,1)$, $p>(d+2)/2$, and a sequence $K_N$ so that $K_N/ N^d\to m$.
For each closed subset $\ms F$ of $\ms M$
\begin{equation*}
\limsup_{T,N\to \infty} \; \sup_{\eta\in \Sigma_{N,K_N}} \frac
{1}{N^d} \frac 1T \,\log \bb P^N_\eta \,\big[ (\pi_{T,N} , J_{T,N})
\in \ms F\,\big] \;\le\; -\, \inf_{(\pi , J)\in \ms F} I_m (\pi ,
J)\;.
\end{equation*}
Moreover, if $E$ is orthogonally decomposable, then for each open subset
$\ms G$ of $\ms M$,
\begin{equation*}
\liminf_{T,N\to \infty} \; \inf_{\eta\in \Sigma_{N,K_N}} \frac
{1}{N^d} \frac 1T \,\log \bb P^N_\eta \,\big[ (\pi_{T,N} , J_{T,N})
\in \ms G\,\big] \;\ge\; -\, \inf_{(\pi , J)\in \ms G} I_m (\pi ,
J)\;.
\end{equation*}
Finally, the functional $I_m \colon   \ms M \to [0, +\infty]$ is 
good and convex.
\end{corollary}

The projections of $I_m$ on the two components can be further analyzed
and computed explicitly under additional conditions which are
satisfied, for instance, in the SEP case.  Denote by
$I_m^{(1)}\colon \mc M_+(\bb T^d) \to [0,+\infty]$ the projection of
the functional $I_m$ on the density, i.e.,\
\begin{equation}
\label{vfd}
I_m^{(1)} (\pi) \; =\; \inf \big\{ \, \bs I_m (P) : P \in \ms P_{{\rm
    stat},m} \,,\, E_P[\, \bs \pi (\, t \, ) \,] \, =\, \pi \, \}\;.
\end{equation}
It turns that when the external field $E$ is a gradient, so that the
WASEP dynamics is reversible, then $I_m^{(1)}$ can be computed
explicitly.  Assume $E=-\nabla U$ for some $U\in C^2(\bb T^d)$ and let
$\ms V_m \colon \mc M_m(\bb T^d)\to [0,+\infty]$ be the functional
defined by
\begin{equation*}
\ms V_m (\pi) :=
\begin{cases}
\displaystyle{\int_{\bb T^d} \!dx \, \frac{ \big|\nabla \rho
    +\sigma(\rho) \nabla U \big|^2}{ 4\, \sigma(\rho)}} & \textrm{ if
  $\pi(\bb T^d) =m$ and $\pi(dx)=\rho \,dx$,}
\\
\vphantom{\Big\{} +\infty & \textrm{ otherwise.}
\end{cases}
\end{equation*}
Let also $\textrm{co}(\ms V_m)$ be the convex hull of $\ms V_m$ and
observe that, in view of the concavity of $\sigma$, if $\nabla U =0$
then $\mathrm{co}(\ms V_m)=\ms V_m$.  The functional $\ms V_m$ can be
seen as a non-linear version of the level two Donsker-Varadhan
functional for reversible diffusions (sometimes called Fisher
information). Indeed, in the case of independent particles
$\sigma(\rho)=\rho$ and the functional $\ms V_m$ reduces to the
Dirichlet form of the square root for the diffusion on $\bb T^d$ with
generator $\Delta -\nabla U\cdot \nabla$.

\begin{theorem}
\label{t:I1}
If  $E=-\nabla U$, then $I^{(1)}_m= \mathrm{co}(\ms V_m)$.
\end{theorem}

As discussed in the introduction, the projection of $I_m$ on the
second component is related to the possible occurrence of dynamical
phase transitions for the current.
For $p>(d+2)/2$, $m\in (0,1)$, denote by
$I_m^{(2)}\colon \mc H^d_{-p}\to [0,+\infty]$ the projection of the
functional $I_m$ on the current, i.e.\
\begin{equation}
\label{vfc}
I_m^{(2)} (J) \; =\; \inf \big\{ \, \bs I_m (P) : P \in \ms P_{{\rm
    stat},m} \,,\, E_P[\, \bs J (\, t \, ) \,] \, =\, t\, J \, \}\;,
\end{equation}
that corresponds to the Varadhan's proposal informally presented in
\eqref{0I2}.  By the contraction principle, Corollary~\ref{th01}
implies that the time-averaged empirical current $J_{T,N}$ satisfies a
large deviation principle with rate function $I_m^{(2)}$.

It has been pointed out in \cite{bdgjl06, bdgjl07, bd} that the
variational problem \eqref{vfc} has a non-trivial solution when $E$ is
constant and large enough. Such behavior is interpreted as a dynamical
phase transition.  Strictly speaking, the problem \eqref{vfc} is not
really considered in \cite{bdgjl06, bdgjl07, bd}, but the analysis
performed there implies the results summarized in the next
statement. We restrict to the one-dimensional case with constant
external field. Since $I_m^{(2)} (J) <+\infty$ implies
$\nabla\cdot J =0$, in the one-dimensional case, $I^{(2)}_m$ is finite
only if $J(x)= j$ for some constant $j\in \bb R$.

\begin{theorem}
\label{tdpt}
Let $d=1$, $m\in(0,1)$, and $E\ge 0$ be constant.
\begin{itemize}
\item[(i)] There exists $E_0>0$ such that if $E\le E_0$ then for
  $J\,=\, j $, $j \in\bb R$,
\begin{equation*}
I^{(2)}_m (J) \;=\; \frac {(j-\sigma(m) E)^2}{ 4 \, \sigma(m)}\;\cdot
\end{equation*}
The optimal $P$ for the variational problem \eqref{vfc} is
$\delta_{(m, j)}$.

\item[(ii)] There exists $E_1>E_0$ such that if $E\ge E_1$ then for
  $J = j$, $j \in\bb R$, with $j$ large enough
\begin{equation*}
I^{(2)}_m (J) \;<\; \frac {(j-\sigma(m) E)^2}{ 4 \, \sigma(m)}\;\cdot
\end{equation*}
Furthermore, taking the time average of a probability concentrated on a
traveling wave provides a measure $P$ in $\ms P_{{\rm stat},m}$ such
that $E_P[\, \bs J (\, t \, ) \,] \, =\, t \,J$,
$\bs I_m (P) < \bs I_m(\delta_{(m, j)})$.
\end{itemize}
\end{theorem}

Regarding the higher dimensional case, we mention that the argument in
\cite[Prop.~5.1]{bdgjl07} implies that in the SEP case ($E=0$) for $J$
with vanishing divergence we have
\begin{equation*}
  I^{(1)}_m(J) = \inf_{\rho} \int_{\bb T^d} \!
  \frac { |J+\nabla \rho|^2}{4 \sigma(\rho)}\, dx \;,
\end{equation*}
where the infimum is carried out over the density profiles $\rho$ of
mass $m$. In other words, the infimum in \eqref{vfc} is achieved for a
probability measure of the form $P= \delta_{(\rho \,dx\,,\, J)}$ and
no dynamical phase transition occurs.



\section{Donsker-Varadhan large deviations principle}
\label{s3}

In this section, we recall the Donsker-Varadhan large deviations
principle in the context of the WASEP dynamics with fixed number of
particles. 

Recall from \eqref{empr} the definition of the empirical process
$R_T(\bs \eta)$.  Referring to \cite{Var84} for equivalent
characterizations, we introduce the rate functional for the family
$(R_T: T>0)$ by a variational representation that will be most useful
for our purposes.  For $t>0$, let
$H_{N,K}(t, \,\cdot\,) : \ms P^{N,K}_{\rm stat} \to [0,+\infty]$ be
the functional given by
\begin{equation*} 
H_{N,K} (t,\bb Q) \;=\; \sup_{\Phi}
\int\! d \bb Q (\bs \eta) \Big[  \Phi (\bs \eta)
\,-\, \log \bb E_{\bs\eta(0)}^N \big( e^{\Phi} \big) \Big]\;,
\end{equation*}
where the supremum is carried over the bounded and continuous
functions $\Phi$ on $ D(\bb R, \Sigma_{N,K})$ that are measurable with
respect to $\sigma\{ \bs\eta(s), \: s\in[0,t]\big\}$.  Let $H_{N,K}
\colon \ms P^{N,K}_{\rm stat} \to [0,+\infty]$ be the functional
defined by
\begin{equation} 
\label{20}
H_{N,K}(\bb Q) \;:=\;
\sup_{t>0} \frac 1t \, H_{N,K} (t,\bb Q)
\;=\;
\lim_{t\to\infty} \frac 1t \, H_{N,K} (t,\bb Q)\;,
\end{equation}
where the second identity follows from the inequality before
\cite[Theorem 10.9]{Var84}.  By \cite[Theorems 10.6 and 10.8]{Var84},
the functional $H_{N,K}$ is good and affine.

The classical Donsker-Varadhan theorem, see \cite{dv1-4} or Theorems
11.6 and 12.5 in \cite{Var84}, states that, uniformly on the initial
configuration $\eta\in \Sigma_{N,K}$, the family of probability
measures $(\bb P^N_{\eta}\circ R_T^{-1}: T>0)$ satisfies a large
deviations principle with rate function $H_{N,K}$.

\begin{theorem}
\label{th02}
Fix $N$ and $K$. For each closed set $\ms F$ and each open set $\ms G$ in
$\ms P^{N,K}_{\rm stat}$,
\begin{gather*}
\limsup_{T\to \infty} \sup_{\eta\in \Sigma_{N,K}} \frac 1T \,
\log \bb P^N_\eta \,\big[ R_T \in \ms F\,\big]  
\;\le\; -\, \inf_{\bb Q \in \ms F} H_{N,K}(\bb Q) \;, \\
\liminf_{T\to \infty} \inf_{\eta\in \Sigma_{N,K}} 
\frac 1T \,\log \bb P^N_\eta \,\big[ R_T \in \ms G\,\big]  
\;\ge\; -\, \inf_{\bb Q \in \ms G}  H_{N,K} (\bb Q)\;.
\end{gather*}
\end{theorem}

The rate function $H_{N,K}(\bb Q)$ can also be understood as the
relative entropy \emph{per unit of time} of the stationary probability
$\bb Q$ with respect to the stationary process $\bb P^N_{\mu_{N,K}}$.
Given $T_0<T_1$, denote by
$i_{T_0,T_1}\colon D(\bb R, \Sigma_{N,K}) \to D([T_0,T_1],
\Sigma_{N,K})$ the canonical projection. Given two probability
measures $\bb Q^1,\bb Q^2$ on $D(\bb R, \Sigma_{N,K})$, let
$\bb H_{[T_0,T_1]}$ be the relative entropy between the marginal of
$\bb Q^1$ relative to the time interval $[T_0,T_1]$ and the marginal
of $\bb Q^2$ on the same interval:
\begin{equation}
\label{rem}
\bb H_{[T_0,T_1]} (\bb Q^1 | \bb Q^2 ) \;=\; \textrm{Ent} \big( \bb
Q^1_{[T_0,T_1]} \big| \bb Q^2_{[T_0,T_1]} \big) : = \int \log \frac{d
  \, \bb Q^1_{[T_0,T_1]}} {d\, \bb Q^2_{[T_0,T_1]} } \; d\, \bb
Q^1_{[T_0,T_1]} \;,
\end{equation}
where $\bb Q^j_{[T_0,T_1]} = \bb Q^j \circ i_{T_0,T_1}^{-1}$, $j=1$, $2$.
We also shorthand $\bb H_{[0,T]}$ by $\bb H^{(T)}$.
By \cite[Theorem~5.4.27]{deustr}, for each $\bb Q$ in
$\ms P^{N,K}_{\rm stat}$
\begin{equation}
\label{13}
H_{N,K}(\bb Q) 
\;=\; \lim_{T\to \infty}  \frac 1T\,  \bb H^{(T)} \big(\, \bb Q \, \big |\,
\bb P^N_{\mu_{N,K}}\,  \big)
\;=\;  \sup_{T>0} \frac 1T\,  \bb H^{(T)}\big(\, \bb Q \, \big |\,
\bb P^N_{\mu_{N,K}}\,  \big)\;,
\end{equation}
where the second identity follows by a super-additivity argument which
stems from \cite[Lemma 10.3]{Var84}.  Actually, Theorem~5.4.27 in
\cite{deustr} states that the empirical process $(R_T: T>0)$ satisfies
a large deviations principle with good rate function given by
$H^\star_{N,K}(\bb Q) := \lim_{T\to \infty} T^{-1} \, \bb H^{(T)} (\,
\bb Q \, \big |\, \bb P^N_{\mu_{N,K}}\, )$.  Since, by Theorem
\ref{th02}, $(R_T: T>0)$ also satisfies a large deviations principle
with good rate function given by $H_{N,K}(\bb Q)$, a simple argument
using the lower semi-continuity of the functionals yields that
$H_{N,K}(\bb Q) = H^\star_{N,K}(\bb Q)$ for all
$\bb Q\in \ms P^{N,K}_{\rm stat}$.

Recall that we denote by $\ms S$ the set of trajectories
$(\bs \pi, \bs J)$ which satisfy the continuity equation
\eqref{31}. By \eqref{31bis}, the map
$(\bs \pi_N, \bs J_N)\colon D (\bb R, \Sigma_{N,K})\to D (\bb R, \ms
M)$ takes values in $\ms S$.  As already observed,
$(\bs \pi_N, \bs J_N) (\vartheta_t \bs \eta) = [\vartheta_t (\bs
\pi_N, \bs J_N)] (\bs \eta)$ so that $(\bs \pi_N, \bs J_N)$ induces a
map from the stationary probabilities on $D(\bb R, \Sigma_{N,K})$ to
the stationary probabilities on $\ms S$.  More precisely, if
$\bb P \in \ms P_{\rm stat}^{N,K}$, then
$\bb P \circ (\bs \pi_N , \bs J_N)^{-1}$ belongs to
$\ms P_{\rm stat}$.  Let
$\bs I_{N,K}: \ms P_{\rm stat} \to [0,\infty]$ be defined by
\begin{equation}
\label{12}
\bs I_{N,K} (P) \;=\; \inf \Big\{H_{N,K}(\bb P)\colon\;
\bb P\in \ms P^{N,K}_{\rm stat},\;
\bb P \circ (\bs \pi_N , \bs J_N)^{-1} = P \Big\}\;.
\end{equation}
Note that the set on the right-hand side is either empty (for example,
if the $P$-measure of the set of piece-wise constant paths is not
equal to $1$) or it is a singleton because the map $(\bs \pi_N , \bs
J_N) \colon D (\bb R, \Sigma_{N,K})\to D (\bb R, \ms M)$ is injective.

\begin{corollary}
\label{prop1}
Fix $N$ and $K=0,\ldots, N^d$.  The functional
$\bs I_{N,K} \colon \ms P_{\rm stat} \to [0, +\infty]$ is affine
and good.  Moreover, for each closed $\ms F$ and each open $\ms G$ in
$\ms P_{\rm stat}$,
\begin{gather*}
\limsup_{T\to \infty} \sup_{\eta\in \Sigma_{N,K}} \frac 1T \,
\log \bb P^N_\eta \,\big[\,  \mf R_{T,N} \in \ms F\,\big]  
\;\le\; -\, \inf_{P \in \ms F} \bs I_{N,K}(P) \;, \\
\liminf_{T\to \infty} \inf_{\eta\in \Sigma_{N,K}} 
\frac 1T \,\log \bb P^N_\eta \,\big[ \,  \mf R_{T,N}  \in \ms G\,
\big]  
\;\ge\; -\, \inf_{P \in \ms G}  \bs I_{N,K} (P)\;.
\end{gather*}
\end{corollary}

\begin{proof}
It is enough to show that the map
$\ms P^{N,K}_{\rm stat} \ni \bb Q \mapsto \bb Q \circ (\bs \pi_N, \bs
J_N)^{-1} \in \ms P_{\rm stat}$ is continuous. The statement then
follows from Theorem~\ref{th02} by the contraction principle.

Since the map
$\bs \pi_N \colon D(\bb R;\Sigma_{N,K}) \to D(\bb R;\mc M_+(\bb T^d))$
introduced in \eqref{empden} is continuous, we directly deduce the
continuity of the map
$\ms P^{N,K}_{\rm stat} \ni \bb P \mapsto \bb P \circ \bs \pi_N^{-1}
\in \ms P_{\rm stat}\big(D(\bb R;\mc M_+(\bb T^d)) \big)$.  In
contrast, the map $\bs \eta \mapsto \bs J_N(\bs \eta)$ is not
continuous. Indeed, consider the sequence $\bs \eta^{(k)}$ in which
$\bs \eta^{(k)}$ has a unique jump at time $1/k$ from $\bs \eta(0)$ to
$\sigma^{x,y} \bs \eta(0)$ for some $(x,y) \in \bb B_N$. Then,
$\bs \eta^{(k)}$ converges to the path $\bs\eta$ with a single jump at
time $t=0$ but $\bs J_N(\bs \eta^{(k)})$ does not converge to
$\bs J_N(\bs \eta)$.  In contrast, the map
$\bs \eta \mapsto \bs J_N(\bs \eta)$ is continuous if $\bs\eta$ does
not have a jump at time $t=0$. Moreover, if $\bb Q$ is a stationary
probability on $D\big(\bb R;\Sigma_{N,K}\big)$, then the
$\bb Q$-probability of the paths $\bs \eta$ which have a jump a time
$t=0$ is necessarily zero. This implies that the map
$\ms P^{N,K}_{\rm stat}\ni \bb Q \mapsto \bb Q\circ \bs J_N^{-1} \in
\ms P_{\rm stat} \big(D(\bb R; \mc H^d_{-p}) \big)$ is continuous.
\end{proof}

\section{Variational convergence of the
  Donsker-Varadhan functional.} 
\label{sec03}

Referring to \cite{Br} for an overview, we recall the definition of
$\Gamma$-convergence.  Fix a Polish space $\mc X$ and a sequence
$(U_n : n\in\bb N)$ of functionals on $\mc X$,
$U_n\colon \mc X \to [0,+\infty]$. The sequence $U_n$ is
\emph{equi-coercive} if for each $\ell \ge 0$ there exists a compact
subset $\mc K_\ell$ of $\mc X$ such that
$\{x\in\mc X\colon U_n(x) \le \ell\} \subset \mc K_\ell$ for any
$n \in \bb N$. The sequence $U_n$ \emph{$\Gamma$-converges} to the
functional $U\colon \mc X\to [0,+\infty]$, i.e.\
$U_n \stackrel{\Gamma}{\longrightarrow} U$, if and only if the two
following conditions are met:
\begin{itemize}
\item [(i)]\emph{$\Gamma$-liminf.} The functional $U$ is a
$\Gamma$-liminf for the sequence $U_n$: For each $x\in\mc X$ and each
sequence $x_n\to x$, we have that $\liminf_n U_n(x_n) \ge U(x)$.

\item [(ii)]\emph{$\Gamma$-limsup.} The functional $U$ is a
$\Gamma$-limsup for the sequence $U_n$: For each $x\in\mc X$ there
exists a sequence $x_n\to x$ such that $\limsup_n U_n(x_n) \le U(x)$.
\end{itemize}

Recall the definition of the functionals $\bs I_m$, $\bs I_{N,K}$
introduced in \eqref{06} and \eqref{12}, respectively.  The main
result of this section reads as follows.

\begin{theorem}
\label{th03} 
Fix $0<m<1$, $p>(d+2)/2$, and a sequence $K_N$ so that $K_N/N^d\to m$.
The sequence $(N^{-d} \bs I_{N,K_N} : N\ge 1)$ is equi-coercive.  The
functional $\bs I_m$ is a $\Gamma$-liminf for $N^{-d} \bs I_{N,K_N}$.
If $E$ is orthogonally decomposable, then the functional $\bs I_m$ is
also a $\Gamma$-limsup for $N^{-d} \bs I_{N,K_N}$.  Therefore, under
this hypothesis on $E$,
$N^{-d} \bs I_{N,K_N}\stackrel{\Gamma}{\longrightarrow} \bs I_m$.
\end{theorem}

The proof of this theorem is divided in three parts. In Subsection
4.1, we prove that the sequence $N^{-d} \bs I_{N,K_N}$ is
equi-coercive. In Subsection 4.2, that $\bs I_m$ is a $\Gamma$-liminf,
and, in Subsection 4.3, that $\bs I_m$ is a $\Gamma$-limsup provided
$E$ is decomposable.  For the rest of this section, fix $m \in (0,1)$,
$p>(d+2)/2$, and a sequence $K_N$ such that $K_N/N^d\to m$.

\smallskip\noindent{\bf \ref{sec03}.1 Equi-coercivity.}  Set
$\color{bblue} P_N := \bb P^N_{\mu_{N,K_N}} \circ \big(\bs \pi_N , \bs
J_N\big)^{-1} \in \ms P_{{\rm stat}}$.  We first establish the
exponential tightness of the sequence
$( P_N: N\ge 1) \subset \ms P_{{\rm stat}}$.


\begin{proposition}
\label{t:et}
There exists a sequence $(\ms K_\ell : \ell\ge 1)$ of compact subsets
of $\ms S$ such that
\begin{equation*}
\lim_{\ell\to +\infty} \limsup_{N\to +\infty}
\frac 1{N^d} \log P_N \big(\, \ms K_\ell^\complement \,\big)
\;=\; -\, \infty \;.
\end{equation*}
\end{proposition}

\begin{proof}
In view of Ascoli-Arzel\`a theorem, the compactness of
$\mc M_+(\bb T^d) $, and the compact embedding
$\mc H_{-p} \hookrightarrow \mc H_{-p'}$ for $p' > p$, the assertion
of this proposition follows from the next three lemmata.
\end{proof}

Let
$\color{bblue} D_{T,\delta} := \{(s,t) \in \bb R^2: 0\le s\le t\le
T\,,\, |t-s|\le \delta \}$.

\begin{lemma}
\label{l01}
For each $T>0$, $\epsilon >0$, and smooth $g\colon \bb
T^d\to\bb R$, 
\begin{gather*}
\lim_{\delta\to 0} \limsup_{N\to\infty} \frac{1}{N^d} \log \bb
P^N_{\mu_{N,K_N}} \Big[ \sup_{(s,t) \in D_{T,\delta}} 
\big\vert \, \<\bs \pi_N (t) \,- \,  \bs \pi_N(s)  \,,\, g\> 
\big\vert \; 
>\; \epsilon \Big ]\; =\; -\infty\; . 
\end{gather*}
\end{lemma}

\begin{lemma}
\label{l02}
For each $T>0$ 
\begin{gather*}
\lim_{A\to \infty} \limsup_{N\to\infty} \frac{1}{N^d} \log \bb
P^N_{\mu_{N,K_N}} \Big[ \, \sup_{0\le t\le T} \big\Vert \,
\bs J_N(t) \big\Vert^2_{-p} \; >\; A\, \Big ]\; =\; -\infty\; .
\end{gather*}
\end{lemma}

\begin{lemma}
\label{l05}
For each $\epsilon>0$, $T>0$, and smooth  $H \colon \bb T^d \to \bb R^d$
\begin{equation*}
\lim_{\delta\to 0}  \, \limsup_{N\to\infty} \frac{1}{N^d} \log \bb
P^N_{\mu_{N,K_N}} \Big[ \, \sup_{(s,t)\in D_{T,\delta}}
\big\vert \, \<\bs J_N (t) \,-\,\bs J_N (s)\,,\,  H\> \, \big\vert\;
\; >\; \epsilon \, \Big]\; =\; -\infty\; .
\end{equation*}
\end{lemma}

Lemma~\ref{l01} is a standard result in the large deviations theory of
hydrodynamical limits, see e.g.\ \cite[\S~10.4]{kl}. Note that this
result can be deduced from Lemma~\ref{l05} by taking $H = \nabla g $
and using the continuity equation \eqref{31bis}. On the other hand,
the exponential tightness of the empirical current is stated in
\cite{bdgjl07} but the proof presented there is incomplete.  For this
reason, we present below a detailed proof of Lemmata \ref{l02} and
\ref{l05}.

\begin{proof}[Proof of Lemma~\ref{l02}]
We use the notation and statements introduced in
\eqref{ap01}-\eqref{ap03}, and denote $\bs J_N(t) ( {H}^{j,n} )$ by
$\< \bs J_N(t) \,,\, {H}^{j,n} \>$. By \eqref{ap03},
\begin{equation*}
\big\Vert \, \bs J_N(t) \big\Vert^2_{-p}
\;=\; \sum_{j=1}^d \sum_{n\in\bb Z^d} \frac 1{\gamma^p_n}\; 
\big|\, \< \bs J_N(t) \,,\, {H}^{j,n} \> \,\big|^2\;,\qquad
\gamma_n \,:=\,  1 + 4\pi^2 |n|^2 \; .
\end{equation*}
Let $\beta_n = \gamma^p_n / [c_p\, (1+|n|)^{2(p -1)}]$. Here, $c_p$ is
a constant such that $d\, \sum_n (\beta_n/\gamma^p_n)=1$, that
is, $c_p = d \sum_n (1+|n|)^{-2(p-1)}$.  Note that this sum is finite
because we assumed $p>1+(d/2)$. Introducing the supremum inside the
sum yields that
\begin{align*}
& \bb P^N_{\mu_{N,K_N}} \Big[ \, \sup_{0\le t\le T}
\sum_{j=1}^d \sum_{n\in\bb Z^d}
\frac{1}{\gamma^p_n} \,\big|\, \< \bs J_N(t) \,,\, {H}^{j,n} \> \,\big|^2
\; >\; A\, \Big ] \\
&\quad \le\; 
\bb P^N_{\mu_{N,K_N}} \Big[ \, \bigcup_{j=1}^d \bigcup_{n\in\bb Z^d}
\Big\{ \frac 1{\beta_n} \, \sup_{0\le t\le T}
\big|\, \< \bs J_N(t) \,,\, {H}^{j,n} \> \,\big|^2
\; >\; A\, \Big\}\, \Big ] \;.
\end{align*}

Fix $1\le j\le d$, $n\in \bb Z^d$ and denote by ${H}^{j,n}_{1}$,
${H}^{j,n}_{-1}$ the real and the imaginary part of ${H}^{j,n}$,
respectively.  The previous expression is then bounded by
\begin{equation*}
\sum_{b=\pm 1} \sum_{j=1}^d \sum_{n\in\bb Z^d}
\bb P^N_{\mu_{N,K_N}} \Big[ \, \sup_{0\le t\le T}
\big|\, \< \bs J_N(t) \,,\, {H}^{j,n}_b \> \,\big|
\; >\; \sqrt{A\, \beta_n/2} \, \Big ] \;.
\end{equation*}
We may remove the absolute value from the previous expression at the
cost of an extra factor $2$ in front of the sum and an estimation of
${H}^{j,n}_b$ and $-\, {H}^{j,n}_b$.  We next bound the probability of
the event
$\{\, \sup_{0\le t\le T} \< \bs J_N(t) \,,\, {H}^{j,n}_b \> \, \, >\,
\sqrt{A\, \beta_n/2} \, \}$, the other one being similar.

Recall the notation introduced in \eqref{dvf} and let
$B_{x,x+\mf e_k} (\eta) = B_{x,x+\mf e_k}(\eta, {H}^{j,n}_b)$,
$1\le k \le d$, $x\in \bb T^d_N$, be given by
\begin{align*}
B_{x,x+\mf e_k} (\eta) \; & =\;
N^2 \, \eta_x  \, [\, 1 - \eta_{x+ \mf e_k} \,] \,  
e^{(1/2) \, E_N(x,x+\mf e_k)}\,
\Big[ e^{{H}^{j,n}_{b,N} (x,x+\mf e_k)} \,-\,1\Big] \\
& +\; N^2 \, \eta_{x+\mf e_k}  \, [\, 1 - \eta_{x} \,] \,  
e^{(1/2) \, E_N(x+\mf e_k,x)}\,
\Big[ e^{{H}^{j,n}_{b,N} (x+\mf e_k,x)} \,-\,1 \Big]\;,
\end{align*}
By \cite[Proposition A1.2.6]{kl}, for each $\eta\in\Sigma_N$ the
process
\begin{align*}
\bb M_N(t) \;:=\; \exp\Big\{ N^d\, \< \bs J_N(t) \,,\,
{H}^{j,n}_b \>  \;-\;
\int_0^t \sum_{k=1}^d \sum_{x\in \bb T_N^d} B_{x,x+e_k} (\bs \eta(s) )\;
ds\, \Big\}
\end{align*}
is a mean-one $\bb P^N_\eta$-martingale.

Since $ N \big| {H}^{j,n}_{b,N}(x,x+ \mf e_k)\big|$ is bounded
uniformly in $b,j,k,x,n,N$, a Taylor expansion yields
\begin{align*}
\bigg|  \, \sum_{k=1}^d \sum_{x\in \bb T_N^d}  B_{x,x+e_k} (\eta) \;-\; N^2 
\sum_{k=1}^d \sum_{x\in \bb T^d_N} {H}^{j,n}_{b,N} (x,x+\mf e_k)
\,( \eta_x \,-\,
\eta_{x+\mf e_k}\,) \bigg| \;\le \; C_1\, N^d\;,
\end{align*}
for some constant $C_1$ independent of $b,j,n,N$.  Summing by parts
and using that  
$\big| \partial_{x_k} {H}^{j,n}_b (x) \big|\le C_2 |n| $ for some
constant $C_2$ independent of $b,j,k,x,n$, we conclude that
\begin{equation*}
\Big| \,\sum_{k=1}^d \sum_{x\in \bb T^d_N}
B_{x,x+e_k} (\eta, {H}^{j,n}_b)\, \Big|
\;\le\;  C_0 N^d (1 +|n|)
\end{equation*}
for some constant $C_0$ independent of $b,j,n,N$.

In view of the previous estimate, adding and subtracting the sum of
the time-integrals of $B_{x,x+e_k}$ and taking exponentials, we get that
\begin{align*}
& \bb P^N_{\mu_{N,K_N}} \Big[ \, \sup_{0\le t\le T}
\< \bs J_N(t) \,,\, H^{j,n}_b \> 
\; >\; \sqrt{A\, \beta_n/2} \; \Big ] \\
& \qquad \;\le\;
\bb P^N_{\mu_{N,K_N}} \Big[ \, \sup_{0\le t\le T}
\bb M_N(t) \; >\; e^{\sqrt{A\, \beta_n/2} N^d - C_0 T N^d (1 +|n|) } \, \Big ]
\\
&\qquad
\;\le\; e^{- \sqrt{A\, \beta_n/2} N^d  + C_0 T N^d (1 +|n|)} \;,
\end{align*}
where we used Doob's inequality in the last step and the fact that
$\bb M_N(t)$ is a mean-one martingale.

We have thus shown that
\begin{align*}
\bb P^N_{\mu_{N,K_N}} \Big[ \, \sup_{0\le t\le T}
\big\| \bs J_N(t)\|^2_{\mc H^d_{-p}}
\; >\; A\, \Big ] \;\le\; 4 \, d\, 
\sum_{n\in\bb Z^d} e^{- N^d \big[ \sqrt{A\, \beta_n /2 } 
- C_0 T (1+|n|)  \big] } \;. 
\end{align*}
By definition of $\beta_n$, there exists a positive constant
$c_0$ such that $\beta_n \ge c_0 (1+|n|)^2$. The statement follows.
\end{proof}

\begin{proof}[Proof of Lemma~\ref{l05}]
By a standard inclusion of events and the stationarity of
$P^N_{\mu_{N,K_N}}$, it is enough to prove that for each $\epsilon>0$
and smooth $H \colon \bb T^d \to \bb R^d$
\begin{equation}
\label{el05}
\lim_{\delta\to 0}  \, \limsup_{N\to\infty} \frac{1}{N^d}
\log \bb P^N_{\mu_{N,K_N}} \Big[ \, \sup_{0\le t\le \delta}
\big\vert \, \<\bs J_N (t) \,,\, H\> \big\vert\; \; >\; \epsilon \, 
\Big ]\; =\; -\, \infty\;,
\end{equation}
where we used that $\bs J_N (0)=0$.  As in the proof of the previous
lemma, we can furthermore remove the modulus in the above bound.

Given a smooth vector-valued function $H\colon \bb T^d\to \bb R^d$ and
$\ell>0$, let
$B_{x,x+\mf e_k}^\ell (\eta) = B_{x,x+\mf e_k}(\eta,\ell H)$,
$x\in \bb T^d_N$, be given by
\begin{align*}
B_{x,x+\mf e_k}^\ell (\eta) \; & =\; N^2 \, \eta_x \,
(\, 1 - \eta_{x+ \mf e_k} \,) \, e^{(1/2) \,
  E_N(x,x+\mf e_k)}\, \Big[ e^{{\ell  H}_{N} (x,x+\mf e_k)} \,-\,
1\Big] \\
& +\; N^2 \, \eta_{x+\mf e_k} \, [\, 1 - \eta_{x} \,] \, e^{(1/2) \,
E_N(x+\mf e_k,x)}\, \Big[ e^{{\ell H}_{N} (x+\mf e_k,x)} \,-\,
1\Big]\;.
\end{align*}
By \cite[App.~1, Prop.~2.6]{kl}, for each $\eta\in\Sigma_N$, the
process
\begin{align*}
\bb M_N^\ell(t) \;:=\; \exp\Big\{ N^d\, \ell \,\< \bs J_N(t) \,,\, H
\> \;-\; \int_0^t \sum_{k=1}^d \sum_{x\in \bb T_N^d} B_{x,x+e_k}^\ell
(\bs \eta(s) )\; ds\, \Big\}
\end{align*}
is a mean-one $\bb P^N_\eta$-martingale.

The same computation of the previous lemma yields that there exists a
constant $C_1=C_1(H)$ such that
\begin{equation*}
\Big| \, \sum_{k=1}^d \sum_{x\in \bb T_N^d} B_{x,x+\mf e_k}^\ell (\eta)
\,\Big| \;\le\; C_1 \, N^d e^{ C_1 \ell/N} (1+\ell^2)
\end{equation*}
for all $N$, $\ell$ and $\eta$.  Therefore, by Doob's inequality,
\begin{equation*}
\begin{split}
& \bb P^N_{\mu_{N,K_N}} \Big[ \, \sup_{0\le t\le \delta} \<\bs J_N (t)
\,,\, H\> \; >\; \epsilon \, \Big ] = \bb P^N_{\mu_{N,K_N}} \Big[ \,
\sup_{0\le t\le \delta} N^d \, \ell \, \<\bs J_N (t) \,,\, H\> \; >\;
N^d\, \ell\, \epsilon \, \Big ]
\\
&\quad \le P^N_{\mu_{N,K_N}} \Big[ \, \sup_{0\le t\le \delta} \bb
M_N^\ell(t) \; >\; \exp\big\{ N^d \big[ \ell \,\epsilon - \delta
\,C_1\, e^{ C_1 \ell/N} (1+\ell^2) \big] \big\} \Big ]
\\
&\quad \le \exp\Big\{ - N^d \big[ \ell \,\epsilon - \delta \, C_1
\,e^{ C_1 \ell/N} (1+\ell^2) \big]\Big\}\;,
\end{split}
\end{equation*}
which yields \eqref{el05} by taking the limit $\ell \to \infty$ after
the limits in $N$ and $\delta$.
 \end{proof}

\begin{proof}[Proof of Theorem \ref{th03}. Equi-coercivity.]
For $\ell\ge 1$, let
\begin{equation*}
\ms E_\ell \;:=\; \bigcup_{N\ge 1}
\Big\{P \in \ms P_{\rm stat} \colon
\frac 1{N^d} \bs I_{N,K_N} (P) \;\le\; \ell \Big\}\;.
\end{equation*}

In view of Ascoli-Arzel\`a theorem, the compactness of
$\mc M_+(\bb T^d) $, and the compact embedding
$\mc H_{-p} \hookrightarrow \mc H_{-p'}$ for $p' > p$,
to show that the set $\ms E_\ell$ is pre-compact, it is enough to
prove that for each $p> (d+2)/2$, $\epsilon>0$, $T>0$, and smooth
functions $g\colon \bb T^d \to \bb R$, $H\colon \bb T^d\to \bb R^d$,
\begin{equation}
\label{30}
\begin{gathered}
\lim_{\delta\to 0} \, \sup_{P\in \ms E_\ell} P\Big[ \,
\sup_{(s,t) \in D_{T,\delta}} \, 
\big\vert \, \<\bs \pi (t) \,-\, \bs \pi(s) \,,\, g\>
\big\vert \; 
>\; \epsilon \, \Big ]\; =\; 0 \; , \\
\lim_{A\to \infty} \, \sup_{P\in \ms E_\ell} 
P \Big[ \sup_{0 \le t\le T} \big\Vert \,
\bs J(t) \, \big\Vert_{-p} \; >\; A \, \Big ]\; =\; 0\; ,
\\
\lim_{\delta\to 0}  \, \sup_{P\in \ms E_\ell} 
P \Big[ \sup_{(s,t) \in D_{T,\delta}} \big\vert \,
\<\bs J(t) \,-\, \bs J(s)  \,,\, H\> \big\vert\; 
>\; \epsilon \, \Big ]\; =\; 0\; ,
\end{gathered}
\end{equation}
where
$D_{T,\delta}$ has been introduced before the statement of Lemma \ref{l01}.

To prove the first assertion in \eqref{30}, fix $\epsilon>0$, $T>0$
and a smooth function $g\colon \bb T^d \to \bb R$. For $\delta>0$, let
\begin{equation*}
\ms B \;=\; \ms B^{\delta, \epsilon}_{g,T} 
\;:=\; \Big \{ (\bs \pi, \bs J) \in \ms S \colon 
\sup_{(s,t) \in D_{T,\delta}} \, 
\big\vert \, \<\bs \pi (t)  \,-\,\bs \pi(s)
\,,\, g\>
\big\vert \; 
>\; \epsilon \, \Big\}\;.
\end{equation*}
Fix $P\in \ms E_\ell$. By definition of the set $\ms E_\ell$, there
exists $N\ge 1$ such that $\bs I_{N,K_N}(P) \le \ell \,
N^d$. Furthermore, by definition \eqref{12} of the rate function
$\bs I_{N,K_N}$, $P= \bb P \circ (\bs \pi_N , \bs J_N)^{-1}$ for some
$\bb P \in \ms P_{\rm stat}^{N,K}$, and
$\bs I_{N,K_N}(P) = H_{N,K_N}(\bb P)$.
  
Since the set $\ms B$ is measurable with respect to
$\sigma\big\{ (\bs\pi(t),\bs J(t)), \, t\in[0,T]\big\}$, by the
definition \eqref{rem} of the relative entropy $\bb H^{(T)}$ and by
the entropy inequality (see e.g.\ \cite[Proposition A1.8.2]{kl}),
\begin{equation*}
P[\ms B] \;=\; \bb P \big[ \, (\bs \pi_N, \bs J_N) \in \ms B\, \big]
\;\le\; \frac{\log 2 \,+\, \bb H^{(T)}(\bb P\,|\, \bb P^N_{\mu_{N,K_N}})}
{\log \Big( 1 + 
\big( \bb P^N_{\mu_{N,K_N}} \big[ \, (\bs \pi_N, \bs J_N) 
\in \ms B\, \big]\big)^{-1} \Big)}\;\cdot
\end{equation*}
By \eqref{13}, and since $H_{N,K_N}(\bb P) = \bs I_{N,K_N}(P) \le \ell\,
N^d$, 
\begin{equation*}
P[\ms B] \;\le\;  \frac{\log 2 \,+\, T\,\ell\, N^d}
{\log \Big( 1 + \big(\bb P^N_{\mu_{N,K_N}} \big[ \, (\bs \pi_N, \bs J_N) 
\in \ms B\, \big]\big)^{-1} \Big)}\;\cdot
\end{equation*}
This bound is uniform over $P\in \ms E_\ell$ provided we take the
supremum over $N\ge 1$ on the right-hand side.

Fix $a>0$, and let $\gamma = (\log 2 + T\, \ell)/a$.  By Lemma
\ref{l01}, there exists $\delta_0=\delta_0(T,g,\epsilon, \gamma)$ and
$N_0=N_0(T,g,\epsilon, \gamma)$ such that
\begin{equation*}
\bb P^N_{\mu_{N,K_N}} \big[ \, (\bs \pi_N, \bs J_N) \in \ms B^{\delta_0, \epsilon}_{g,T}\, \big]
\;\le\; e^{-\, \gamma\, N^d}
\end{equation*}
for all $N\ge N_0$. By changing the value of $\delta_0$ we may extend
this inequality to all $N\ge 1$. In particular, by definition of
$\gamma$, 
\begin{equation*}
\sup_{P\in \ms E_\ell} P\big[\, \ms B^{\delta_0, \epsilon}_{g,T}
\, \big] \;\le\;  \sup_{N\ge 1}
\frac{\log 2 \,+\, T\,\ell\, N^d}
{\gamma \, N^d}\;\le\; a \;.
\end{equation*}
As
$\ms B^{\delta, \epsilon}_{g,T} \subset \ms B^{\delta_0,
  \epsilon}_{g,T}$ for $0<\delta\le \delta_0$, the previous inequality
holds for all $0<\delta \le \delta_0$. Since $a>0$ is arbitrary, this
proves the first assertion of \eqref{30}.

The second and third assertions in \eqref{30} are proven similarly,
based on Lemmata \ref{l02} and \ref{l05}.
\end{proof}

\smallskip\noindent{\ref{sec03}.2 \bf The $\Gamma$-liminf.}  Let
$(P_N\colon N\ge 1) $ be a sequence of probability measures in
$\ms P_{\rm stat}$ such that
$\liminf_N N^{-d} \bs I_{N,K_N} (P_N) < \infty$.
The following lemma lists properties of the cluster points of these
sequences.

\begin{lemma}
\label{lem3}
Let $(P_N\colon N\ge 1) $ be a sequence of probability measures in
$\ms P_{\rm stat}$ such that
$\liminf_{N} N^{-d} \bs I_{N,K_N} (P_N)< +\infty$. Assume that
$P_N\to P$ for some $P \in \ms P_{\rm stat}$.  Then, $P$-almost surely
$(\bs \pi, \bs J)$ belongs to $\ms S_{m, \mathrm{ac}}$ and there
exists a constant $C_0$ such that for all $T>0$,
\begin{equation*}
\begin{split}
& E_{P} \Big[\, \int_0^T dt\, \int_{\bb T^d} dx \,
\frac{|\bs j(t,x)|^2}{\sigma(\bs \rho(t,x))} 
\;+\;  \int_0^T dt\, \int_{\bb T^d} dx \,
\frac{|\nabla \bs \rho(t,x)|^2}{\sigma(\bs \rho(t,x))} \,\Big] 
\\
& \qquad \;\le\;  C_0\, (1\,+\, T)  \,+\, 2\, T\, 
\liminf_{N} \frac 1 {N^{d}} \bs I_{N,K_N} (P_N) \;,
\end{split}
\end{equation*}
where $\bs\pi(t,dx) = \bs \rho(t,x)\, dx$ and
$\bs J (t) \;=\; \int_0^t \bs j(s)\, ds $.
\end{lemma}

\begin{proof}
  By passing to a subsequence, we may assume $\liminf_{N} N^{-d} \bs
  I_{N,K_N} (P_N) =\lim_{N} N^{-d} \bs I_{N,K_N} (P_N)$.  By
  \eqref{12} and \eqref{13}, for every $T>0$,
\begin{equation*}
\bs I_{N,K_N}(P_N) \;\ge\; \frac 1T \, \bb H^{(T)} \big(\, \bb Q_N \, \big |\,
\bb P^N_{\mu_{N,K_N}}\,  \big)\;,
\end{equation*}
where $\bb Q_N$ is the unique stationary probability on $D(\bb R,
\Sigma_{N,K})$ such that $\bb Q_N \circ (\bs \pi_N , \bs J_N)^{-1} =
P_N$. In particular, for every $T>0$,
\begin{equation}
\label{24}
\limsup_N \frac 1{N^d} \, \bb H^{(T)} \big(\, \bb Q_N \, \big |\,
\bb P^N_{\mu_{N,K_N}}\,  \big)  \;\le \;  T\, 
\lim_{N} \frac 1{N^{d}} \bs I_{N,K_N} (P_N) \;.
\end{equation}
By this bound, the marginal of $\bb Q_N$ in the time interval
$[0,T]$ is absolutely continuous with respect to the marginal of
$\bb P^N_{\mu_{N,K_N}}$ in the same interval.
Moreover, for each
continuous function $g \colon \bb R \times \bb T^d \to \bb R$ with support
on $(0,T) \times \bb T^d$, $P_N$-almost surely,
\begin{equation*}
\Big| \, \int_{\bb R} dt \int_{\bb T^d} \bs \pi (t,dx)
\, g(t,x)\, \Big| \;\le\;
\frac 1{N^d} \sum_{x\in \bb T^d_N} \int_{\bb R} dt\, \big|\,  g(t,x)\, \big|
\end{equation*}
because there is at most one particle per site.  
Since the left-hand side is a continuous function of $\bs\pi$ in the
Skorohod topology, taking the limit $N\to\infty$ we deduce that
$P$-almost surely 
\begin{equation*}
\Big| \, \int_{\bb R} dt \int_{\bb T^d} \bs \pi (t,dx)
\, g(t, x)\, \Big| \;\le\;
\int_{\bb R} dt\, \int_{\bb T^d} \big|\,  g(t,x)\, \big|\, dx\;.
\end{equation*}
This implies that $P$-almost surely, for Lebesgue almost all $t$, the
measure $\bs \pi (t,dx)$ is absolutely continuous with respect to the
Lebesgue measure: $\bs \pi (t,dx) = \bs \rho(t,x) \, dx$ for some
density $\bs \rho$ satisfying $0\le \bs \rho \le 1$.

On the other hand, as $\bb Q_N$ is absolutely continuous with respect
to $\bb P^N_{\mu_{N,K_N}}$, $P_N [ \bs \pi(t,\bb T^d) = K_N/N^d] = 1$
for all $t\in \bb R$. As $K_N/N^d \to m$ and $P_N \to P$,
$P [ \bs \pi(t,\bb T^d) = m] = 1$ for Lebesgue almost all $t\in \bb R$.

For a vector field $F$ in $C^1(\bb R \times \bb T^d; \bb R^d)$ with
compact support in $(0,T) \times \bb T^d$, $\epsilon>0$ and $a>0$, let
\begin{equation}
  \label{25}
  \begin{split}
    \mc E_{a, \epsilon} (F, \bs \pi) &\;=\; \int_{\bb R}  \!dt\,
\big\langle \bs \pi(t)\,,\, \nabla\cdot F(t) \big\rangle
\;-\; a\, \int_{\bb R}  dt
\int_{\bb T^d} dx \; \sigma(\bs \pi^\epsilon) \; |F|^2  \;,
\\
\mc V_{a,\epsilon} (F,\bs \pi,\bs J) &\;=\; \bs J (F) \;-\; a\,
\int_{\bb R} \!dt\, \int_{\bb T^d} dx \, \sigma(\bs \pi^\epsilon) \,
|F|^2\;.
\end{split}
\end{equation}
where
$\color{bblue} (\bs \pi^\epsilon) (t,x) = (2\epsilon)^{-d} \bs \pi (t,
[x-\epsilon , x+\epsilon]^d)$ and
$\bs J (F) = - \int_{\bb R} \!dt\, \langle \bs J(t), \partial_t F
\rangle$.

Let $(F_j : j\ge 1)$ be a family of vector fields in
$C^{1}((0,T) \times \bb T^d; \bb R^d)$ with compact support and dense
in $L^2([0,T] \times \bb T^d; \bb R^d)$. Assume that $F_1=0$.  In view
of Lemma~\ref{lem04}, the entropy bound \eqref{24},
and a classical argument which allows to bound a maximum over a finite
set in exponential estimates, there exist finite constants $a$ and
$C_0$ such that for all $k\ge 1$,
\begin{gather*}
\limsup_{\epsilon \to 0} \limsup_{N\to\infty} 
E_{P_N} \Big[ \max_{1\le j\le k} 
\, \mc E_{a, \epsilon} (F_j, \bs \pi) \, \Big] \;\le\; A
\;, \\
\limsup_{\epsilon \to 0} \limsup_{N\to\infty} 
E_{P_N} \Big[ \max_{1\le j\le k}
\mc V_{a, \epsilon} (F_j, \bs \pi,\bs J) \, \Big] \;\le\; A
\;,
\end{gather*}
where
\begin{equation*}
A\;:=\;  C_0 \, (1+T) \,+\, T \,
\lim_{N} \frac 1{N^{d}} \bs I_{N,K_N} (P_N) \;.
\end{equation*}

Since $P_N$ converges to $P$ that is concentrated on measures
which are absolutely continuous with respect to the Lebesgue measure, 
$\bs \pi (t,dx) = \bs \rho(t,x)\, dx$, and whose density $\bs \rho$ is
bounded below by $0$ and above by $1$, taking the limit in $N$ and
$\epsilon$ yields
\begin{gather*}
E_{P} \Big[ \max_{1\le j\le k} \Big\{\, \bs J(F_j) \;-\;
a \int_{\bb R}  dt \int_{\bb T^d} dx \, \sigma(\bs \rho) \,
|F_j|^2 \Big\}\, \Big] \;\le\; A  \;, \\
E_{P} \Big[ \max_{1\le j\le k} \Big\{\, \int_{\bb R}  \!dt 
\int_{\bb T^d} \! dx\,
\Big[
\bs \rho \,\, \nabla \cdot F_j \;-\;
a \, \sigma(\bs \rho) \,
|F_j|^2 \Big]\, \Big\}\, \Big] \;\le\; A \;.
\end{gather*}
Each maximum is positive because $F_1=0$. By monotone convergence,
taking the limit in $k$, we obtain a similar bound, where the maximum
over $1\le j\le k$ is replaced by the maximum over $j\ge 1$.  Since
the sequence $F_j$ is dense in $L^2([0,T] \times \bb T^d; \bb R^d)$,
by Riesz representation theorem, $P$-almost surely,
$\bs J(t) = \int_0^t \bs j(s) \, ds$ for some $\bs j$ in
$L^2([0,T] \times \bb T^d, \sigma(\bs \rho)^{-1} dt\, dx; \bb
R^d)$. These arguments also yield the bounds stated in the lemma.

We turn to the proof that $P$-almost surely
$\bs\pi \in C\big(\bb R; \mc M_m(\bb T^d)\big)$.  By the continuity
equation, $P$-almost surely for all functions $g$ in $C^1(\bb T^d)$
and $0\le s<t\le T$,
\begin{equation*}
\int_{\bb T^d} \bs \pi (t,dx) \, g(x) \;-\; 
\int_{\bb T^d} \bs \pi (s,dx) \, g(x) \;=\;
\int_s^t dr \int_{\bb T^d} dx\, \bs j (r,x) \cdot (\nabla g)(x)\;.
\end{equation*}
Since $\bs j$ belongs to $L^2([0,T] \times \bb T^d, \sigma(\bs
\rho)^{-1} dt\, dx; \bb R^d)$, $P$-almost surely $\bs\pi$ belongs to
$C([0,T], \mc M_m (\bb T^d))$, as claimed.
\end{proof}

Fix $T>0$ and a continuous vector field
$w \colon [0,T] \times \bb T^d \times \mc M_{+}(\bb T^d) \times D( \bb
R; \mc H_{-p}^d) \to \bb R^d$ that is continuously differentiable in
$x$ and such that for each
$(x,\pi) \in \bb T^d \times \mc M_{+}(\bb T^d)$ and $t\in[0,T]$ the
map
$[0,t]\times D( \bb R; \mc H_{-p}^d) \ni (s,\bs J) \to w(s,x, \pi, \bs
J)$ is measurable with respect to the Borel $\sigma$-algebra on
$[0,t]\times D([0,t];\mc H^d_{-p}\big)$.  Let
$G_w\colon \bb R \times \bb T^d \times \ms S_{m, {\rm ac}} \to \bb
R^d$ be the progressively measurable map defined by
\begin{equation}
\label{23}
G_w (t,x, \bs \pi, \bs J) \;=\; w (t, x,\bs \pi(t), \bs J)\;.
\end{equation}
Finally, for $(\bs \pi, \bs J) \in \ms S_{m, {\rm ac}}$, let 
\begin{equation}
\label{21}
V_{T,w} (\bs \pi,\bs J) \;=\; \frac 1{T} \int_0^T dt \int_{\bb T^d}
dx \, 
\Big\{  \,  G_w \cdot \big[\, \bs j \,+ \, \nabla \bs \rho \,  - \,
\sigma(\bs\rho) \, E \, \big]
\,  -\,   \sigma(\bs \rho ) \, |G_w|^2 \,  \Big\}\;,
\end{equation}
where  $\bs\pi(t,dx) = \bs \rho(t,x)\, dx$, $\bs J (t) \;=\;
\int_0^t \bs j(s)\, ds$, $t\in\bb R$.

\begin{lemma}
\label{lem2}
Let $(P_N\colon N\ge 1) $ be a sequence of probability measures in
$\ms P_{\rm stat}$ such that
$\liminf_{N} N^{-d} \bs I_{N,K_N} (P_N)< +\infty$. Assume that
$P_N\to P$ for some $P \in \ms P_{\rm stat}$.  Then, for each $T>0$
and each function $w$ as above,
\begin{align}
\label{e:lem2}
\liminf_{N\to\infty} \frac 1{N^d}\, \bs I_{N,K_N} (P_N) \;
\ge\; E_P \big[ \, V_{T,w} \, \big] \;.
\end{align}
\end{lemma}

By Lemma~\ref{lem3}, $P$--almost surely $(\bs \pi,\bs J)$ belong to
$\ms S_{m,\mathrm{ac}}$ so that the right hand side of \eqref{e:lem2}
is well defined.
  
\begin{proof}[Proof of Lemma \ref{lem2}]
By passing to a subsequence, if needed, we may assume that
$\liminf_N N^{-d} \bs I_{N,K_N} (P_N) = \lim_{N} N^{-d}\, \bs
I_{N,K_N} (P_N) $.  By definition of $\bb H_{N,K} (t, \bb Q)$,
\eqref{12} and \eqref{20}, for each $T>0$ and each bounded, continuous
functions $\Phi$ on $D(\bb R, \Sigma_{N,K_N})$, measurable with
respect to $\sigma\big\{ \bs \eta (t), \, t\in[0,T] \big\}$,
\begin{equation*}
\bs I_{N,K_N}(P_N) \;\ge\; \frac 1T \, \int d \bb Q_N (\bs \eta)
\Big\{ \Phi (\bs \eta) \,-\, \log \bb E^N_{\bs \eta (0)} \big[e^\Phi
\big] \, \Big\} \;,
\end{equation*}
where $\bb Q_N$ is the unique stationary probability on $D(\bb R,
\Sigma_{N,K_N})$ such that $\bb Q_N \circ (\bs \pi_N , \bs J_N)^{-1} =
P_N$. 

Recalling \eqref{27}, let $\Phi$ be given by
\begin{equation*}
\Phi (\bs \eta) = 
\sum_{(x,y)\in \bb B_N}\Big\{  \int_0^T
\phi^{x,y}(t)\,  \ms N^{x,y}_{(t,t+dt]} (\bs\eta) \,-\, N^2 \int_0^T
\bs \eta_x(t) [1-\bs\eta_y(t)] \big[e^{\phi^{x,y}(t)} - 1\big] \, dt
\Big\}\;, 
\end{equation*}
where 
\begin{equation*}
\phi^{x,y}(t) \;=\; \int_x^y w(t, \cdot \, , \bs \pi_N(t),
\bs J_N ) \, \cdot  d\ell\;.
\end{equation*}
By Lemma \ref{lem01}, $\bb E_{\eta}^N \big[e^\Phi\big] =1$ for each
$\eta\in \Sigma_{N,K}$ and by Lemma \ref{lem02},
\begin{equation*}
\frac 1{T}\, \lim_{N\to\infty} 
\frac 1{N^d}\, E_{\bb Q_N} [\Phi] \;=\; E_P \big[ \,
V_{T,w} \, \big] \;, 
\end{equation*}
which completes the proof. 
\end{proof}

\begin{proof}[Proof of Theorem \ref{th03}. $\Gamma$-liminf]
Fix $P$ in $\ms P_{\rm stat}$ and a sequence $(P_N : N\ge 1)$,
$P_N\in \ms P_{\rm stat}$, such that $P_N\to P$,
$\liminf_{N} N^{-d} \bs I_{N,K_N} (P_N)< +\infty$.  By Lemma
\ref{lem3}, we may assume that $P$-almost surely $(\bs \pi, \bs J)$
belongs to $\ms S_{m, {\rm ac}}$ and that there exists a constant
$C_1$ such that for all $T>0$,
\begin{equation}
\label{07}
E_{P} \Big[\, \int_0^T dt\, \int_{\bb T^d} dx \,
\frac{|\bs j(t,x)|^2}{\sigma(\bs \rho(t,x))} 
\;+\;  \int_0^T dt\, \int_{\bb T^d} dx \,
\frac{|\nabla \bs \rho(t,x)|^2}{\sigma(\bs \rho(t,x))} \,\Big] 
\;\le\;  C_1\, (1+T)\;,
\end{equation}
where $\bs\pi(t,dx) = \bs \rho(t,x)\, dx$, $\bs J (t) \;=\;
\int_0^t \bs j(s)\, ds$.

By Lemma \ref{lem2}, it is enough to show that for some $T>0$,
\begin{align}
\label{08}
\bs I_m (P)\;\le\; \sup_w E_P \big[\, V_{T,w} \big] \;,
\end{align}
where the supremum is carried over all continuous vector fields
$w \colon [0,T] \times \bb T^d \times \mc M_{+}(\bb T^d) \times D( \bb
R; \mc H_{-p}^d) \to \bb R^d$ satisfying the assumption enunciated
above \eqref{23}.

Fix $T>0$ and a bounded and continuous function
$f \colon \bb T^d \times \mc H_1 \times L^2 (\bb T^d ; \bb R^d) \to
\bb R^d$ that is continuously differentiable in $x$.  For $\delta>0$,
let
$f_\delta \colon (0,T) \times \bb T^d \times \mc M_{+}(\bb T^d) \times
D \big(\bb R; \mc H^d_{-p} \big) \to \bb R^d$ be given by
\begin{equation*}
f_\delta (t,x, \pi , \bs J) \;=\; \chi_{\delta}(t) \,
f\big(x, \pi_\delta , \bs j_\delta (t)\big) \,.
\end{equation*} 
Here $\chi_{\delta}$, $0<\delta<1$, stands for a sequence of
continuous functions, bounded below by $0$ and above by $1$, whose
support is contained in $[\delta, T]$, and which converges in $L^1$ to
the indicator functions of the set $[0,T]$.
Moreover, $\pi_\delta(x) = \langle\pi, \kappa_\delta(x-\cdot)\rangle$
where $\kappa_\delta\colon \bb T^d\to \bb R_+$ is a smooth
approximation of the identity, and
\begin{equation}
\label{29}
\bs j_\delta (t,x) \;=\; \int_{-\infty}^t \!ds\; a'_\delta(t-s) \, \langle
\bs J(s), \imath_\delta (x-\cdot) \rangle\,,
\end{equation}
where $a_\delta\colon \bb R\to \bb R_+$ is a smooth approximation of
identity with compact support in $(0,\delta)$, $a'_\delta$ the
derivative of $a_\delta$, and $\imath_\delta\colon \bb T^d\to \bb R^d$
another smooth approximation of the identity. 
Observe that $\bs j_\delta(t)$ depends on $\bs J(s)$ only for
$s\in (t-\delta,t)$. Hence, since $\chi_{\delta}(t)=0$ for
$t\le \delta$, the function $f_\delta(t,x,\pi,\cdot)$ depends on
$\bs J(s)$ only for $s \in [0,t]$.  This is a requirement of the test
functions $w$ introduced above \eqref{23}.  Since $f_\delta$ satisfies
the conditions presented above \eqref{23} for each $\delta>0$, we
deduce
\begin{equation*}
\sup_w E_P \big[\, V_{T,w} \big] \;\ge\; \limsup_{\delta\to 0}
E_P \big[\, V_{T,f_\delta} \big]\;.
\end{equation*}

Let
$H_f\colon [0,T]\times \bb T^d \times L^2 \big([0,T] ; \mc H_1\big)
\times L^2 ([0,T]\times \bb T^d ; \bb R^d) \to \bb R^d$ be defined by
\begin{equation*}
H_f(t,x, \bs \rho, \bs j) \;=\; f (x,\bs \rho(t), \bs j (t))\;.
\end{equation*}
By \eqref{07}, $P$--almost surely, $\bs \pi_\delta(t)\to \bs \rho(t)$
and $\bs j_\delta(t)\to \bs j(t)$ for Lebesgue almost all $t\in(0,T)$.
Since $f$ is bounded, by dominated convergence,
\begin{equation}
\label{28}
\lim_{\delta \to 0} E_P \Big[ \int_{0}^T dt \int_{\bb T^d} dx \, \Big\{  \,
G_{f_\delta} (t,x, \bs \pi  , \bs J) \,-\, H_f (t,x, \bs \rho
, \bs j) \Big\}^2 \, \Big] \;=\; 0\;,
\end{equation}
where $G_{f_\delta}$ has been introduced in \eqref{23}.

Denote by $W_{T,f}$ the right-hand side of \eqref{21} when $G_w$ is
replaced by $H_f$.  By \eqref{07}, \eqref{28} and the dominated
convergence theorem,
\begin{equation*}
\lim_{\delta \to 0}
E_P \big[ \, V_{T,f_\delta} \,\big] \;=\;
E_P \big[\, W_{T,f} \, \big] \;.
\end{equation*}

In view of the previous estimates, it remains to show that
\begin{equation}
\label{03}
\sup_f E_P \big[\, W_{T,f} \, \big] \;\ge\; \bs I_m(P)\;.
\end{equation}
Let $\hat f \colon \bb T^d \times \mc H_1
\times L^2 (\bb T^d ; \bb R^d) \to \bb R^d$ be given by
\begin{equation*}
\hat f (x, \rho, j) \;=\; \frac{j(x) + (\nabla \rho)(x) 
- \sigma(\rho(x))\, E(x)}{2\, \sigma(\rho(x))} 
\end{equation*}
and note that $\bs I_m (P) =E_P[ \, W_{T,\hat f} \, ]$, which is
finite in view of \eqref{07}.  Using this bound and approximating the
function $\hat f$ by a sequence of bounded and continuous functions
that are continuously differentiable in $x$, we obtain \eqref{03} by
the dominated convergence theorem.
\end{proof}

\smallskip\noindent{\bf \ref{sec03}.3 The $\Gamma$-limsup.}
Given $P\in \ms P_{\rm stat}$, we shall construct a sequence
$(P_N : N \ge 1)$, such that $P_N \to P$ and
\begin{equation*}
\limsup_{N\to\infty} \bs I_{N,K_N} (P_N) \;\le\; \bs I_m (P)\;.
\end{equation*}
We carry this out first for $P$ satisfying certain regularity assumptions,
and then use density arguments to extend the result to any $P$
with finite rate function, $\bs I_m(P)<+\infty$.

Fix $T>0$. Recalling \eqref{36}, a path $(\bs \pi, \bs J)$ in
$\ms S_m$ is $T$-periodic if
$\vartheta_T (\bs \pi, \bs J) = (\bs \pi, \bs J)$.  An element $P$ in
$\ms P_{\rm stat}$ is said to be \emph{$T$-holonomic} if there exists
a $T$-periodic path $(\bs \pi, \bs J)$ such that
\begin{equation}
\label{16}
P \;=\; \frac 1T\, \int_0^T \delta_{\vartheta_s (\bs \pi, \bs J)} \,
ds \;.
\end{equation}
An element of $\ms P_{\rm stat}$ is \emph{holonomic} if it is
$T$-holonomic for some $T>0$.

Fix $T>0$ and a $T$-periodic path $(\bs \pi, \bs J)$ in
$\ms S_{m, {\rm ac}}$. Denote by $(\bs \rho, \bs j)$ the densities so
that $\bs \pi (t,dx) = \bs\rho(t,x)\, dx$,
$\bs J (t) = \int_0^t\!ds\, \bs j(s)$.  Assume that $\bs \rho$,
$\bs j$ are smooth and that there exists $\delta>0$ such that
$\delta \le \bs \rho(t,x) \le 1-\delta$ for all $(t,x)$.
Denote by $F\colon \bb R\times \bb T^d\to  \bb R^d$ 
the smooth, $T$-periodic vector field defined by
\begin{equation*}
F \;=\;
\frac{ \bs j \,+\, \nabla \bs\rho \,-\, \sigma (\bs\rho)  E}
{\sigma(\bs\rho)}\; \cdot
\end{equation*} 
As the path $(\bs \rho, \bs J)$ satisfies the continuity
equation \eqref{31}, $\partial_t \bs \rho + \nabla\cdot \bs j =
0$, by definition of $F$,
\begin{equation}
\label{35}
\partial_t \bs \rho \;=\;\Delta \bs \rho \;-\; \nabla\cdot \big(\,
\sigma (\bs\rho) \, \big[\, E + F(t)\, \big]\,\big)\;.
\end{equation}
Let finally $P \in \ms P_{\rm stat}$ be the $T$-holomic probability
corresponding, as in \eqref{16}, to the $T$-periodic path $(\bs
\pi,\bs J)$.

Let $L_{N}^F$ be the time-dependent generator of a perturbed WASEP
defined by
\begin{equation}
\label{genF}
(L_{N}^F \, f)(\eta) \;=\; N^2 \sum_{(x,y)\in \bb B_N} \eta_x\, [
1-\eta_y] \, e^{(1/2) [\, E_N(x,y) + F_N(t,x,y)\,]} \big[
f(\sigma^{x,y} \eta)-f(\eta)\big] \;,
\end{equation}
where $F_N(t,x,y)$ represents the line integral of $F(t, \cdot)$ along
the oriented segment from $x$ to $y$ introduced in \eqref{dvf}.
Denote by $(\bs \eta^F(t) : t\ge 0)$ the continuous-time,
time-inhomogeneous Markov chain whose generator is $L_{N}^F$.
By the hydrodynamic limit of the time dependent WASEP dynamics, see e.g.\
\cite[Section 10.5]{kl}, the empirical density $\bs \pi_N(\bs \eta^F)$
associated to the process $\bs \eta^F$ evolves, in the limit
$N\to \infty$, according to the solution of the PDE \eqref{35}. This
explains the introduction of the process $\bs\eta^F$.

Let $(\xi_k : k\ge 0)$ be the discrete-time, $\Sigma_N$-valued Markov
chain given by $\xi_k = \bs\eta^F(kT)$. Since $F$ is $T$-periodic,
$\xi_k$ is time-homogeneous. As it is irreducible on each set
$\Sigma_{N,K}$, $\xi_k$ has a unique stationary state, denoted by
$\mu^F_{N,K}$. Let finally, $\bb P^F_{N,K}$ be the law of $\bs
\eta^F$ when the initial condition is sampled according to
$\mu^F_{N,K}$. Note that $\bb P^F_{N,K}$ is invariant by
$T$-translations: $\bb P^F_{N,K} \circ \vartheta^{-1}_T = \bb
P^F_{N,K}$.  Since this measure, defined on $D(\bb R_+, \Sigma_N)$, is
invariant by $T$-translations, we may extend it to $D(\bb R,
\Sigma_N)$.

Let $P_N$ be the measure on $\ms S$ given by
\begin{equation}
\label{10}
P_N \;=\; \Big( \frac 1T \int_0^T \, \bb P^F_{N,K_N} \circ
\vartheta^{-1}_t  \, dt \Big) \circ (\bs \pi_N, \bs J_N)^{-1}\;,
\end{equation}
that, by construction, belongs to $\ms P_{\rm stat}$.

\begin{proposition}
\label{prop02}
For each $T$-periodic path $(\bs\pi,\bs J)$ in $\ms S_{m,\mathrm{ac}}$
with smooth densities $(\bs \rho,\bs j)$ with $\bs\rho$ bounded away
from zero and one, the sequence $(P_N : N \ge 1)$ introduced in
\eqref{10} converges to the $T$-holonomic probability $P$ given by
\eqref{16} and
\begin{equation*}
\limsup_{N\to\infty} \frac 1{N^d} \bs I_{N,K_N} (P_N) \;\le\; \bs I_m(P)\;.
\end{equation*}
\end{proposition}

The proof of this proposition relies on the following lemma.

\begin{lemma}
\label{le07}
In the setting of Proposition \ref{prop02}, the sequence of
probability measures on $D(\bb R;\mc M_+(\bb T^d)\big)$ given by
$\bb P_{N,K_N}^F \circ \bs\pi_N^{-1}$ converges to $\delta_{\bs\pi}$.
\end{lemma}

\begin{proof}
By the smoothness of the external field $F$, Lemma \ref{l01} holds
also when $\bb P_{\mu_{N,K_N}}^N$ is replaced by $\bb P^F_{N,K_N}$.
By the compactness of $\mc M_+(\bb T^d)$, this implies the
pre-compactness of the family
$\big(\bb P_{N,K_N}^F \circ \bs\pi_N^{-1}\,: \, N\ge 1\big)$.  Let
$\mc P^F$ be a cluster point of this sequence. By the $T$-periodicity
of $\bb P_{N,K_N}^F$ and the hydrodynamic limit for the time dependent
WASEP with generator \eqref{genF}, $\mc P^F$ is $T$ periodic and
$\mc P^F$ almost surely $\bs \pi(t,dx) =\bs \rho^F (t,x) dx$ for some
density $\bs \rho^F$ of mass $m$ that solves \eqref{35}.  By the
uniqueness of $T$-periodic solutions to \eqref{35} and the
$L^1(\bb T^d)$ convergence to this unique solution, as stated in
Theorem \ref{up01}, $\bs \rho^F = \bs \rho$.  Hence
$\mc P^F=\delta_{\bs\pi}$ as claimed.
\end{proof}

\begin{proof}[Proof of Proposition \ref{prop02}]
By the smoothness of the external field $F$, Lemmata \ref{l01},
\ref{l02}, and \ref{l05} holds also when $\bb P_{\mu_{N,K_N}}^N$ is
replaced by $\bb P^F_{N,K_N}$. This implies the pre-compactness of the
sequence of probabilities on $\ms S$ given by
$\big\{ \bb P^F_{N,K_N} \circ (\bs \pi_N, \bs J_N)^{-1} \big\}$.  Let
now $P_0$ be a cluster point of this sequence.  By the hydrodynamic
limit for the perturbed WASEP, $P_0$-almost surely, $(\bs \pi, \bs J)$
belongs to $\ms S_{m,{\rm ac}}$ and the corresponding densities
$(\bs \rho,\bs j)$ is a $T$-periodic weak solution to the hydrodynamic
equation
\begin{equation*}
\label{hlp}
\left\{
\begin{aligned}
& \partial_t \bs \rho \;+\; \nabla \cdot \bs j \;=\; 0\;, \\
& \bs j \;=\; -\, \nabla \bs \rho \; +\; \sigma(\bs \rho) \,[\,
E+F\,]\; .
\end{aligned}
\right.
\end{equation*}
By Lemma \ref{le07}, $P_0= \delta_{(\bs \pi, \bs J)}$.  Taking
time-averages, we deduce $P_N\to P$.

We turn to the second claim of the proposition. By \eqref{12},
\begin{equation}
\label{15}
\bs I_{N,K_N} (P_N) \;=\; H_{N,K_N} \Big( \, \frac 1T\, \int_0^T \,
\bb P^F_{N,K_N} \circ \, \vartheta^{-1}_t \, dt \, \Big) \;.
\end{equation}
Fix $\ell\in \bb N$. By \eqref{13} and the convexity of the relative
entropy, the right-hand side of the previous equation is equal to
\begin{equation*}
\begin{split}
& \lim_{\ell\to\infty} \frac 1{\ell T}\, \bb H^{(\ell T)} \Big( \, 
\frac 1T\, \int_0^T \, \bb P^F_{N,K_N} \circ \, 
\vartheta^{-1}_t \, dt\, \big|\, \bb P^N_{\mu_{N,K_N}}\Big)
\\
&\qquad 
\;\le\; \lim_{\ell\to\infty} 
\frac 1{\ell\, T^2}\, \int_0^T \bb H^{(\ell T)} \Big( \, 
\bb P^F_{N,K_N} \circ \, \vartheta^{-1}_t \, \big|\,
\bb P^N_{\mu_{N,K_N}} \, \Big) \, dt \;.
\end{split}
\end{equation*}

Since $\bb P^N_{\mu_{N,K_N}} $ is translation invariant, for fixed
$\ell$ and $0\le t\le T$, by definition of the translations
$(\vartheta_s : s\in\bb R)$, introduced in \eqref{36}, recalling
\eqref{rem}, 
\begin{align*}
\bb H^{(\ell T)} \Big( \, 
\bb P^F_{N,K_N} \circ \, \vartheta^{-1}_t \, \big|\,
\bb P^N_{\mu_{N,K_N}} \, \Big) \; & =\;
\bb H_{[0\,,\,\ell T]} \Big( \, 
\bb P^F_{N,K_N} \circ \, \vartheta^{-1}_t \, \big|\,
\bb P^N_{\mu_{N,K_N}} \circ \, \vartheta^{-1}_t \, \Big) \\
& =\;
\bb H_{[-t\,,\, \ell T-t]}
\Big( \,  \bb P^F_{N,K_N} \, \big|\, \bb P^N_{\mu_{N,K_N}} \, \Big)\;.
\end{align*}
As $\bb P^F_{N,K_N} $ is $T$-translation invariant and
$\bb P^N_{\mu_{N,K_N}} $ is translation invariant, the dynamical
contribution to the relative entropy of the time interval $[-t,0]$
corresponds to the one of the time interval $[-t + \ell T, \ell
T]$. Hence, the previous expression is equal to
\begin{equation*}
\bb H^{(\ell T)} \Big( \, \bb P^F_{N,K_N}
\, \big|\, \bb P^N_{\mu_{N,K_N}} \, \Big)  \;-\;
\textrm{Ent}\big(\, \mu_{N,K_N}^F \,\big|\, \mu_{N,K_N}\,\big)  
\;+\; \textrm{Ent}\Big(\, \bb
P_{N,K_N}^F \circ \imath_{-t}^{-1} \,\big|\, \mu_{N,K_N}\, \Big) \;, 
\end{equation*}
where
$\bb P_{N,K_N}^F \circ \imath_s^{-1}$ is the marginal at time $s$ of
$\bb P_{N,K_N}^F$.

Using again that $\bb P^F_{N,K_N} $ is $T$-translation invariant and
$\bb P^N_{\mu_{N,K_N}}$ is translation invariant,
\begin{align*}
& \bb H^{(\ell T)} \Big( \, \bb P^F_{N,K_N}
\, \big|\, \bb P^N_{\mu_{N,K_N}} \, \Big)  \;-\;
\textrm{Ent}\big(\, \mu_{N,K_N}^F \,\big|\, \mu_{N,K_N}\,\big)  \\
&\qquad \;=\; \ell \, \bb H^{(T)} \Big( \, \bb P^F_{N,K_N}
\, \big|\, \bb P^N_{\mu_{N,K_N}} \, \Big)  \;-\; \ell \,
\textrm{Ent}\big(\, \mu_{N,K_N}^F \,\big|\, \mu_{N,K_N}\,\big)  \;. 
\end{align*}
Putting together the previous estimates and letting $\ell\to\infty$
yields that
\begin{equation*}
\bs I_{N,K_N} (P_N) \;\le\;
\frac{1}{T} \,\Big\{
\bb H^{(T)} \Big( \, \bb P^F_{N,K_N}
\, \big|\, \bb P^N_{\mu_{N,K_N}} \, \Big) \;-\; 
\textrm{Ent}\big(\, \mu_{N,K_N}^F \,\big|\, \mu_{N,K_N}\,\big)
\,\Big\}\;.
\end{equation*}

By Lemma \ref{le07} and the large deviations lower bound in
hydrodynamical limits, see e.g.\ \cite[Lemma~10.5.4]{kl}, 
\begin{equation*}
\limsup_{N\to \infty} \frac 1{N^d}\, 
\Big\{\, \bb H^{(T)} \Big( \, \bb P^F_{N,K_N} 
\, \big|\, \bb P^N_{\mu_{N,K_N}}  \, \Big)
- \textrm{Ent }\big(\mu_{N,K_N}^F\big| \mu_{N,K_N}\big)\,
\Big\}
\;\le\; A_{T,m}(\bs \pi,\bs J) \;.
\end{equation*} 
Therefore,
\begin{equation*}
\limsup_{N\to \infty} \frac 1{N^d}\, \bs I_{N,K_N} (P_N)
\;\le\; \frac 1{T} \, A_{T,m}(\bs \pi,\bs J) \;.
\end{equation*}
The right-hand side is equal to $\bs I_m(P)$ in view of the equations
\eqref{06}, \eqref{16}, and the $T$-periodicity of the path
$(\bs \pi,\bs J)$.
\end{proof}

To complete the proof of the $\Gamma$-limsup, we show that any $P$ in
$\ms P_{\rm stat}$ can be approximated by convex combinations of
holonomic probability measures supported by smooth paths bounded away
from zero and one and that the corresponding rate function converges
to $\bs I_m(P)$.
Denote by $\color{bblue} \ms P^\epsilon_{{\rm stat}, m}$,
$\epsilon>0$, the subset of $\ms P_{{\rm stat}, m}$ formed by the
stationary measures $P$ such that $P$-almost surely $(\bs \pi, \bs J)$
belongs to $\ms S_{m,\mathrm{ac}}$ with smooth densities
$(\bs \rho,\bs j)$ such that
$\epsilon \le \bs \rho(t,x) \le 1-\epsilon$ for all $(t,x)$.

\begin{theorem}
\label{t:dtI}
Assume that $E$ is orthogonally decomposable and fix
$P\in \ms P_{\rm stat}$ such that $\bs I_m(P) <+\infty$.  There exist
a sequence $(\epsilon_n : n\ge 1)$, a triangular array
$(\alpha_{n,i} \,,\, 1\le i\le n \,,\, n\ge 1)$ with
$\alpha_{n,i}\ge 0$, $\sum_i \alpha_{n,i}=1$, and a triangular array
$(P_{n,i} \,,\, 1\le i\le n \,,\, n\ge 1)$ of holonomic measures
belonging to $\ms P^{\epsilon_n}_{{\rm stat}, m}$ such that by setting
$P_n := \sum_{i} \alpha_{n,i} P_{n,i}$ we have $P_n\to P$ and
$\bs I_m(P_n) \to \bs I_m(P)$.
\end{theorem}

Postponing the proof of this statement, we first conclude the
$\Gamma$-convergence of the Donsker-Varadhan functional.

\begin{proof}[Proof of Theorem \ref{th03}. $\Gamma$-limsup]
  The statement follows, by a diagonal argument, from 
  Proposition~\ref{prop02} and Theorem~\ref{t:dtI}.
\end{proof}

We turn to the proof of Theorem \ref{t:dtI}. It relies on two
lemmata. 

\begin{lemma}
\label{lem7}
Fix $P$ satisfying \eqref{fe}. There exists a sequence
$(P_n\colon n\ge 1)$ converging to $P$ and such that $P_n$ belongs to
$\ms P^{\epsilon_n}_{{\rm stat}, m}$ for some $\epsilon_n>0$,
and
\begin{equation*}
\limsup_{n\to\infty} \bs I_m(P_n) \; = \; \bs I_m(P) \;.
\end{equation*}
\end{lemma}

\begin{proof}
Fix a smooth probability density
$\phi : \bb R \times \bb R^d \to \bb R_+$ whose support is
contained in $[-1,1]\times [-1,1]^d$, so that
\begin{equation*}
\int_{\bb R \times \bb R^d} \phi(t,x) \, dt\, dx \;=\;1 \;.
\end{equation*}
Let
$\phi_\epsilon (t,x) = \epsilon^{-(d+1)} \phi
(t/\epsilon,x/\epsilon)$, $\epsilon>0$.  For a trajectory
$(\bs \pi, \bs J)$ in $\ms S_{m, \mathrm{ac}}$, whose density is represented by
$(\bs \rho, \bs j)$, let
\begin{equation*}
\bs \rho_\epsilon \;: =\; (1-\epsilon) \, (\bs \rho* \phi_\epsilon) \;+\;
\epsilon\, m \;, \quad
\bs j_\epsilon \;: =\; (1-\epsilon) \, (\bs j* \phi_\epsilon) \;,
\end{equation*}
where $*$ denotes space-time convolution and $0<\epsilon <1$. Observe
that $(\bs \rho_\epsilon , \bs j_\epsilon)$ satisfy the continuity
equation. 

Denote by $\Psi_\epsilon$ the map
$(\bs \rho, \bs j) \mapsto (\bs \rho_\epsilon, \bs j_\epsilon)=:
\Psi_\epsilon(\bs \rho,\bs j)$ and set
$P^\epsilon = P \circ \Psi^{-1}_\epsilon$.  Then, for each
$\epsilon>0$, the probability $P^\epsilon$ belongs to
$\ms P^\delta_{{\rm stat}, m}$ for some $\delta = \delta (\epsilon)>0$
and $P^\epsilon \to P$ as $\epsilon \to 0$.

It remains to show that
$\lim_{\epsilon} \bs I_m(P^\epsilon) = \bs I_m(P)$. As
\begin{equation*}
\bs I_m(P^\epsilon) \;=\; E_{P} \Big[\, \frac 1T \int_0^T dt \int_{\bb
  T^d} dx \, \frac{|\bs j_\epsilon + \nabla \bs \rho_\epsilon -
  \sigma(\bs \rho_\epsilon)\, E|^2}{4\, \sigma(\bs \rho_\epsilon)} \,
\Big]\;,
\end{equation*}
and since the sequence 
\begin{equation}
\label{17}
\frac{|\bs j_\epsilon + \nabla \bs \rho_\epsilon - \sigma(\bs
\rho_\epsilon)\, E|^2}
{4\, \sigma(\bs  \rho_\epsilon)}
\end{equation}
converges $(dP\, dt\, dx)$-almost surely to the same expression without
the subscript $\epsilon$, it is enough to prove that the sequence
\eqref{17} is uniformly integrable.

Since $P$ satisfies \eqref{fe}, by \cite[Lemma 5.3]{blm} there exist
increasing convex functions $\Upsilon_1$,
$\Upsilon_2 : \bb R_+ \to \bb R_+$ such that
$\lim_{r\to\infty} \Upsilon_a(r)/r = \infty$, $a=1$, $2$, and
\begin{equation}
\label{19}
E_P \Big[\, \int_0^T dt \int_{\bb T^d} dx \, \Big\{\, \Upsilon_1 \Big(
\frac{|\bs j|^2}{\sigma(\bs \rho)} \Big) + \Upsilon_2 \Big(
\frac{|\nabla \bs \rho|^2}{\sigma(\bs \rho)} \Big) \,\Big\}\, \Big]
\;<\;+\infty\;.
\end{equation}
Moreover, the uniform integrability of the sequence \eqref{17} follows
from the bound
\begin{equation}
\label{18}
\limsup_{\epsilon \to 0} E_{P} \Big[\, \int_0^T dt \int_{\bb T^d} dx
\, \Big\{\, \Upsilon_1 \Big( \frac{|\bs j_\epsilon|^2}{\sigma(\bs
  \rho_\epsilon)} \Big) + \Upsilon_2 \Big( \frac{|\nabla \bs
  \rho_\epsilon|^2} {\sigma(\bs \rho_\epsilon)} \Big) \,\Big\}\, \Big]
\;<\;+\infty\;.
\end{equation}

Note that
\begin{align*}
|\bs j_\epsilon|^2 \; = \; [1-\epsilon]^2 \, \big| \, \bs j*
\phi_\epsilon \, \big|^2 \;\le\; 
\big[\, (1-\epsilon) \, \sigma(\bs \rho * \phi_\epsilon) \,+\,
\epsilon \, \sigma(m) \,\big] \,
\Big\{ \, [1-\epsilon] \, \frac{ |\, \bs j* \phi_\epsilon\,|^2 }
{\sigma(\bs \rho * \phi_\epsilon)} \,\Big\} \;.
\end{align*}
By the concavity of $\sigma$, the first term on the right-hand side is
bounded by $\sigma(\bs \rho_\epsilon)$. In conclusion,
\begin{equation*}
\frac{|\bs j_\epsilon|^2}{\sigma(\bs \rho_\epsilon)} \;\le\; 
\, [1-\epsilon] \, \frac{ |\, \bs j* \phi_\epsilon\,|^2 }
{\sigma(\bs \rho * \phi_\epsilon)}
\;\le\; 
\, \frac{ |\, \bs j* \phi_\epsilon\,|^2 }
{\sigma(\bs \rho * \phi_\epsilon)}
\; \cdot
\end{equation*}
On the other hand, by concavity of $\sigma$ and Cauchy-Schwarz
inequality,
\begin{align*}
& \frac{|\bs j * \phi_\epsilon|^2}{\sigma(\bs
  \rho * \phi_\epsilon)} (t,x)  \\
& \;\le\;
\frac{1}{ [\, \sigma(\bs \rho)  * \phi_\epsilon\,](t,x) } \Big(\int_{\bb R} ds
\int_{\bb R^d} dy \, \phi_\epsilon (t-s, x-y) \, \frac{j(s,y)}
{\sqrt{\sigma(\bs \rho (s,y))}}  \sqrt{\sigma(\bs \rho (s,y))} \,
\Big)^2 \\
&\qquad \le\; \Big(\, \phi_\epsilon \,*\, 
\frac{|\bs j|^2}{ \sigma(\bs \rho)} \, \Big)(t,x) \;.
\end{align*}
Whence, as $\Upsilon_1$ is increasing,
\begin{equation*}
E_P \Big[\, \int_0^T dt \int_{\bb T^d} dx \, \Upsilon_1 \Big(
\frac{|\bs j_\epsilon|^2}{\sigma(\bs \rho_\epsilon)} \Big) \, \Big]
\le\; E_P \Big[\, \int_0^T dt \int_{\bb T^d} dx \, \Upsilon_1 \Big(
\phi_\epsilon \,*\, \frac{|\bs j|^2}{ \sigma(\bs \rho)} \Big) \,
\Big]\;.
\end{equation*}
Since $\Upsilon_1$ is convex, integrating the convolution, we deduce
that the previous expression is bounded by
\begin{equation*}
E_P \Big[\, \int_0^T dt \int_{\bb
  T^d} dx \, \phi_\epsilon \,*\,  \Upsilon_1 \Big( \frac{|\bs j|^2}
{\sigma(\bs \rho)} \Big) \, \Big] 
\;=\; E_P \Big[\,  \int_0^T dt \int_{\bb
  T^d} dx \, \Upsilon_1 \Big( \frac{|\bs j|^2}
{\sigma(\bs \rho)} \Big) \, \Big] \;.
\end{equation*}

Since the previous argument for $\bs j$ applies to $\nabla\bs \rho$,
the bound \eqref{18} follows from \eqref{19}.
\end{proof}

\begin{lemma} 
\label{t:l4.7}
Assume that $E$ is orthogonally decomposable.  There exists $T_0>0$
and $C_0>0$ such that the following holds.  For any
$\pi_0,\pi_1\in \mc M_m(\bb T^d)$ there exists $\bar T \le T_0$ and a
path $\big(\bs\pi(t),\bs J(t) \big)$, $t\in [0,\bar T]$, with
$\bs \pi (0)=\pi_0$, $\bs \pi(\bar T)=\pi_1$, satisfying the
continuity equation \eqref{31} for each $0\le s<t \le \bar T$ and
\begin{equation}
\label{112}
A_{m,\bar T}(\bar{\bs\pi},\bar{\bs J}) \le C_0\;.
\end{equation}
\end{lemma}

\begin{proof}
In the case of the symmetric exclusion process, this statement is
proven in \cite[Lemma~4.7]{bdgjl07} with the following strategy. Start
from $\pi_0$ and follow the hydrodynamic equation for a long but fixed
time interval $[0,T_1]$ so that $\bs\pi(T_1)$ lies in small
neighborhood of the stationary solution with mass $m$. Then
interpolate, in the time interval $[T_1,T_1+1]$ from $\bs \pi(T_1)$ to
a suitable $\hat\pi$ that is still close to the stationary
solution. Finally, from $\hat \pi$ use the optimal path for the escape
problem to reach $\pi_1$.  Provided that the \emph{quasi-potential} is
bounded, this argument applies also to the WASEP case.  As discussed
in \cite[\S~V.C]{mft} and \cite{bfg}, if the external field is
orthogonally decomposable then the quasi-potential can be computed
explicitly and it is indeed bounded.  
\end{proof}

\begin{proof}[Proof of Theorem~\ref{t:dtI}]
Fix $P\in \ms P_{{\rm stat}, m}$.  By Lemma~\ref{lem7} we can assume
that $P\in \ms P^\epsilon_{{\rm stat}, m}$ for some $\epsilon>0$.
Since $P$ can be written as a convex combination of ergodic
probabilities and $\bs I_m$ is affine, it suffices to show that for
each ergodic $P\in \ms P^\epsilon_{{\rm stat}, m}$ with
$\bs I_m(P) < +\infty$ there exists a sequence of holonomic measures
$P_T$ in $\ms P^\epsilon_{{\rm stat}, m}$ converging to $P$ and such
that $\lim_T \bs I_m(P_T) = \bs I_m(P)$.

Recalling that the $T$-periodization of paths in $\ms S$ as been
defined in \eqref{Tper}, set
\begin{equation*}
\begin{split}
\ms A_P := \Big\{ & (\bs\pi, \bs J) \in \ms S_{m}\colon \; \frac 1T
\int_0^T\delta_{\vartheta_t (\bs \pi^T,\bs J^T)}\,dt \to P
\\
& \quad \textrm{ and } \lim_{T\to +\infty} \frac 1T A_{T,m}(\bs
\pi,\bs J) \to E_P \big[ A_{1,m} \big] \Big\}\;.
\end{split}
\end{equation*}
Since $P$ is ergodic, by the Birkhoff ergodic theorem $P(\ms A_P)=1$.
Pick an element $({\bs\pi}^*,{\bs J}^*) \in \ms A_P$.  By definition,
the $T$-holonomic probability associated to the $T$-periodization of
$(\bs\pi^*,\bs J^*)$ converges to $P$ but, in general, its rate
function does not since when $T$-periodizing paths we may insert
jumps.  By using Lemma~\ref{t:l4.7}, we now show how the path
$(\bs\pi^*,\bs J^*)$ can be modified to accomplish our needs.

Given $T^*>0$, let $\pi_0 ={\bs\pi}^*(T^*)$ and
$\pi_1 ={\bs\pi}^*(0)$. Let also $(\bar{\bs\pi} (u),\bar {\bs J}(u))$,
$u\in [0,\bar T]$ the path provided by Lemma~\ref{t:l4.7} satisfying
$\bar{\bs\pi} (0)=\pi_0$, $\bar{\bs\pi} (\bar T)=\pi_1$.

Set $T := T^*+\bar T$ and let $({\bs\pi}(u),{\bs J}(u))$, $u\in [0,T]$
be the path defined by
\begin{equation*}
\big({\bs\pi}(u),{\bs J}(u)\big) =
\begin{cases}
\big( {\bs\pi}^*(u),{\bs J}^*(u)\big) & \textrm{ if $u\in
  [0,T^*]$,}\\
\big( \bar{\bs\pi} (u-T^*),\bs J(T^*)+ \bar{\bs J}(u-T^*) \big)) &
\textrm{ if $u\in(T^*, T]$.}
\end{cases}
\end{equation*}

Observe that $\bs\pi(0)=\bs\pi(T)$ and extend $(\bs\pi,\bs J)$ to the
path $(\bs \pi^T,{\bs J}^T)$ defined on $\bb R$ by periodicity.  By
construction, $t\mapsto \bs\pi^T(t)$ is continuous and denote by $P_T$
the $T$-holonomic measure associated to $(\bs\pi^T,\bs J^T)$ as in
\eqref{16}.  Since $\bar T \le T_0$ for some fixed $T_0$, $P_T\to P$
as $T\to\infty$. Moreover, by construction and by Lemma \ref{t:l4.7},
\begin{equation*}
\begin{split}
\bs I_m (P_T) & = \frac 1T A_{m,T} (\bs \pi,\bs J) = \frac 1 {T}
A_{m,T^*}(\bs \pi^*, \bs J^*) + \frac {1}{T} A_{m,\bar
  T}(\bar{\bs\pi},\bar {\bs J})
\\
&\le \frac 1 {T} A_{m,T^*}(\bs \pi^*, \bs J^*) + \frac {1}{T} C_0
  \end{split}
\end{equation*}
so that, since $(\bs \pi^*,\bs J^*)$ belongs to $\ms A_P$,
$\limsup_{T\to \infty} \bs I_m (P_T) \leq \bs I_m(P)$. As $\bs I_m$ is
lower semi-continuous, actually,
$\lim_{T\to \infty} \bs I_m (P_T) = \bs I_m(P)$, as claimed.
\end{proof}

\section{
Long time behavior of the hydrodynamical rate function
}
\label{s33}

In this section, we consider the asymptotic in which we take first the
limit as $N\to\infty$ and then $T\to\infty$. The former limit is
essentially the content of the large deviations from the
hydrodynamical scaling limit in which we emphasize that the
corresponding $T$-dependent rate function still depends on the initial
condition. To analyze the limit as $T\to\infty$ we first lift this
rate function to the set of translation invariant probabilities on
$\ms S$ and then analyze its variational convergence, showing in
particular that the limit is independent of the initial condition.

Hereafter, fix $m\in (0,1)$ and a sequence $K_N$ such that
$N^{-d}K_N\to m$.

\subsection{Hydrodynamical large deviations}

Recall that the sequence $\{\eta^N, \, N\ge 1\}$,
$\eta^N\in\Sigma_{N,K_N}$ is \emph{associated} to a measurable density
$\rho\colon \bb T^d\to [0,1]$ satisfying $\int\rho \,dx =m$ (hereafter
of total mass $m$) if and only if $\pi_N(\eta^N)\to \rho(x) \,dx $ in
the topology of $\mc M_+(\bb T^d)$.

Recalling that $\ms M = \mc M_+(\bb T^d)\times \mc H_{-p}^d$, let
$\color{bblue} \ms S^T$, $T>0$, be the subset of the paths in
$D([0,T]; \ms M)$ which satisfy the continuity equation \eqref{31} for
any $0<s<t<T$. Let also $\color{bblue} \ms S^T_{m}$ be the subset of
$\ms S^T$ given by the elements $(\bs\pi,\bs J)$ which satisfy
$\bs \pi(t,\bb T^d)=m$ for all $0\le t\le T$.

Finally, given a measurable function $\rho\colon \bb T^d\to [0,1]$ of
total mass $m$, that plays the role of the initial datum, let
$\color{bblue} \ms S^T_{m,\mathrm{ac},\rho}$ be the subset of elements
$(\bs \pi, \bs J)$ in $\ms S^T_{m}$ such that
\begin{itemize}
\item[(a')] $\bs\pi \in C([0,T], \mc M_m (\bb T^d))$, and
  $\bs\pi(t,dx) = \bs \rho(t,x)\, dx$ for some $\bs \rho$ such that
  $0\le \bs \rho(t,x) \le 1$. Moreover, 
\begin{equation*}
 \int_0^T dt \int_{\bb T^d} dx\, \frac{|\, \nabla
   \bs\rho\,|^2}{\sigma(\bs\rho)}\: < \; \infty\;;
\end{equation*}
\item[(b')] $\bs J \in C([0,T], \mc H^d_{-p})$, and
  $ \bs J (t) \,=\, \int_0^t \bs j(s)\, ds$, $t\in [0,T]$,
  for some $\bs j$ in
  $L^2\big([0,T] \times \bb T^d, \sigma(\bs \rho(t,x))^{-1} \,
  dt\, dx ; \bb R^d\big)$. Thus, 
  \begin{equation*}
  \int_{0}^T dt \int_{\bb T^d} dx\, \frac{|\, \bs j\,|^2}{\sigma(\bs
    \rho)}\: < \; \infty\;;
  \end{equation*}
\item[(c')] $\bs\pi(0, dx)= \rho(x) \,dx$.
\end{itemize}
Note that conditions (a') and (b') are the same of conditions (a) and
(b) below equation \eqref{31} apart from the fact that the here the
path $\big(\bs \rho (t),\bs j (t) \big)$ is defined only for
$t\in[0,T]$. 

Let the \emph{action} $A_{m,T,\rho}: \ms S^{T} \to [0, +\infty]$, be
defined by
\begin{equation}
\label{act}
A_{m,T,\rho} (\bs\pi, \bs J) \, =\,
\begin{cases}
\displaystyle \int_{0}^T\! dt\!\int_{\bb T^d}\! dx\, \frac{|\bs j +
  \nabla \bs \rho - \sigma(\bs \rho)\, E|^2}{4\, \sigma(\bs \rho)} &
\text{
  if $(\bs \pi, \bs J) \in \ms S_{T,m, {\rm ac},\rho}$, } \\ 
\vphantom{\bigg\{} +\infty & \text{ otherwise.}
\end{cases}
\end{equation}

The large deviation principle with respect to the hydrodynamical limit
for the WASEP dynamics can be stated as follows. Here, we understand
that the empirical density and current $(\bs\pi_N,\bs J_N)$ is
defined as a map from $D\big([0,T];\Sigma_{N,K_N}\big)$ to $\ms S^T$.

\begin{theorem}
\label{th10}
Fix $T>0$, $m>0$ and a density profile $\rho:\bb T^d \to [0,1]$ of
total mass $m$. 
For each sequence 
$(\eta^N: N\ge 1)$ associated to $\rho$, each
closed set $\ms F\subset \ms S^{T}$, and each open set 
$\ms G\subset \ms S^{T}$,
\begin{gather*}
\limsup_{N\to \infty} \frac {1}{N^d} \, \log \bb P^N_{\eta^N} \,
\big[ (\bs\pi_{N} , \bs J_{N}) \in \ms F\,] \;\le\; -\, \inf_{(\bs\pi ,
  \bs J)\in \ms F}
A_{m,T,\rho} (\bs\pi , \bs J)\;, \\
\liminf_{N\to \infty} \frac {1}{N^d} \, \log \bb P^N_{\eta^N} \,
\big[ (\bs\pi_{N} , \bs J_{N}) \in \ms G\,] \;\ge\; -\, \inf_{(\bs\pi ,\bs
  J)\in \ms G} A_{m,T,\rho} (\bs\pi ,\bs J)\;.
\end{gather*}
Moreover, 
$A_{m,T,\rho}\colon \ms S^{T} \to [0, +\infty]$ is a good rate
function. 
\end{theorem}

If one considers only the empirical density and disregards the
empirical current, the above result has been proven in \cite{kov} in
case of SEP, see also \cite[Ch.~10]{kl}. This result has been extended
to WASEP in \cite{blm}.  Relying on the super-exponential estimates
proven in \cite{kov}, the case in which one considers the empirical
current is discussed in \cite{bdgjl07} for the SEP. However, the
topology on the set of currents there introduced is different from the
one used in the present paper and the proof of the exponential
tightness is incomplete.  The issue of the exponential tightness of
the empirical current is fixed in the present paper (in the topology
here introduced).  Indeed, Lemmata~\ref{l01}, \ref{l02}, and \ref{l05}
which hold also in the present setting, yield the exponential
tightness of the sequence
$\big( \bb P^N_{\eta^N}\circ (\bs\pi_{N} , \bs J_{N})^{-1}\colon N\ge
1 \big)$ thus completing, together with \cite{bdgjl07}, the proof of
the above result for the SEP. The extension to WASEP requires, for the
lower bound, a density argument that has been carried out in detail in
\cite[Thm.~5.1]{blm} and can be adapted to include the current.

Recall, from \eqref{Tper}, the definition of the $T$-periodization of
a path $(\bs\pi,\bs J)\in \ms S_m$, which depends only on the
restriction of the path $(\bs\pi,\bs J)$ to the time interval
$[0,T]$. Let $\chi_T \colon \ms S^T \to \ms P_\mathrm{stat}$ be the
continuous map defined by
\begin{equation*}
\chi_T(\bs\pi,\bs J) := \frac 1T \int_0^T \delta_{\vartheta_s
  (\bs\pi^T,\bs J^T)} \, ds\;,
\end{equation*}
where the translation $\vartheta_s$ acting on $\ms S$ has been
introduced above \eqref{06}.  Namely, $ \chi_T(\bs\pi,\bs J)$ is the
$T$-holonomic measure associated to the $T$-periodic path
$(\bs\pi^T,\bs J^T)$ obtained by $T$-periodizing the path
$(\bs \pi,\bs J) \in \ms S^T$.

Recall the definition of the empirical process $R_T(\bs\eta)$
introduced in \eqref{empr}. Since
$(\, \bs\pi_N(\bs\eta^T) \,,\, \bs J_N(\bs \eta^T)\, ) =
(\bs\pi_N(\bs\eta)^T, \bs J_N(\bs \eta)^T )$, for each
$\bs\eta\in D\big([0,T];\Sigma_{N,K_N}\big)$,
\begin{equation}
\label{38}
\begin{aligned}
& \mf R_{T,N}(\bs \eta) 
=\; R_T(\bs\eta) \circ (\bs
\pi_N,\bs J_N)^{-1}= \frac 1T \int_0^T \delta_{\vartheta_s \bs\eta^T} \circ (\bs
\pi_N,\bs J_N)^{-1}\, ds \\
&\quad  =\; \frac 1T \int_0^T \delta_{\vartheta_s
  \big(\bs\pi_N(\bs\eta)^T, \bs J_N(\bs \eta)^T \big)} \, ds
\;=\; \chi_T \big(\bs\pi_N(\bs\eta), \bs J_N(\bs \eta)\big)\;.
\end{aligned}
\end{equation}
Observe that the image of $\ms S^T$ by $\chi_T$ corresponds to the set
of $T$-holonomic measures.  For $\rho$ of total mass $m$, let
$\bs I_{m,T,\rho}\colon \ms P_\mathrm{stat} \to [0,+\infty]$ be the
functional defined by
\begin{equation}
\label{113}
\bs I_{m,T,\rho}(P) \; :=\;  \inf\big\{ A_{m,T,\rho} (\bs\pi ,\bs J)\,,\;
(\bs\pi,\bs J) \in \chi_T^{-1} \big( P \big)\big\} \;,
\end{equation} 
where we adopted the convention that $\inf \varnothing = +\infty$. In
particular, $\bs I_{m,T,\rho}(P) < +\infty$ only for $T$-holonomic
measures $P$.  Moreover, $\bs I_{m,T,\rho}(P)<+\infty$ only if the
$T$-periodic path $(\bs\pi,\bs J)\in \ms S$ associated to $P$
satisfies the following condition. There there exists $s\in [0,T]$
such that the restriction of $\vartheta_s(\bs\pi,\bs J)$ to $[0,T]$
belongs to $\ms S^T_{m,\mathrm{ac},\rho}$. In particular $\bs \pi (t)
=\rho$ for some $t\in\bb R$.

In view of the identity \eqref{38}, Theorem~\ref{th10}, by the
contraction principle, yields the following statement.

\begin{corollary}
\label{th11}
Fix $T>0$, $m>0$ and a density profile $\rho:\bb T^d \to [0,1]$ of
total mass $m$. 
For each sequence $(\eta^N: N\ge 1)$ associated to
$\rho$, each closed set $\ms F\subset\ms P_\mathrm{stat} $, and each 
open set $\ms G\subset\ms P_\mathrm{stat}$,
\begin{gather*}
\limsup_{N\to \infty} \frac {1}{N^d} \, \log \bb P^N_{\eta^N} \,
\big[ \mf R_{T,N} 
\in \ms F\,\big] \;\le\; -\,
\inf_{P \in \ms F}
\bs I_{m,T,\rho} (P)\;, \\
\liminf_{N\to +\infty} \frac {1}{N^d} \, \log \bb P^N_{\eta^N} \,
\big[ \mf R_{T,N} 
\in \ms G\,\big] \;\ge\; -\,
\inf_{P \in \ms G} \bs I_{m,T,\rho} (P)\;.
\end{gather*}
Moreover, $\bs I_{m,T,\rho}\colon \ms P_\mathrm{stat} \to [0,
+\infty]$ is a good rate function.
\end{corollary}

\subsection{Variational convergence of the hydrodynamical rate
  function.}

The main result of this section reads as follows.

\begin{theorem}
\label{th12}
Fix $m \in (0,1)$ and a density profile $\rho:\bb T^d \to [0,1]$ of
total mass $m$.  As $T\to+\infty$, the sequence
$T^{-1} \bs I_{m,T,\rho}$ is equi-coercive uniformly in $\rho$, and
$\Gamma$-converges uniformly in $\rho$ to the functional $\bs I_m$
introduced in \eqref{06}. That is:
\begin{itemize}

\item [(i)] For each $\ell >0$ there exists a compact
$\mc K_\ell \subset \ms P_\mathrm{stat}$ such that for any $T>1$ and
any $\rho$,
$\big\{ P : T^{-1} \bs I_{m,T,\rho}(P) \le \ell \big\} \subset \mc
K_\ell$;

\item[(ii)] For any $P\in \ms P_\mathrm{stat}$, any sequence of
density profiles $\rho_T:\bb T^d \to [0,1]$ of total mass $m$ and any
sequence $P_T\to P$,
\begin{equation*}
\liminf_{T\to \infty} \frac 1T \bs I_{m,T,\rho_T}(P_T) \ge \bs I_m
(P)\;;
\end{equation*}

\item[(iii)] If the external field $E$ is orthogonally decomposable,
  then for any $P\in \ms P_\mathrm{stat}$ and any sequence of
density profiles $\rho_T:\bb T^d \to [0,1]$ of total mass $m$, there
exists a sequence $P_T\to P$ such that
\begin{equation*}
\limsup_{T\to \infty} \frac 1T \bs I_{m,T,\rho_T}(P_T) \le \bs I_m
(P)\;.
\end{equation*}
\end{itemize}
\end{theorem}

\begin{proof} The proof is divided in three parts.

\smallskip
\noindent\emph{Equicoercivity.}
In view of the compactness of $\mc M_m(\bb T^d)$, the compact
embedding of $L^2(\bb T^d;\bb R^d)$ into $\mc H_{-p}^d$ for $p > 0$,
Ascoli-Arzel\`a theorem, and Chebyshev inequality, it is enough to
prove the following bounds. For each $T_0>0$ and each smooth vector
field $H\colon \bb T^d \to \bb R^d$ there exists constants
$C_0=C_0(T_0)$ and $C_1=C_1(T_0,H)$ such that for any $T\ge 1$ and
$\delta\in (0,1)$
\begin{gather}
\label{ec1}
E_P \Big[ \sup_{t\in [-T_0,T_0]} \big\|\bs J(t)
\big\|_{L^2(\bb T^d;\bb R^d)}^2 \Big] \le C_0 \, 
\Big[ \frac 1T \bs I_{m,T,\rho} (P) + 1 \Big]
\\
\label{ec2}
E_P \bigg[ \sup_{\substack{ t,s \in [-T_0,T_0] \\ |t-s|< \delta}}
\big| \langle \bs J(t), H \rangle - \langle\bs J (s), H \rangle
\big|^2 \bigg] \le C_1 \delta \, 
\Big[\frac 1T \bs I_{m,T,\rho} (P)+1\Big].
\end{gather}
As already remarked right after the statement of Lemma \ref{l05}, by
choosing $H=\nabla g$, the bound \eqref{ec2} provides indeed also a
control on the continuity modulus of the map $t\mapsto \bs\pi(t)$.

By the stationarity of $P$ and the argument below \eqref{06}, if
$\bs I_{m,T,\rho} (P)< +\infty$, then there exists a constant $C$
depending only on $E$ such that 
\begin{equation*}
\frac {1}{2T}\, E_P \bigg[ \int_{-T}^T
dt \int_{\bb T^d} dx\, \frac{|\bs j\,|^2}{4\,
  \sigma(\bs\rho)} 
\bigg]
\le C \Big[
\frac 1T \, \bs I_{m,T,\rho} (P) +1
\Big]
\;,
\end{equation*}
where $\bs J(t)=\int_0^t \bs j(s)\,ds$.
By Cauchy-Schwarz inequality and condition (b') on the current $\bs J$
stated at the beginning of this section,
\begin{equation*}
\sup_{t\in [-T_0,T_0]} \big\|\bs J(t)\big\|_{L^2(\bb T^d;\bb R^d)}^2
\le T_0 \int_{-T_0}^{T_0}\!dt\int \! dx \big|\bs j(t,x)\big|^2 \le
T_0 \int_{-T_0}^{T_0}\!dt\int \! dx \frac{\big|\bs
  j(t,x)\big|^2}{4\sigma(\bs \rho(t,x))} \;\cdot
\end{equation*}
Since $\bs I_{m,T,\rho} (P)< +\infty$ implies
that $P$ is a $T$-holonomic probability measure,
\begin{equation*}
E_P \bigg[ \int_{-T_0}^{T_0}\!dt\int \! dx \frac{\big|\bs
  j(t,x)\big|^2}{4\sigma(\bs \rho(t,x))} \bigg] = \frac {T_0}{T} E_P
\bigg[ \int_{-T}^{T}\!dt\int \! dx \frac{\big|\bs
  j(t,x)\big|^2}{4\sigma(\bs \rho(t,x))} \bigg] 
\end{equation*}
which completes the proof of \eqref{ec1}.

For $s<t$, Cauchy-Schwarz inequality and condition (b') yield 
\begin{equation*}
\big| \langle \bs J(t), H \rangle - \langle\bs J (s), H \rangle
\big|^2 \le (t-s) \|H\|_{L^2(\bb T^d;\bb R^d)}^2 \int_s^t\!du\int\!dx
\, \big|\bs j(u,x)\big|^2\;,
\end{equation*}
so that \eqref{ec2} is obtained by the same argument as before.

\medskip
\noindent\emph{$\Gamma$-liminf.}
Denote by $\color{bblue} \ms S^T_{m,\mathrm{ac}}$ the subset of the
paths $(\bs\pi,\bs J) \in \ms S^T$ satisfying only conditions (a'),
(b') and let $A_{m,T}$ be the action defined in \eqref{act} with the
constraint $(\bs\pi,\bs J)\in \ms S^T_{m,\mathrm{ac},\rho}$ replaced
by $(\bs\pi,\bs J)\in \ms S^T_{m,\mathrm{ac}}$. Accordingly, let
$\bs I_{m,T}\colon \ms P_\mathrm{stat}\to [0,+\infty]$ be the
functional defined by
\begin{equation*}
\bs I_{m,T}(P) := \inf\big\{ A_{m,T} (\bs\pi ,\bs J)\,,\; (\bs\pi,\bs
J) \in \chi_T^{-1} \big( P \big)\big\}\;.
\end{equation*}
By the translation invariance of $A_{m,T}$, if $\chi_T^{-1}(P)$ is not
empty (i.e., if $P$ is $T$-holonomic) then
\begin{equation*}
\bs I_{m,T}(P) = E_P \big[ A_{m,T}\big] \;.
\end{equation*}
Hence, in view of the translation invariance of $P$,
\begin{equation*}
\bs I_{m,T,\rho}(P) \ge \bs I_{m,T}(P) \;=\;
E_{P} \big[ A_{m,T} \big] \;=\; T E_{P} \big[ A_{m,1} \big]\;.
\end{equation*}

Let $(\rho_T\colon T>0)$ be an arbitrary sequence and
$\big(P_T\colon T>0 \big)\subset \ms P_\mathrm{stat}$ be a sequence
converging to $P$. By \eqref{06}, the previous displayed bound and the
lower semi-continuity of $A_{m,1}$,
\begin{equation*}
\liminf_{T\to \infty} \frac 1T \bs I_{m,T,\rho_T}(P_T) \ge
\liminf_{T\to \infty} E_{P_T} \big[ A_{m,1} \big] \ge E_{P} \big[
A_{m,1} \big] = \bs I_m(P)\;.
\end{equation*}

\noindent\emph{$\Gamma$-limsup.}
By Theorem~\ref{t:dtI}, it suffices to consider the case in which $P$
is an $S$-holonomic measure with smooth density. More precisely, we
may assume that
\begin{equation}
\label{111}
P = \frac 1S \int_0^S\delta_{\vartheta_s ({\bs\pi}^*,{\bs J}^*)} \, ds
\end{equation}
for some $S>0$, where the $S$-periodic path $({\bs\pi}^*,{\bs J}^*)$
has smooth densities $({\bs\rho}^*,{\bs j}^*)$ with ${\bs \rho}^*$
bounded away from $0$ and $1$.  Given the sequence
$(\rho_T,\, T>0 ) \subset \mc M_m(\bb T^d)$, let
$\big(\bar{\bs\pi}(t),\bar{\bs J}(t)\big)$, $t\in [0,\bar T]$ be the
path provided by Lemma~\ref{t:l4.7} with $\pi_0=\rho_T \,dx$ and
$\pi_1={\bs \pi}^*(0)$. Let also $\big(\bs\pi(t),\bs J (t)\big)$,
$t\in [0,+\infty)$ be the path defined by
\begin{equation*}
\big({\bs\pi}(t),\bs{J}(t)\big) :=
\begin{cases}
\big( \bar{\bs\pi}(t)\,,\,\bar{\bs{J}}(t) \big) & \textrm{ if }
t\in[0,\bar T], \\
\big( {\bs\pi}^*(t-\bar T) \,,\, \bar{\bs J}(\bar T) + {\bs
  J}^*(t-\bar T)\big) & \textrm{ if } t > \bar T.
\end{cases}
\end{equation*}
Note that, although not explicit in the notation, the path
$({\bs\pi},{\bs J})$ depends on $T$ via the sequence $\rho_T$.  Denote
finally by $\big({\bs\pi}^T,{\bs J}^T\big)$ the $T$-periodization, as
defined in \eqref{Tper}, of $\big({\bs\pi},{\bs J}\big)$ and by $P_T$
the associated $T$-holonomic probability, i.e.,
\begin{equation*}
P_T = \frac 1T \int_0^T\delta_{\vartheta_s ({\bs\pi}^T,{\bs J}^T)} \,
ds \;.
\end{equation*}

Since $\bar T/T\to 0$, as
$T\to\infty$ the sequence $P_T$ converges to $P$ given by \eqref{111}.
Moreover, in view of \eqref{113} and \eqref{112},
\begin{equation*}
\bs I_{m,T,\rho_T}(P_T) = A_{m,\bar T,\rho_T}\big(\bar{\bs
  \pi},\bar{\bs J}\big) + A_{m,T-\bar T} ({\bs \pi}^*, {\bs J}^*) \le
C_0 + A_{m,T-\bar T} ({\bs \pi}^*, {\bs J}^*)\;.
\end{equation*}
Hence, by the $S$-periodicity of $({\bs\pi}^*, {\bs J}^*)$,
\begin{equation*}
\limsup_{T\to \infty} \frac 1T \bs I_{m,T,\rho_T}(P_T) = \frac 1S
A_{m,S} ({\bs \pi}^*, {\bs J}^*) = \frac 1S E_P \big[ A_{m,S} \big] =
\bs I_m (P)\;,
\end{equation*}
which concludes the proof of Theorem~\ref{th12}. 
\end{proof}

\section{Large deviations and projections}
\label{s:pmt}

In this section, relying on the variational convergence proven before,
we discuss the large deviations asymptotics in the joint limit
$N\to\infty$, $T\to \infty$. In particular, we conclude the proof of
Theorem~\ref{cor01} and discuss the corresponding projections.

\begin{proof}[Proof of Theorem \ref{cor01}]
We start by considering the case in which we first perform the limit
as $T\to \infty$ and then take limit as $N\to \infty$.  The asymptotic
as $T\to \infty$ follows directly from the Donsker-Varadhan large
deviation principle for the empirical process, see
Corollary~\ref{prop1}. By \cite[Lemma 4.1]{bdgjl07} or \cite[Corollary
4.3]{mar18} the limit as $N\to \infty$ is accomplished by the
$\Gamma$-convergence of the family
$\big(N^{-d} \bs I_{N,K_N}, \, N\ge 1\big)$ that has been proven in
Theorem~\ref{th03}.  Actually, the statements in \cite{bdgjl07, mar18}
give the upper bound only for compact sets. However, the goodness of
the functional $\bs I_{N,K}$ together with the equi-coercivity in
Theorem~\ref{th03} allow to deduce the upper bound for closed sets.
	
The proof of the statement when the limit as $T\to\infty$ is carried
out after the limit as $N\to\infty$ is accomplished by the similar
argument. Indeed, the asymptotic as $N\to \infty$ follows directly
from the hydrodynamical large deviations, see Corollary~\ref{th11},
while the $\Gamma$-convergence of the family
$\big(T^{-1} \bs I_{m,T,\rho}, \, T\ge 1\big)$ has been proven,
uniformly with respect to $\rho$, in Theorem~\ref{th12}.
\end{proof}

We now discuss the level two projection and the hydrostatic limit.

\begin{proof}[Proof of Corollary~\ref{cor02}]
Recall \eqref{Is} and that
$\wp_{T,N} = \mf R_{T,N} \circ \imath_t^{-1}$ where
$\imath_t\colon \ms S\to\mc M_+(\bb T^d)$ is the map
$(\bs \pi,\bs J) \mapsto \bs \pi (t)$.  Note that $i_t$ is not
continuous since we are using the Skorohod topology. However, the map
$\ms P_{\rm stat}\ni P \mapsto P\circ \imath_t^{-1} \in \ms P\big(\mc
M_+(\bb T^d)\big)$ is continuous since, by stationary, the
$P$-probability of a jump at a time $t$ is zero.  The large deviations
asymptotic thus follows from Theorem~\ref{cor01} by the contraction
principle.
	
We now show that the zero level set of $\ms I_m$ is equal to the set
of invariant probability measures for the flow $\Phi^m$ associated to
the hydrodynamic evolution \eqref{26}. By the goodness of $\bs I_m$,
if $\wp$ lies in the zero level set of $\ms I_m$ then the infimum in
\eqref{Is} is achieved, i.e.\ there exists $P\in \ms P_\textrm{stat}$
satisfying $\bs I_m(P)=0$ and $\wp=P\circ \imath_t^{-1}$.  As follows
from \eqref{06} and \eqref{hyld} $\bs I_m(P)=0$ implies that $P$
almost surely $(\bs\pi,\bs J)$ have densities $(\bs\rho,\bs j)$ that
satisfy $\bs j = -\nabla\bs\rho +\sigma(\bs\rho) E$. Hence, the
marginal of $P$ on the first variable is concentrated on the set of
$\bs \pi$ whose density $\bs\rho$ solves \eqref{0hld} with $D=1$.  By
stationarity of $P$, this implies that $P\circ \imath_t^{-1}$ is an
invariant probability of $\Phi^m$ as claimed.
	
It remains to show that if $E$ is orthogonally decomposable then
$\ms I_m(\wp)=0$ implies $\wp=\delta_{\bar\rho \,dx}$ where $\bar\rho$
is the unique stationary solution to \eqref{0hld} with $D=1$.  As
already remarked, if $E$ is orthogonally decomposable then the
quasi-potential of the WASEP dynamics can be explicitly computed and
it is a Lypuanov functional for the hydrodynamic evolution. The
argument in \cite[Thm.~7.7]{bfg} then implies that there exists a
unique stationary solution of mass $m$ to the hydrodynamic equation
that is globally attractive, hence a unique stationary probability for
the flow $\Phi^m$ that is concentrated on the stationary solution.
\end{proof}  

We consider now the level one projection.

\begin{proof}[Proof of Corollary~\ref{th01}]
Let
$\psi\colon \ms P_{\mathrm{stat}} \to \mc M_+(\bb T^d)\times \mc
H_{-p}^d$ be the map defined by
\begin{equation*}
\psi(P) := \Big( E_P \big[ \bs \pi (t)\big] \; ,\; \frac 1t E_P \big[
\bs J(t) \big] \Big)
\end{equation*}
where we understand that $\psi$ is defined only for the probabilities
$P$ such that for any $t\in \bb R$ we have
$E_P\big[ \|\bs J(t)\|_{\mc H_{-p}^d}\big] <+\infty$.  Note that
$\psi$ does not depend on $t\neq 0$ by the stationarity of $P$.
	
Recall the definitions of $\pi_{T,N}$, $J_{T,N}$, and $\mf R_{T,N}$,
in \eqref{05}, \eqref{04}, and \eqref{Rgot}.  Recall also the relation
\eqref{pl1} according to which for each
$\bs\eta\in D(\bb R_+, \Sigma_{N})$ and $t\in \bb R$, $t\neq 0$, we
have
\begin{equation}
\label{recalled}
\psi\big(\mf R_{T,N} \big) = \int\! \mf R_{T,N}(d\bs \pi, d\bs J) \,
\Big( \bs\pi(t), \frac {\bs J(t)}{t} \Big) = \big(\pi_{T,N},
J_{T,N}\big)-\frac{1}{T}\big(0, \mathcal E_{T,N}\big)\,.
\end{equation}
Since by \eqref{djj} the error term $\mathcal E_{T,N}$ is uniformly
bounded in $N$ and $T$, it is irrelevant in the large deviations
asymptotics for $T\to +\infty$. We can therefore deduce the large
deviations for the pair $\big(\pi_{T,N}, J_{T,N}\big)$ from the large
deviations for $\psi\big(\mf R_{T,N} \big)$.

Since $\psi$ is not continuous the result does not follow directly
from Theorem \ref{cor01} and the contraction principle. However, in
the terminology of \cite[\S~4.2.2]{DZ}, it possible to approximate
$\psi$ by a sequence of continuous functions and construct
exponentially good approximations of the family
$\big( \bb P^N_\eta \circ \big(\pi_{T,N}, J_{T,N} \big)^{-1} \colon
T>0,\, N\ge 1\big)$.  We obtain in this way the result and observe
that the rate functional is given by
$I_m(\pi,J)=\inf\left\{{\bf I}_m(P): P\in \ms{P}_{\rm stat},
\psi(P)=(\pi,J)\right\}$.
\end{proof}  

The next result concerns the projection on the density for the level
one large deviations functional.

\begin{proof}[Proof of Theorem~\ref{t:I1}]
In the case $E=-\nabla U$, developing the square in formula
\eqref{hyld} we have that the cross term
\begin{equation*}
\int_0^Tdt\int_{\mathbb T^d}dx\, \frac{{\bs j}\cdot
\left(\nabla{\bs \rho} + 
\sigma({\bs \rho}) \nabla U \right)}{2\sigma({\bs
\rho})}
\end{equation*}
after an integration by parts and using the continuity equation
coincides with
\begin{equation}
\label{telescx}
\frac 12 \int_{\mathbb T^d}dx\,\left[h(\bs \rho(T))-h(\bs
\rho(0))-\left(\bs \rho(T)-\bs \rho(0)\right)U\right]\,,
\end{equation}
where $h(\rho)=\rho\log\rho +(1-\rho)\log (1-\rho)$. By stationarity
the expected value of \eqref{telescx} with respect to any
$P\in \ms P_{\rm stat}$ is zero. We have therefore that when
$E=-\nabla U$
\begin{equation}
\label{fefin}
\bs I_m(P)=\frac {1}{T}\,
E_P \bigg[ \int_{0}^T dt \int_{\bb T^d} dx\,
\Big( \frac{|\, \nabla \bs\rho
+ \sigma(\bs \rho)\nabla U\,|^2}
{4\, \sigma(\bs\rho)} +
\frac{|\, \bs j\,|^2}{4\, \sigma(\bs \rho)} \Big)
\bigg] \;.
\end{equation}
Consider a $P\in \ms P_{\rm stat}$ and call
$\wp\in \ms P\left(\mathcal M_m(\mathbb T^d)\right)$ its 1-marginal
$\wp=P\circ \imath_t^{-1}$ (see notation above formula \eqref{Is}).
Let $A\colon\mathcal M_+(\mathbb T^d)\to \ms S$ be the map that
associate to $\pi\in \mathcal M_+(\mathbb T^d) $ the element
$(\bs \pi, \bs J)\in \ms S$ defined by $\bs \pi(t)=\pi$ and
$\bs J(t)=0$ for any $t\in[0,T]$. Finally let us define
$\tilde P\in \ms P_{\rm stat}$ as $\wp\circ A^{-1}$.
	
Since the second term in the right hand side of \eqref{fefin} is non
negative and since
$P\circ \imath_t^{-1}=\tilde P\circ \imath_t^{-1}=\wp$, we deduce
$$
\bs I_m(P)\geq \bs I_m(\tilde P)=E_\wp \bigg[ \int_{\bb T^d} dx\,
\Big( \frac{|\, \nabla \rho+\sigma( \rho)\nabla U\,|^2}{4\,
\sigma(\rho)} \Big) \bigg]=E_\wp\left[\ms V_m\right]\,.
$$
We have therefore that
$$
I^{(1)}_m(\pi)=\inf \left\{E_\wp(\ms V_m)\;\ \wp\in \ms
P\left(\mathcal M_m(\mathbb T^d)\right) \,
E_\wp(\pi')=\pi\right\}=\mathrm{co}(\ms V_m)(\pi)\,,
$$
the last equality follows since in the middle we have one of the
possible definitions of convex hull.  Since $\sigma$ is concave, in
the case $\nabla U=0$ we have that $\ms V_m$ is convex and therefore
$\mathrm{co}(\ms V_m)=\ms V_m$.
\end{proof}

Finally, we give a sketch of the proof of Theorem \ref{tdpt}. This is
based on analysis in \cite{bdgjl06, bdgjl07, bd} and we just show how
to deduce the result based on the arguments there.

\begin{proof}[Proof of Theorem~\ref{tdpt}]

Let us call $(\bs \pi^*, \bs J^*)$ the element of $\ms S_m$ defined by
$\bs \pi^*(t)=m$ and $\bs J^*(t)=jt$.  The result is obtained by the
analysis of the action functional $A_{m,T}$ \eqref{hyld}. In the case
$(i)$ for $E\leq E_0$ by \cite{bdgjl06, bdgjl07, bd} we have that for
any $(\bs \pi, \bs J)\in \ms S_m$ such that $\bs J(T)=jT$ it holds
$A_{m,T}(\bs \pi^*, \bs J^*)\leq A_{m,T}(\bs \pi, \bs J)$ and this
allows to deduce that $\delta_{(\bs \pi^*, \bs J^*)}$ is the minimizer
in \eqref{vfc}.

In the case $(ii)$ for $E>E_1$ it is possible to construct
\cite{bdgjl06, bdgjl07, bd} a time dependent
$(\bs \pi, \bs J)\in \ms S_m$, that has indeed the structure of a
traveling wave, such that $\bs J(T)=jT$ and
$A_{m,T}(\bs \pi^*, \bs J^*)> A_{m,T}(\bs \pi, \bs J)$. Considering
$P\in \ms P_{\rm stat}$ defined by
$P=\frac 1T\int_0^T dt\delta_{\theta_t(\bs \pi, \bs J)}$ we have
therefore that $\bs I_m(P)<\bs I_m(\delta_{(\bs \pi^*, \bs
J^*)})$. 
\end{proof}

\section{Uniqueness of periodic solutions}
\label{sec06}

Fix $T>0$.  Throughout this section,
$F: \bb R \times \bb T^d \to \bb R^d$ is a smooth, $T$-periodic vector
field. We investigate in this section the asymptotic behavior of
solutions to the Cauchy problem
\begin{equation}
\label{6.04}
\left\{
\begin{aligned}
&  \partial_t \bs u \;=\; \Delta \bs u \,
+ \, \nabla \cdot [\, \sigma (\bs u) \,
F \,] \\
& \bs u(0,\cdot) = u_0(\cdot)\;,
\end{aligned}
\right.
\end{equation}
where the initial condition $u_0 : \bb T^d \to [0,1]$ is such that
$0\le u_0(x) \le 1$ for all $x\in \bb T^d$.

Existence of weak solutions is provided by the hydrodynamic limit of
WASEP. This argument shows that the solution takes value in the
interval $[0,1]$. These bounds can be derived also from the maximum
principle and the observation that $\sigma(1) = \sigma (0)=0$.  By
parabolic regularity, a weak solution is smooth in
$(0,\infty) \times \bb T^d$.  Uniqueness is derived as in \cite[Lemma
7.2]{flm}. The proof of this lemma yields that the $L^1(\bb T^d)$ norm
of the difference of two weak solutions does not increase in time. The
main result of this section strengthen this lemma and asserts that the
$L^1(\bb T^d)$ distance of two different weak solutions decreases in
time. It reads as follows.

\begin{theorem}
\label{up01}
Let $F: \bb R \times \bb T^d \to \bb R^d$ be a smooth, $T$-periodic
vector field.  For each $m\in [0,1]$, the equation
\begin{equation*}
\partial_t \bs u \;=\; \Delta \bs u \, + \, \nabla \cdot [\, \sigma (\bs u) \,
F \,]\;. 
\end{equation*}
admits a unique $T$-periodic solution
$\bs u: \bb R \times \bb T^d \to [0,1]$ such that
$\int_{\bb T^d} \bs u(t,x)\, dx =m$ and $0\le \bs u(t,x) \le 1$ for all
$t$. This solution is represented by $\bs u^{(m)}$.  Moreover, for each
$u_0 : \bb T^d \to [0,1]$ such that $\int_{\bb T^d} u_0(x)\, dx =m$,
$0\le u_0(x) \le 1$ for all $x\in\bb T^d$, the unique weak solution of
\eqref{6.04} converges to $\bs u^{(m)}$ in $L^1(\bb T^d)$ as $t\to\infty$.
\end{theorem}

The proof of this result relies on a method of coupling of two
diffusions due to Lindvall and Rogers \cite{lr86}.

\subsection{Coupling diffusions.}
Let $G: \bb R_+ \times \bb T^d \to \bb R^d$ be a smooth vector field
uniformly bounded: there exists $C_0<\infty$ such that
$\sup_{(t,x) \in \bb R_+ \times \bb T^d} \Vert \, G(t,x)\,\Vert \le
C_0$.

Denote by $\ms L_t$ the time-dependent generator
\begin{equation}
\label{6.06}
\ms L_t f \;=\;   \Delta f \, + \, \nabla f \cdot
G_t \;,\quad f \,\in\, C^2(\bb T^d)\;.
\end{equation}
Let $\color{bblue}(Z^x_t : t\ge 0)$, $x\in \bb T^d$, be the
$\bb T^d$-valued, continuous-time Markov process whose generator is $\ms
L_t$ and which starts from $x$.

Recall that a coupling between $Z^x_t$ and $Z^y_t$ is a process
$(\widetilde Z^x_t, \widetilde Z^y_t)$ whose first (resp. second)
coordinate evolves as $Z^x_t$ (resp. $Z^y_t$). The coupling time,
denoted by $\tau^Z_{x,y}$, is the first time at which the processes
meet:
\begin{equation*}
\tau^Z_{x,y} \;:=\; \inf\big\{\, t>0 : \widetilde Z^{x}_t \,=\,
\widetilde Z^{y}_t \,\big\}\;. 
\end{equation*}

Next result relies on the Lindvall--Rogers coupling \cite{lr86}.

\begin{proposition}
\label{ul01}
There exist constants $A<\infty$ and $\lambda>0$, which depends only
on $\sup_{(t,x) \in \bb R_+ \times \bb T^d} \Vert\, G(t,x)\,\Vert$, and, for
each $x$, $y\in\bb T^d$, a coupling between $Z^x_t$, $Z^y_t$, such
that
\begin{equation*}
\sup_{x, y\in \bb T^d} P\big[\, \tau^Z_{x,y} \,\ge  \, t \, \big]
\; \le \; A\, e^{-\lambda t}   
\end{equation*}
for all $t\ge 0$.
\end{proposition}

Let $\color{bblue} \bb P_{x}$ be the probability measure on
$C(\bb R_+, \bb T^d)$ induced by the diffusion associated to the
generator $\ms L_t$ starting from $x\in \bb T^d$.  Expectation with
respect to $\bb P_{x}$ is represented by $\color{bblue} \bb E_{x}$.

\begin{corollary}
\label{up02}
There exist constants $A<\infty$ and $\lambda>0$, which depends only
on $\sup_{(t,x) \in \bb R_+ \times \bb T^d} \Vert\, G(t,x)\,\Vert$,  such that
\begin{equation*}
\sup_{x, y\in \bb T^d}
\Big|\, \bb E_x\big[\, f(Z_t)\, \big]
\,-\, \bb E_y\big[\, f(Z_t)\, \big]\, \Big|
\; \le \; A \, e^{-\lambda t} \, \Vert f\Vert_\infty   
\end{equation*}
for every $f\in C(\bb T^d)$.
\end{corollary}

\begin{proof}
Since the difference may be written as
\begin{equation}
\label{6.07}
\Big|\, E\big[\, f(\widetilde Z^x_t)\,
-\, f(\widetilde Z^y_t)\, \big]\, \Big| \;\le\;
2\, \Vert f\Vert_\infty\, P\big[\, \tau^Z_{x,y} \,\ge \, t \, \big]\;,
\end{equation}
the assertion is an elementary consequence of Proposition \ref{ul01}.
\end{proof}

\begin{proof}[Proof of Proposition \ref{ul01}]
Denote by $\color{bblue}(W_t : t\ge 0)$, the standard Brownian motion
on $\bb R^d$.  Let
$\color{bblue} b: \bb R_+ \times \bb R^d \to \bb R^d$ be the
$\bb T^d$-periodic vector field whose restriction to $\bb T^d$
coincides with $G$. Denote by $X^x_t$ the solutions of the SDE
\begin{equation*}
\left\{
\begin{aligned}
& dX^x_t \;=\;  c(t,X^x_t)\, dt \,+\, dW_t \;, \\
& X^x_0 = x\;,
\end{aligned}
\right.
\end{equation*}
where $\color{bblue} c(t,x) = (1/2) \, b(t,x)$.  For each $x\in \bb R^d$,
$X^x_t$ is a diffusion on $\bb R^d$ whose time-dependent generator,
denoted by $\ms A_t$, is given by
\begin{equation*}
\ms A_t f \;=\; (1/2)\,  
\Delta f \, + \, \nabla f \cdot  c_t  \;,\quad
f \,\in\, C^2_0(\bb R^d)\;,
\end{equation*}
where $C^2_0(\bb R^d)$ stands for the twice continuously
differentiable functions with compact support. We replaced $b(t,x)$ by
$c(t,x)$ in order to have a simple relation between the generators
$\ms A_t$ and $\ms L_t$.

Fix $x$, $y\in \bb T^d$. Lindvall and Rogers \cite{lr86} provide a
coupling between $X^x_t$ and $X^y_t$, represented by
$(\widetilde X^x_t, \widetilde X^y_t)$, such that, before hitting the
origin, $D_t := \Vert \widetilde X^x_t - \widetilde X^y_t\Vert$
evolves as
\begin{equation}
\label{6.09}
dD_t \;=\; 2\, dB_t \;+\; \big\< \frac{\widetilde X^x_t - \widetilde
X^y_t}{D_t} \, ,\, c(t, \widetilde X^x_t) - c(t, \widetilde X^y_t) \,
\big \> \, dt\;,
\end{equation}
where $\< \,\cdot \,,\, \cdot\, \>$ stands for the scalar product in
$\bb R^d$, and $B_t$ is a one-dimensional Brownian motion. Note that
the drift term is bounded.

Denote by $(\widetilde Z^x_t , \widetilde Z^y_t)$ the projection of
the process $(\widetilde X^x_t , \widetilde X^y_t)$ on
$\bb T^d \times \bb T^d$. Each coordinate of the pair
$(\widetilde Z^x_t , \widetilde Z^y_t)$ is a Markov process whose
generator is equal to $(1/2) \ms L_t$. Hence,
$(\widetilde Z^x_{2t} , \widetilde Z^y_{2t})$ is a coupling to
$Z^x_t$, $Z^y_t$, and, to prove the proposition, it is enough to show
that there exist constants $A<\infty$ and $\lambda>0$, which depends
only on
$\sup_{(t,x) \in \bb R_+ \times \bb T^d} \Vert\, G(t,x)\,\Vert$,  such
that
\begin{equation}
\label{6.10}
\sup_{x, y\in \bb T^d} P\big[\, \tau^{\widetilde Z}_{x,y} \,\ge  \, t \, \big]
\; \le \; (1/2)\, A\, e^{-\lambda t}   
\end{equation}
for all $t\ge 0$, where $\tau^{\widetilde Z}_{x,y}$ is the first time
the processes $(\widetilde Z^x_t , \widetilde Z^y_t)$ meet.

By construction, before hitting $0$,
$\widetilde D_t := \Vert \widetilde Z^x_t - \widetilde Z^y_t\Vert$
evolves as $D_t$, except when $D_t$ attains the maximal distance
between two points in $\bb T^d$, that is, when $D_t$ hits
$\color{bblue} L:=\sqrt{d}/2$, in which case $\widetilde D_t$ is
reflected, while $D_t$ evolves according to \eqref{6.09}.

Let
$\color{bblue} M:= 2 \sup_{(t,x)\in \bb R_+ \times \bb R^d} \Vert\,
c(t,x) \,\Vert = \sup_{(t,x)\in \bb R_+ \times \bb R^d} \Vert\, b(t,x)
\,\Vert$. By definition of $b$,
$M = \sup_{(t,x)\in \bb R_+ \times \bb T^d} \Vert\, G(t,x)
\,\Vert$. Moreover, for all $z$ such that $\Vert z\Vert = 1$,
$|\, \< z \, ,\, c(t, \widetilde X^x_t) - c(t, \widetilde X^y_t) \,
\big \>\,| \le M$.

Let $\widehat D_t$ the diffusion on $[0,L+1]$ which is absorbed at the
origin, reflected at $L+1$ and which evolves according to the SDE
\begin{equation*}
d \widehat D_t \;=\; 2\, dB_t \;+\; M \, dt\;.
\end{equation*}
By the previous bound on the drift term of $D_t$, we may couple
$\widetilde D_t$ and $\widehat D_t$ in such a way that
$\widetilde D_t \le \widehat D_t$ for all $t\ge 0$, almost surely,
provided $\widetilde D_0 \le \widehat D_0$. In particular,
$\widetilde D_0$ hits the origin before $\widehat D_0$. Therefore, the
coupling time of $(\widetilde Z^x_t , \widetilde Z^y_t)$ is bounded
above by the absorption time of $\widehat D_t$, represented by
$H^r_0$, where $r$ stands for the initial state.
An elementary computation yields that there exists a finite constant
$T_0$, depending only on $M$ and $L$ such that
\begin{equation*}
\sup_{r\in [0,L+1]} E [\, H^r_0 \,] \;\le\; T_0
\quad\text{so that}\quad \sup_{r\in [0,L+1]} P [\, H^r_0 > 2T_0 \,]
\;\le\; \frac{1}{2}\;\cdot
\end{equation*}
In consequence, 
\begin{equation*}
\sup_{x,y\in \bb T^d}
P\big[\, \tau^{\widetilde Z}_{x,y} \,
\ge \, 2T_0\, \big]\;\le\; \frac{1}{2}  \;\cdot
\end{equation*}
To complete the proof of \eqref{6.10} [and the one of the
proposition], it remains to apply the Markov property.
\end{proof}

\subsection{Asymptotic behavior of linear parabolic equations.}
Let $G: \bb R_+ \times \bb T^d \to \bb R^d$ be a smooth vector field
satisfying the hypotheses stated in the previous subsection.

\begin{proposition}
\label{up03}
Fix two probability densities $w_1$, $w_2$ on $\bb T^d$,
$w_j \colon \bb T^d\to \bb R_+$, $\int_{\bb T^d} w_j(x)\, dx =1$,
$j=1$, $2$. Denote by
$\bs w_j \colon \bb R_+ \times \bb T^d \to \bb R_+$ the unique weak
solution of the linear parabolic equation
\begin{equation}
\label{6.05}
\left\{
\begin{aligned}
&  \partial_t \bs w_j \;=\; \Delta \bs w_j
\, - \, \nabla \cdot [\, \bs w_j \, G \,] \\
& \bs w_j(0,\cdot) = w_j(\cdot)
\end{aligned}
\right.
\end{equation}
Then, there are $A<\infty$ and $\lambda>0$, which depends
only on $\sup_{(t,x) \in \bb R_+ \times \bb T^d} \Vert\, G(t,x)\,\Vert$, such
that
\begin{equation*}
\int \big|\,  \bs w_2(t,x) \,-\, \bs w_1(t,x) \, \big|\, dx \;\le\; A\,
e^{-\lambda t}
\end{equation*}
for all $t\ge 0$.
\end{proposition}

\begin{proof}
Recall the definition of the diffusions $(Z^x_t : t\ge 0)$, $x\in \bb
T^d$, introduced in the previous subsection. Denote its transition
probability by $p_t(x,dy) = p_t(x,y) dy$ so that
\begin{equation*}
\bb E_x[\, f(Z_t)\,] \;=\; \int p_t(x,y) \, f(y)\, dy
\end{equation*}
for all functions $f\in C(\bb T^d)$.

Since $\ms L_t$ is the generator of the diffusion $Z_t$,
for every function $f$ in $C^2(\bb T^d)$ and $t>0$,
\begin{equation*}
\bb E_x [\, f(Z_t)\,] \;=\; f(x) \;+\; \int_0^t \bb E_x [\, (\ms L_s
f) (Z_s)\,]\; ds\;.
\end{equation*}
integrating both sides of this identity with respect to $w_j(x) dx$
and integrating by parts yields that $\bs v_j(t,x) := \int_{\bb T^d}
w_j(y) p_t(y,x) dy$ solves \eqref{6.05} with initial condition
$\bs v_j(0,x) = w_j(x)$. By the uniqueness of weak solutions, 
\begin{equation*}
\bs w_j(t,x) \;=\; \int_{\bb T^d} w_j(y) \, p_t(y,x) \, dy \;.
\end{equation*}

Since
$\int \bs w_j(t,x) \, f(x)\, dx = \int w_j(x) \, \bb E_x[\, f(Z_t) \,] \,
dx $ for every continuous function $f: \bb T^d \to \bb R$,
as $w_j(x)$ is a probability density,
\begin{equation*}
\begin{aligned}
& \int_{\bb T^d} \bs w_2(t,x) \, f(x)\, dx \;-\;
\int_{\bb T^d} \bs w_1(t,x) \, f(x)\, dx  \\
&\quad \;=\;
\int_{\bb T^d} dx \, w_1(x) \, \int_{\bb T^d} dy \,
w_2(y) \,  \Big\{ \bb E_x\big[\, f(Z_t) \,\big]  \;-\;
\bb E_y\big[\, f(Z_t) \,\big] \,\Big\}\;.
\end{aligned}
\end{equation*}
Therefore, by \eqref{6.07}, for every $t\ge 0$,
\begin{equation*}
\sup_{f}\, \Big|\, \int \bs w_2(t,x) \, f(x)\, dx \;-\;
\int \bs w_1(t,x) \, f(x)\, dx \,\Big| \;\le \;
2\, \sup_{x, y\in \bb T^d}\, 
P\big[\, \tau^Z_{x,y} \,\ge \, t \, \big]\;,
\end{equation*}
where the supremum is carried over all continuous function $f$ such
that $\Vert f\Vert_\infty\le 1$. Hence,
\begin{equation*}
\int \big|\,  \bs w_2(t,x) \,-\, \bs w_1(t,x) \, \big|\, dx  \;\le \;
2\, \sup_{x,y\in \bb T^d}\, 
P\big[\, \tau^Z_{x,y} \,\ge \, t \, \big]\;,
\end{equation*}
and the assertion of the proposition follows from Proposition
\ref{ul01}.
\end{proof}

We turn to Theorem \ref{up01} whose proof relies on the following
estimate.

\begin{proposition}
\label{ul04}
There exist constants $A<\infty$ and $\lambda>0$, which depends only
on $\sup_{(t,x) \in [0,T] \times \bb T^d} \Vert\, F(t,x)\,\Vert$, with
the following property.  Fix $0<m<1$ and
$u_j \colon \bb T^d\to [0,1]$, $j=1$, $2$, such that
$\int_{\bb T^d} u_j(x)\, dx =m$. Denote by
$\bs u_j \colon \bb R_+ \times \bb T^d \to \bb R_+$ the unique weak
solution of \eqref{6.04} with initial condition $u_j$.  Then,
\begin{equation*}
\int \big|\,  \bs u_2(t,x) \,-\, \bs u_1(t,x) \, \big|\, dx \;\le \; A\,
e^{-\lambda t}
\end{equation*}
for all $t\ge 0$.
\end{proposition}

\begin{proof}
Let $\bs v(t,x)=\bs u_2(t,x)-\bs u_1(t,x)$, so that
$\int_{\bb T^d} \bs v(t,x)\, dx =0$ for all $t\ge 0$.  Since
$\sigma(b) - \sigma(a) = (b-a)(1-a-b)$, $\bs v(t,x)$ solves the linear
equation
\begin{equation}
\label{6.08}
\partial_t \bs w \;=\; \Delta \bs w \, + \, \nabla \cdot [\, \bs w \,
G \,]\;,
\end{equation}
where $G$ is the vector field $G= (1-\bs u_1-\bs u_2) F$.

Let $v_0\colon \bb T^d \to \bb R$ be given by
$v_0(x) = u_2(x) - u_1(x)$. Denote by $v^+$, $v^-$, the positive,
negative part of $v_0$, respectively. Note that
$\int_{\bb T^d} v^+(x) dx = \int_{\bb T^d} v^-(x) dx =: m' \in [0,
m]$. If $m'=0$, $0 = v_0(x) = u_2(x) - u_1(x)$, and there is
nothing to prove. Assume that $m' >0$ and let $w_2(x) = v^+(x)/m'$,
$w_1(x) = v^-(x)/m'$ so that $w_j$ is the density of a probability
measure on $\bb T^d$.

Denote by $\bs w_j(t,x)$ the solution of \eqref{6.08} with initial
condition $w_j(x)$. By linearity $m'[\bs w_2(t,x) - \bs w_1(t,x)]$ solves
\eqref{6.08} with initial condition $m'[w_2(x) - w_1(x)] =
v_0(x)$. Since $\bs v(t,x)$ solves the same Cauchy problem, $\bs v(t,x) =
m'[\bs w_2(t,x) - \bs w_1(t,x)]$. Thus, as $m'\le 1$,
\begin{equation*}
\int \big|\,  \bs u_2(t,x) \,-\, \bs u_1(t,x) \, \big|\, dx \;=\;
\int \big|\,  \bs v(t,x) \, \big|\, dx \;\le\; 
\int \big|\,  \bs w_2(t,x) \,-\, \bs w_1(t,x) \, \big|\, dx \;.
\end{equation*}

To complete the proof, it remains to recall the statement of
Proposition \ref{up03}, and to observe that
$\sup_{(t,x) \in \bb R_+ \times \bb T^d} \Vert\, G(t,x)\,|\Vert \le
\Vert \, F\, \Vert_\infty$ because $\bs u_j(t,x)$ takes value in the
interval $[0,1]$.
\end{proof}

From this result we can deduce the first assertion of Theorem
\ref{up01}. 

\begin{corollary}
\label{ul05}
For each $m\in (0,1)$, the equation \eqref{6.04} admits a unique
$T$-periodic solution $\bs u: \bb R \times \bb T^d \to [0,1]$.
\end{corollary}

\begin{proof}
Fix $m\in (0,1)$ and denote by $L^1_m(\bb T^d)$ the closed subspace of
$L^1(\bb T^d)$ defined by
$\color{bblue} L^1_m(\bb T^d) = \{ u \in L^1(\bb T^d) : \int_{\bb T^d}
u(x) \, dx = m \,,\, 0\le u(x) \le 1\,\}$.

Define the operator $\mf P: L_m^1(\bb T^d) \to L_m^1(\bb T^d)$ given
by $\mf P (u) = \bs u(T, \cdot)$, where $\bs u(t,x)$ is the weak
solution of \eqref{6.04} with initial condition $u(\cdot)$. Let
$u_0: \bb T^d \to [0,1]$ be given by $u_0(x) = m$ for all $x$, and set
$u_{j+1} = \mf P u_j$, $j\ge 0$. We claim that the sequence
$(u_j : j\ge 1)$ is Cauchy in $L^1(\bb T^d)$. Fix $n$, $j\ge 1$. Since
$\mf P_{n+j} u = \mf P_{n}\mf P_{j}u$, by Proposition \ref{ul04},
\begin{equation*}
\big\Vert\, \mf P_{n+j} u \,-\, \mf P_{n} u\, \big\Vert_1 \;=\;
\big\Vert\, \mf P_{n} [\, \mf P_{j} u \,-\, u\,] \, \big\Vert_1
\;\le\; A e^{-\lambda n T}\;.
\end{equation*}

Denote by $w$ the limit in $L^1$ of the sequence $u_j$, and observe
that $\mf P w = w$. This proves that the solution of equation
\eqref{6.04} with initial condition $w$ is $T$-periodic. By
Proposition \ref{ul04}, such $T$-periodic solution is unique.
\end{proof}

\begin{proof}[Proof of Theorem \ref{up01}]
Fix $m\in [0,1]$. As the result is trivial for $m=0$ or $1$, we may
assume that $0<m<1$. In this range, the assertions of the theorem
corresponds to the ones of Proposition \ref{ul04} and Corollary
\ref{ul05}. 
\end{proof}

\appendix

\section{Dynamical bounds}
\label{sec04}

We present in this section some estimates used in the article. Let
\begin{equation*}
c_{x,y}(\eta) \;=\; \eta_x\, [ 1-\eta_y] \,
e^{(1/2)\,  E_N(x,y)}\;, \quad (x,y)\,\in\, \bb B_N\;
\end{equation*}
and recall the notation introduced in \eqref{26}.

\begin{lemma}
\label{lem01}
Given a set of bounded functions
$\phi^{x,y}_s: D([0,\infty), \Sigma_N) \to \bb R$,
$(x,y) \in \bb B_N$, progressively measurable, the process
\begin{equation}
\label{02}
\bb M^\phi_t \;=\; \exp \sum_{(x,y) \in \bb B_N}\Big\{  \int_0^t
\phi^{x,y}_s \, \ms N^{x,y}_{(s,s+ds]}(\bs \eta) \,-\, N^2 \int_0^t c_{x,y}(\bs \eta(s))
\{e^{\phi^{x,y}_s} - 1\} \, ds \Big\}
\end{equation}
is a mean one, positive martingale with respect to $\bb P_\eta^N$ for
any configuration $\eta\in \Sigma_N$.
\end{lemma}

The proof of this lemma is similar to the one of Proposition A.2.6 in
\cite{kl} and left to the reader. This martingale corresponds to the
Radon-Nikodym derivative (restricted to the interval $[0,t]$) of the
law of a jump process with rates $c_{x,y}(\eta) e^{\phi^{x,y}} $ with
respect $\bb P_\eta^N$.

Recall that the symmetric simple exclusion process is the Markov chain
on $\Sigma_N$ whose generator is $L_N$, introduced in \eqref{22}, with
$E \equiv 0$. Denote by $\color{bblue} \nu_\alpha$,
$0\le \alpha\le 1$, the Bernoulli product measure on $\Sigma_N$ with
density $\alpha$, and by $\color{bblue} \bb P^0_{\nu_\alpha}$ the
probability measure on $D(\bb R_+, \Sigma_N)$ induced by the symmetric
simple exclusion process starting from $\nu_\alpha$.  Expectation with
respect to $\bb P^0_{\nu_\alpha}$ is represented by
$\color{bblue} \bb E^0_{\nu_\alpha}$.

\begin{lemma}
\label{lem05}
For all $1<p<\infty$ and $E\in C^1(\bb T^d ; \bb R^d)$, there exists a
constant $C_p$ such that
\begin{equation*}
\log\, \bb E^0_{\nu_{1/2}} \Big[ \, \Big( \frac{d \bb P^N_{\mu_{N,K}}}
{d \bb P^0_{\nu_{1/2}}} \Big|_{[0,T]}\, \Big)^p \, \Big]
\;\le\; C_p \, (1\,+\, T) \, N^d
\end{equation*}
for all $T>0$, $N\ge 1$, $0\le K\le N^d$.
\end{lemma}

\begin{proof}
The proof reduces to a standard computation of exponential
martingales. To emphasize the dependence of the measure
$\bb P^N_{\mu_{N,K}}$ on the external field $E$, in this proof, we
represent the measure $\bb P^N_{\mu_{N,K}}$ by $\bb
P^E_{\mu_{N,K}}$. Clearly,
\begin{equation*}
\bb E^0_{\nu_{1/2}} \Big[ \, \Big( \frac{d \bb P^E_{\mu_{N,K}}}{d
  \bb P^0_{\nu_{1/2}}} \Big|_{[0,T]}\, \Big)^p \, \Big] \;=\;
\bb E^{0}_{\mu_{N, K}} \Big[ \,
\Big( \frac{d \mu_{N,K}}
{d \nu_{1/2}} \Big)^{p-1}\,
\Big( \frac{d \bb P^E_{\mu_{N,K}}}
{d \bb  P^0_{\mu_{N,K}}} \Big|_{[0,T]}\, \Big)^p \, \,
\Big]\;.
\end{equation*}
As $\nu_{1/2} (\eta) = (1/2)^{N^d}$, this expression is bounded by
\begin{equation*}
2^{(p-1) N^d}\, \bb E^{0}_{\mu_{N, K}} \Big[ \,
\Big( \frac{d \bb P^E_{\mu_{N,K}}}
{d \bb  P^0_{\mu_{N,K}}} \Big|_{[0,T]}\, \Big)^p \, \,
\Big]\;.
\end{equation*}
On the other hand,
\begin{equation*}
\bb E^0_{\mu_{N,K}} \Big[ \, \Big( \frac{d \bb P^E_{\mu_{N,K}}}{d
  \bb P^0_{\mu_{N,K}}} \Big|_{[0,T]}\, \Big)^p \, \Big] \;=\;
\bb E^{pE}_{\mu_{N, K}} \Big[ \, \Big( \frac{d \bb P^E_{\mu_{N,K}}}
{d \bb  P^0_{\mu_{N,K}}} \Big|_{[0,T]}\, \Big)^p \; 
\frac{d \bb  P^0_{\mu_{N,K}}}{d \bb P^{pE}_{\mu_{N, K}}} \Big|_{[0,T]}\, \,
\Big]\;.
\end{equation*}
Note that in the last expectation the external field is $pE$. A direct
computation based on the explicit formula for the Radon-Nikodym
derivatives provided by Lemma \ref{lem01} yields that
\begin{equation*}
\Big\Vert\, \Big( \frac{d \bb P^E_{\mu_{N,K}}}
{d \bb  P^0_{\mu_{N,K}}} \Big|_{[0,T]}\, \Big)^p \; 
\frac{d \bb  P^0_{\mu_{N,K}}}{d \bb P^{pE}_{\mu_{N, K}}}
\Big|_{[0,T]}\, \Big\Vert_\infty \;\le\; e^{ C_p\, T\, N^d}
\end{equation*}
for some finite constant $C_p$, see \cite[Lemma~4.5]{bfg}.
\end{proof}

Until the end of the appendix, fix $T>0$, $m\in (0,1)$ and a sequence
$(K_N : N\ge 1)$ such that $K_N/N^d\to m$.  Consider a progressively
measurable, continuous function
$w\colon \bb R \times \bb T^d \times \mc M_{+}(\bb T^d) \times D(\bb R; \mathcal H^d_{-p})\to \bb R^d$ with support on
$[0,T] \times \bb T^d \times \mc M_{+}(\bb T^d) \times D(\bb R;\mathcal H_{-p}^d)$. Recall from \eqref{23} the definition of
the progressively measurable function
$G_w\colon \bb R \times \bb T^d \times \ms S_{m, {\rm ac}} \to \bb
R^d$. For $\epsilon>0$ and a cylinder function $\Psi$, let
\begin{equation*}
F^{w,\Psi}_{N,\epsilon}(t, \bs \eta) \;=\; \frac 1 {N^d} \sum_{x\in \bb
  T^d_N} G_w\Big(t\,,\, x \,,\, \bs \pi_N \,,\,
\bs J_N \Big) \, \big\{ \, (\tau_x \Psi)(\eta(t)) -
\widehat \Psi ((\bs \pi_N^\epsilon) (t,x)) \, \big\} \;,
\end{equation*}
where $(\bs \pi^\epsilon_N) (t,x) = (2\epsilon)^{-d} \bs \pi_N (t,
[x-\epsilon , x+\epsilon]^d)$, and $\widehat \Psi : [0,1] \to \bb R$
is the function given by
\begin{equation*}
\widehat \Psi (\alpha) \;=\; E_{\nu_\alpha} \big[\, \Psi\,\big]\;.
\end{equation*}

\begin{lemma}
\label{lem1}
For all $\delta>0$
\begin{equation*}
\lim_{\epsilon \to 0} \limsup_{N\to\infty} \frac 1{N^d} \log
\bb P^N_{\mu_{N,K_N}} \Big [ \, \Big| \int_0^T
F^{w,\Psi}_{N,\epsilon}(t, \bs \eta ) \, dt \, \Big| > \delta \, \Big] \;=\;
-\infty\;.
\end{equation*}
\end{lemma}

\begin{proof}
First, we claim that for all $\delta>0$,
\begin{equation*}
\lim_{\epsilon \to 0} \limsup_{N\to\infty} \frac 1{N^d} \log
\bb P^0_{\nu_{1/2}} \Big [ \, \Big| \int_0^T
F^{w,\Psi}_{N,\epsilon}(t, \bs \eta ) \, dt \, \Big| > \delta \, \Big] \;=\;
-\infty\;.
\end{equation*}
We refer to \cite[Theorem 10.3.1]{kl} for the proof in the case in
which $w$ does not depend on $\pi$ and $\bs J$. The arguments to
include this dependence are tedious, but straightforward and left to the
reader. The extension to the measure $\bb P^N_{\mu_{N,K_N}}$ follows
from Schwarz inequality and Lemma \ref{lem05}.
\end{proof}

Next result is a consequence of the entropy inequality,
\cite[Proposition A1.8.2]{kl}, and the previous lemma.

\begin{corollary}
\label{cor2}
Let $(\bb Q_N : N \ge 1)$ be a sequence of probability measures in
$\ms P^{N,K_N}_{\rm stat}$. Assume that there exists a finite constant
$C_0$ such that for all $S>0$,
\begin{equation*}
\limsup_{N\to\infty} \frac{1}{N^d}\,
\bb H^{(S)} \big(\, \bb Q_N \,|\, \bb P^N_{\mu_{N,K_N}}\,\big)
\;\le\; C_0\, S\;.
\end{equation*}
Then, for all $\delta>0$
\begin{equation*}
\lim_{\epsilon \to 0} \limsup_{N\to\infty} 
\bb Q_N \Big [ \, \Big| \int_0^T
F^{w,\Psi}_{N,\epsilon}(t, \bs \eta ) \, dt \, \Big|
> \delta \, \Big] \;=\; 0\;.
\end{equation*}
\end{corollary}

Fix a vector field $F$ in $C^1(\bb R \times \bb T^d; \bb R^d)$ with
compact support in $(0,T) \times \bb T^d$, $a>0$, and recall the
definition of $\mc E_{a, \epsilon} (F, \bs \pi)$,
$\mc V_{a, \epsilon} (F, \bs \pi,\bs J)$, $\epsilon>0$ in \eqref{25}.

\begin{lemma}
\label{lem04}
There exist finite, positive constants $a$ and $C_0$ such that for all
vector fields $F$ in $C^1(\bb R \times \bb T^d; \bb R^d)$ with compact
support in $(0,T) \times \bb T^d$,
\begin{gather*}
\limsup_{\epsilon \to 0} \limsup_{N\to\infty} \frac 1{N^d} \,\log
\bb E^N_{\mu_{N,K_N}} \Big[ \exp 
\big\{N^d \, \mc E_{a, \epsilon} (F, \bs \pi_N) \, \big\}\, \Big]
\;\le\; C_0 \, (1+T) \;, \\
\limsup_{\epsilon \to 0} \limsup_{N\to\infty} \frac 1{N^d} \,\log \bb
E^N_{\mu_{N,K_N}} \Big[ \, \exp \big\{ N^d\,
\mc V_{a,\epsilon} (F,\bs \pi_N,\bs J_N )
\big\} \, \Big] \;\le\; C_0 \, (1+T)  \;.
\end{gather*}
\end{lemma}

\begin{proof}
We claim that there exists a finite constant $a$ such that for any
$T>0$ and any $F$ in $C^1(\bb R \times \bb T^d; \bb R^d)$
with compact support in $(0,T) \times \bb T^d$,
\begin{gather*}
\limsup_{\epsilon \to 0} \limsup_{N\to\infty} \frac 1{N^d} \,\log
\bb E^0_{\nu_{1/2}} \Big[ \exp N^d \,
\big\{\, \mc E_{a, \epsilon} (F, \bs \pi_N) \, \big\}\, \Big]
\;\le\; 0 \;.
\end{gather*}
This statement is proven in \cite[Section 3.2]{blm} (see the proof of
the bound presented in the last displayed equation at page 2367). To
deduce the first assertion of the lemma from this result, it suffices
to apply Schwarz inequality, to recall the statement of Lemma
\ref{lem05}, and to observe that
$2 \mc E_{a, \epsilon} (F, \bs \pi_N) = \mc E_{a/2, \epsilon} (2F, \bs
\pi_N)$.

We turn to the second assertion of the lemma. By Lemma \ref{lem01},
\begin{equation*}
\bb E^N_{\mu_{N,K_N}} \Big[ \exp \Big\{\, 2\, N^d \, \bs J_N(F) \,-\,
\bs W_N (T)   \, \Big\}\, \Big] \;=\; 1\;,
\end{equation*}
provided
\begin{equation*}
\bs W_N (T) \;=\; N^2  \sum_{(x,y) \in \bb B_N}
\int_0^T \bs\eta_x(s)\, [1- \bs\eta_y(s)]
\, \{e^{2 \, F_N (s,x,y)} - 1\} \, ds\;.
\end{equation*}
Recalling \eqref{25}, by adding and subtracting
$(1/2 N^d) \bs W_N (T)$ and applying Schwarz inequality we get
\begin{equation*}
\begin{aligned}
& \bb E^N_{\mu_{N,K_N}} \Big[ \, \exp N^d\, \Big\{\, \bs J_N (F) \;-\; a\,
\int_{\bb R} ds \int_{\bb T^d} dx \, \sigma(\bs \pi_N^\epsilon) \,
|F|^2 \Big\}\, \Big] \\
&\quad \;\le\;
\bb E^N_{\mu_{N,K_N}} \Big[ \, \exp \Big\{\, \bs W_N (T) \;-\; 2 \, N^d\,  a\,
\int_{\bb R} ds \int_{\bb T^d} dx \, \sigma(\bs \pi_N^\epsilon) \,
|F|^2 \Big\}\, \Big]^{1/2}\;.
\end{aligned}
\end{equation*}
Expanding the exponential $\exp\{ 2\, F_N (s,x,y) \}$ which appears in
the definition of $\bs W_N (T)$, summing by parts, using Lemma
\ref{lem1}, and the first part of the proof yields the desired bound.
\end{proof}

Consider a continuous function
$w\colon [0,T] \times \bb T^d \times \mc M_{+}(\bb T^d) \times D( \bb
R; \mc H_{-p}^d) \to \bb R^d$ that is continuously differentiable in
$x$ and such that for each
$(x,\pi) \in \bb T^d \times \mc M_{+}(\bb T^d)$ and $t\in [0,T]$ the
map
$[0,t]\times D( \bb R; \mc H_{-p}^d) \ni (s,\bs J) \to w(s,x, \pi, \bs
J)$ is measurable with respect to the Borel $\sigma$-algebra on
$[0,t]\times D([0,t];\mc H^d_{-p}\big)$.  Let
$\phi^{x,y}\colon [0,T]\times D([0,t], \Sigma_N) \to \bb R$,
$(x,y) \in \bb B_N$, be given by
\begin{equation*}
\phi^{x,y}(t) \;=\; \int_x^y w(t, \cdot \, , \bs \pi_N(t),
\bs J_N ) \, \cdot  d\ell\;,
\end{equation*}
and let $\bb M^\phi_T$ be the martingale introduced in \eqref{02},

\begin{lemma}
\label{lem02}
Let $(\bb Q_N : N \ge 1)$ be a sequence of probability measures in
$\ms P^{N,K_N}_{\rm stat}$ such that
$\bb Q_N \circ (\bs \pi_N, \bs J_N)^{-1} \to P$ for some
$P\in \ms P_{\rm stat}$ satisfying \eqref{fe}. Assume that there
exists a finite constant $C_0$ such that for all $S>0$,
\begin{equation*}
\limsup_{N\to\infty} \frac{1}{N^d}\,
\bb H^{(S)} \big(\, \bb Q_N \,|\, \bb P^N_{\mu_{N,K_N}}\,\big)
\;\le\; C_0\, S\;.
\end{equation*}
Then, for each $w$ as above,
\begin{equation*}
\frac{1}{T}\, \lim_{N\to\infty} \frac 1{N^d}\, E_{\bb Q_N} 
\big[\,\log \bb M^\phi_T\, \big] \;=\; E_P \big[ \,
V_{T,w} \, \big] \;,
\end{equation*}
where $V_{T,w}$ has been introduced in \eqref{21}
\end{lemma}

\begin{proof}
On the one hand, by definition of $\phi^{x,y}$ and of the current
$\bs J_N$, and since $\bb Q_N \circ (\bs \pi_N, \bs J_N)^{-1} \to P$
for some measure $P\in \ms P_{\rm stat}$ satisfying \eqref{fe},
\begin{equation*}
\lim_{N\to\infty} \frac 1{N^d}\, E_{\bb Q_N} 
\Big[\, \sum_{(x,y) \in \bb B_N}   \int_0^T
\phi^{x,y}_s \, \ms N^{x,y}_{(s,s+ds]} \, \Big] \;=\; E_P \Big[ \,
\int_0^T ds \int_{\bb T^d} dx \; G_w \cdot \bs j  \, \Big] \;.
\end{equation*}

On the other hand, a straightforward computation yields that
\begin{equation*}
\begin{aligned}
& N^2  \sum_{(x,y) \in \bb B_N} c_{x,y}(\eta)
\{e^{\phi^{x,y}} - 1\} \;=\; N^d\, \< \pi_N\,,\, \nabla\cdot G_w\> \\
& \quad
\;+\; \frac{1}{2} \sum_{j=1}^d \sum_{x\in\bb T^d_N} [\eta_{x+\mf e_j}
- \eta_x]^2 \,w_j(x) \,[ w_j(x) + E_j(x)] \;+\; o(N^d) \;.
\end{aligned}
\end{equation*}
Therefore, by Corollary \ref{cor2} and since
$\bb Q_N \circ (\bs \pi_N, \bs J_N)^{-1} \to P$ for some measure
$P\in \ms P_{\rm stat}$ satisfying \eqref{fe},
\begin{equation*}
\begin{aligned}
& \lim_{N\to\infty} \frac 1{N^d}\, E_{\bb Q_N} 
\Big[\, N^2  \sum_{(x,y) \in \bb B_N} \int_0^T ds\, c_{x,y}(\bs \eta(s))
\{e^{\phi^{x,y}_s} - 1\} \, \Big] \\
& \quad
\;=\;   E_P \Big[ \int_0^T dt \int_{\bb T^d}
dx \, 
\Big\{  \,  G_w \cdot \big[\, -\, \nabla \bs \rho \,  +\,
\sigma(\bs\rho) \, E \, \big]
\,  + \,   \sigma(\bs \rho ) \, |G_w|^2 \,  \Big\} \,\Big] \;.
\end{aligned}
\end{equation*}
The assertion of the lemma follows from the two previous estimates.
\end{proof}

\subsection*{Acknowledgments}
We are grateful to M.\ Mariani and L.\ Rossi for useful discussions.

\end{document}